\documentclass{amsart}

\usepackage{amsmath}
\usepackage{amsthm}
\usepackage{amsfonts}
\usepackage{amssymb}
\usepackage{graphics}

\newcommand\R{{\mathbb{R}}}
\newcommand\C{{\mathbb{C}}}

\newcommand\I{{\mathbf{I}}}
\renewcommand\P{{\mathbf{P}}}
\newcommand\E{{\mathbf{E}}}

\newcommand\Var{\mathbf{Var}}
\renewcommand\Im{{\operatorname{Im}}}
\renewcommand\Re{{\operatorname{Re}}}
\newcommand\eps{{\varepsilon}}

\newcommand\tr{\operatorname{trace}}

\newcommand\sgn{\operatorname{sgn}}

\renewcommand\th{{\operatorname{th}}}
\newcommand\erf{{\operatorname{erf}}}

\newcommand\condone{{{\bf C1}}}

\newcommand\CE{{\mathcal E}}

\subjclass{15A52}

\theoremstyle{plain}
  \newtheorem{theorem}{Theorem}

  \newtheorem{proposition}[theorem]{Proposition}
  
  \newtheorem{lemma}[theorem]{Lemma}
  \newtheorem{corollary}[theorem]{Corollary}

\theoremstyle{definition}
  \newtheorem{definition}[theorem]{Definition}
  
  \newtheorem{remark}[theorem]{Remark}

\include{psfig}

\begin{document}

\title[Universality for non-Hermitian matrices]{Random matrices: Universality of local spectral statistics of non-Hermitian matrices}

\author{Terence Tao}
\address{Department of Mathematics, UCLA, Los Angeles CA 90095-1555}
\email{tao@math.ucla.edu}
\thanks{T. Tao is supported by NSF grant DMS-0649473.}

\author{Van Vu}
\address{Department of Mathematics, Rutgers, Piscataway, NJ 08854}
\email{vanvu@math.rutgers.edu}
\thanks{V. Vu is supported by research grants DMS-0901216 and AFOSAR-FA-9550-09-1-0167.}

\begin{abstract} It is a classical result of Ginibre that the normalized bulk $k$-point correlation functions of a complex $n \times n$ gaussian matrix with independent
 entries of mean zero and unit variance are asymptotically given by the determinantal point process on $\C$ with kernel $K_\infty(z,w) := \frac{1}{\pi} e^{-|z|^2/2 - |w|^2/2 + z \overline{w}}$ in the limit $n \to \infty$.  In this paper we show that this asymptotic law is  universal among all random $n \times n$ matrices $M_n$ whose entries are jointly independent, exponentially decaying, have independent real and imaginary parts, and whose moments match that of the complex gaussian ensemble to fourth order. Analogous results at the edge of the spectrum are also obtained.  As an application, we extend a central limit theorem for the number of eigenvalues of complex gaussian matrices in a small disk to these more general ensembles.

These results are non-Hermitian analogues of some recent universality results for Hermitian Wigner matrices.  However, a key new difficulty arises in the non-Hermitian case, due to the instability of the spectrum for such matrices. To resolve this issue, we the need to work with the log-determinants $\log|\det(M_n-z_0)|$ rather than with the Stieltjes transform $\frac{1}{n} \tr (M_n-z_0)^{-1}$, in order to exploit Girko's Hermitization method.  Our main tools are a four moment theorem for these log-determinants, together with a strong concentration result for the log-determinants in the gaussian case.  The latter is established by studying the solutions of a certain nonlinear stochastic difference equation.

With some extra consideration, we can extend our arguments to the real case, proving
 universality for correlation functions  of real matrices which match the real gaussian ensemble to the fourth order.  As an application, we show that a real $n \times n$ matrix whose entries are jointly independent, exponentially decaying, and whose moments match the real gaussian ensemble to fourth order has $\sqrt{\frac{2n}{\pi}} + o(\sqrt{n})$ real eigenvalues asymptotically almost surely.
\end{abstract}

\maketitle

\setcounter{tocdepth}{1}
\tableofcontents

\section{Introduction}

Let $M_n$ be a random $n \times n$ matrix with complex entries, which is not necessarily assumed to be Hermitian, and can be either a continuous or discrete ensemble of matrices. Then, counting multiplicities, there are $n$ complex (algebraic) eigenvalues, which we enumerate in an arbitrary fashion as
$$ \lambda_1(M_n), \dots, \lambda_n(M_n) \in \C.$$
One can then define, for each $1 \leq k \leq n$, the \emph{$k$-point correlation function}
$$\rho^{(k)}_n = \rho^{(k)}_n[M_n]: \C^k \to \R^+ $$  of the random matrix ensemble $M_n$ by requiring that
\begin{equation}\label{ck}
\begin{split}
& \int_{\C^k} F(z_1,\dots,z_k) \rho^{(k)}_n(z_1,\dots,z_k)\ dz_1 \dots dz_k  \\
&\quad = \E \sum_{1 \leq i_1, \dots, i_k \leq n, \hbox{ distinct}} F( \lambda_{i_1}(M_n), \dots, \lambda_{i_k}(M_n))
\end{split}
\end{equation}
for all continuous, compactly supported test functions $F$, where $dz$ denotes Lebesgue measure on the complex plane $\C$.  Note that this definition does not depend on the exact order in which the eigenvalues of $M_n$ are enumerated.

If $M_n$ is an absolutely continuous matrix ensemble with a continuous density function, then $\rho^{(k)}$ is a continuous function; but if $M_n$ is a discrete ensemble then $\rho^{(k)}$ is merely a non-negative measure\footnote{Here, we have abused notation by identifying a measure $\rho^{(k)}_n(z_1,\dots,z_k)\ dz_1 \dots dz_k$ with its density $\rho^{(k)}_n$.}.  In the absolutely continuous case with a continuous density function, one can equivalently define $\rho^{(k)}_n(z_1,\dots,z_k)$ for distinct $z_1,\dots,z_k$ to be the quantity such that the probability that there is an eigenvalue of $M_n$ in each of the disks $\{ z: |z-z_i| \leq \eps\}$ for $i=1,\dots,k$ is asymptotically $(\rho^{(k)}_n(z_1,\dots,z_k)+o(1)) (\pi \eps^2)^k$ in the limit $\eps \to 0^+$.

We note two model cases of continuous matrix ensembles that are of interest.  The first is the \emph{real gaussian matrix ensemble}\footnote{Strictly speaking, the real gaussian matrix ensemble is only absolutely continuous with respect to Lebesgue measure on the space of \emph{real} $n \times n$ matrices, rather than on the space of \emph{complex} $n \times n$ matrices.  However, both ensembles are still continuous in the sense that any individual matrix occurs in the ensemble with probability zero.}, in which coefficients $\xi_{ij}$ are independent and identically distributed (or \emph{iid} for short) and have the distribution $N(0,1)_\R$ of the real gaussian with mean zero and variance one.  We will discuss this case in more detail later, but for now we will focus instead on the simpler and better understood case of the \emph{complex gaussian matrix ensemble}, in which the $\xi_{ij}$ are iid with the distribution of a complex gaussian $N(0,1)_\C$ with mean zero and variance one (or in other words, the probability distribution of each $\xi_{ij}$ is  $\frac{1}{\pi} e^{-|z|^2}\ dz$, and the real and imaginary parts of $\xi_{ij}$ independently have the distribution $N(0,1/2)_\R$).  As is well known, the correlation functions of a complex gaussian matrix are given by the explicit \emph{Ginibre formula} \cite{gin}
\begin{equation}\label{Det}
 \rho^{(k)}_n(z_1,\dots,z_k) = \det( K_n(z_i,z_j) )_{1 \leq i,j \leq k}
 \end{equation}
where $K_n:\C \times \C \to \C$ is the kernel
\begin{equation}\label{knzw}
 K_n(z,w) := \frac{1}{\pi} e^{-(|z|^2+|w|^2)/2} \sum_{j=0}^{n-1} \frac{(z\overline{w})^j}{j!}.
\end{equation}
In particular, one has
\begin{equation}\label{kl2}
\rho^{(1)}_n(z) = K_n(z,z) = \frac{1}{\pi} e^{-|z|^2} \sum_{j=0}^{n-1} \frac{|z|^{2j}}{j!}
\end{equation}
and thus (by Taylor expansion of $e^{-|z|^2}$) one has the asymptotic
$$ \rho^{(1)}_n(\sqrt{n} z) \to \frac{1}{\pi} 1_{|z| \leq 1}$$
as $n \to \infty$ for almost every $z \in \C$.  This gives the well-known \emph{circular law} for complex gaussian matrices, namely that the empirical spectral distribution of $\frac{1}{\sqrt{n}} M_n$ converges (in expectation, at least) to the circular measure $\frac{1}{\pi} 1_{B(0,1)}\ dz$, where we use $B(z_0,r) := \{ z \in \C: |z-z_0| < r \}$ to denote an open disk in the complex plane.   Informally, this means that the eigenvalues of $M_n$ are asymptotically uniformly distributed on the disk $B(0,\sqrt{n})$.  The circular law is also known to hold for many other ensembles of matrices, and for several modes of convergence. In particular, it holds (both in probability and in the almost sure sense) for random matrices with iid entries having mean $0$ and variance $1$;
see the surveys \cite{TV-survey, chafai} for further discussion of this and related results.  Figures \ref{figure:cl-1}, \ref{figure:cl-2} later in this paper illustrate the circular law for two model instances of iid ensembles, namely the real gaussian and real Bernoulli ensembles.

We also remark that from the obvious inequality
$$ \sum_{j=0}^{n-1} \frac{|z|^{2j}}{j!} \leq \sum_{j=0}^\infty \frac{|z|^{2j}}{j!} = e^{|z|^2}$$
and \eqref{kl2} we have the uniform bound
$$ |K_n(z,z)| \leq \frac{1}{\pi}$$
for all $z$, and hence by positivity of $\rho^{(2)}_n(z,w) = K_n(z,z) K_n(w,w) - |K_n(z,w)|^2$ we also have
\begin{equation}\label{knzwp}
 |K_n(z,w)| \leq \frac{1}{\pi}
\end{equation}
for all $z,w$.  In particular, from \eqref{Det} one has
\begin{equation}\label{rhok-2}
0 \leq \rho^{(k)}_{n,z_1,\dots,z_k}(w_1,\ldots,w_k) \leq C_k
\end{equation}
in the case of the complex gaussian ensemble for all $w_1,\ldots,w_k \in \C$, all $n$, and some constant $C_k$ depending only on $k$.  (Indeed, from the Hadamard inequality one can take $C_k = \pi^{-k} k^{k/2}$, for instance.)  This uniform bound will be technically convenient for some of our applications.  We will also need an analogous bound for the real gaussian ensemble; see Lemma \ref{corf} below.

Our first main result is to show a universality result of the $k$-point correlation functions $\rho^{(k)}_{n,z_1,\dots,z_k}(w_1,\ldots,w_k)$, in the spirit of the ``Four Moment Theorems'' for Wigner matrices that first appeared in \cite{TVlocal1}.  Very roughly speaking, the result is that (when measured in the vague topology), the asymptotic behaviour of these correlation functions for matrices with independent entries depend only on the first four moments of the entries, though due to our reliance on the Lindeberg exchange method, we will also need to require these matrices to match moments with the complex gaussian ensemble.  To make this statement more precise, we will need some further notation.

\begin{definition}[Independent-entry matrices]  An \emph{independent-entry matrix ensemble} is an ensemble of random $n \times n$ matrices $M_n = (\xi_{ij})_{1 \leq i,j \leq n}$, where the $\xi_{ij}$ are independent and complex random variables, each with mean zero and variance one; we call the $\xi_{ij}$ the \emph{atom distributions} of $M_n$.  We say that the independent-entry matrix has \emph{independent real and imaginary parts} if for each $1 \leq i,j \leq n$, $\Re(\xi_{ij}), \Im(\xi_{ij})$ are independent.  We say that the matrix obeys Condition {\condone} if one has
$$ \P( |\xi_{ij}| \geq t ) \leq C \exp( - t^c )$$
for some fixed $C,c>0$ (independent of $n$) and all $i,j$.

If $k \geq 0$, we say that two independent-entry matrix ensembles $M_n = (\xi_{ij})_{1 \leq i,j \leq n}$ and $M'_n = (\xi'_{ij})_{1 \leq i,j \leq n}$ have \emph{matching moments to order $k$} if one has
\begin{equation}\label{emm}
 \E \Re(\xi_{ij})^a \Im(\xi_{ij})^b = \E \Re(\xi'_{ij})^a \Im(\xi'_{ij})^b
\end{equation}
whenever $1 \leq i,j \leq n$, $a,b \geq 0$ and $a+b \leq k$.
\end{definition}

Our first main result is then as follows.

\begin{theorem}[Four Moment Theorem for complex matrices]\label{main-alt}  Let $M_n, \tilde M_n$ be independent-entry matrix ensembles with independent real and imaginary parts, obeying Condition \condone, such that $M_n$ and $\tilde M_n$ both match moments with the complex gaussian matrix ensemble to third order, and match moments with each other to fourth order.  Let $k \geq 1$ be a fixed integer, let $z_1,\ldots,z_k \in \C$ be bounded (thus $|z_i| \leq C$ for all $i=1,\ldots,k$ and some fixed $C>0$), and let $F: \C^k \to \C$ be a smooth function, which admits a decomposition of the form
\begin{equation}\label{factor}
 F(w_1,\ldots,w_k) = \sum_{i=1}^m F_{i,1}(w_1) \ldots F_{i,k}(w_k)
\end{equation}
for some fixed $m$ and some smooth functions $F_{i,j}: \C \to \C$ for $i=1,\ldots,m$ and $j=1,\ldots,k$ supported on the disk $\{ w: |w| \leq C \}$ obeying the derivative bounds\footnote{See Section \ref{notation-sec} for the definition of the $a$-fold gradient $\nabla^a F_{i,j}$.}
\begin{equation}\label{fawc}
 |\nabla^a F_{i,j}(w)| \leq C
 \end{equation}
for all $0 \leq a \leq 5$, $i=1,\dots,m$, $j=1,\dots,k$ and $w \in \C$, and some fixed $C$.  Let $\rho^{(k)}_n, \tilde \rho^{(k)}_n$ be the correlation functions for $M_n, \tilde M_n$ respectively.  Then
\begin{align*}
& \int_{\C^k} F(w_1,\dots,w_k) \rho^{(k)}_n(\sqrt{n} z_1 + w_1,\dots,\sqrt{n} z_k + w_k)\ dw_1 \dots dw_k \\
&\quad =
\int_{\C^k} F(w_1,\dots,w_k) \tilde \rho^{(k)}_n(\sqrt{n} z_1 + w_1,\dots,\sqrt{n} z_k + w_k)\ dw_1 \dots dw_k + O(n^{-c}).
\end{align*}
for some absolute constant $c>0$ (independent of $k$).  Furthermore, the implicit constant in the $O(n^{-c})$ notation is uniform over all $z_1,\ldots,z_k$ in the bounded region $\{ z: |z| \leq C \}$.
\end{theorem}

\begin{remark}
The regularity hypotheses on the test function $F$ here are somewhat technical, but they are needed to obtain the uniform polynomial decay $O(n^{-c})$ in the conclusion, which is useful for several applications.  Note that by rescaling one could allow the bound $C$ in \eqref{fawc} to be enlarged somewhat, to $C n^{c/2k}$, without impacting the conclusion (other than to degrade the $O(n^{-c})$ error slightly to $O(n^{-c/2})$).
If one is only seeking a qualitative error term of $o(1)$, then by applying the Stone-Weierstrass theorem, one only needs $F$ to be continuous and compactly supported, instead of having a smooth factorization of the form \eqref{factor}; see the proof of Corollary \ref{main} below.  Also, if $F$ is smooth and compactly supported, then by using a partial Fourier expansion one can again obtain a polynomial decay rate $O(n^{-c})$ (with the implied constant depending on the bounds on finitely many derivatives of $F$).  It is possible to improve the value of $c$ somewhat by adding additional matching moment hypotheses, but then one also requires the derivative bounds \eqref{fawc} for a larger range of exponents $a$; we will not quantify this variant of Theorem \ref{main-alt} here.  The requirement that $M_n, M'_n$ match the complex gaussian ensemble to third order can be removed if $z_1,\ldots,z_k$ stays a bounded distance away from the origin, using an extremely recent result of Bourgade, Yau, and Yin \cite{byy}; see Remark \ref{breaking}.
\end{remark}

Theorem \ref{main-alt} is motivated by the phenomenon, first observed in
\cite{TVlocal1},  that the asymptotic local statistics of the spectrum of
a random Hermitian matrix of Wigner type typically depend only on the first four moments of the entries; formalizations of this phenomenon are known as \emph{four moment theorems}.  In particular, Corollary \ref{main} is analogous\footnote{Thanks to more recent results by many authors \cite{ERSTVY}, \cite{ESY}, \cite{TVlocal2}, \cite{ESYY}, \cite{EYY}, \cite{TVmeh}, these results are no longer the sharpest results available in the Wigner setting, as the moment matching conditions have now largely been removed, the exponential decay condition relaxed to a finite moment condition, and the bulk results extended to the edge; see
the discussion in \cite{TVmeh} or the  surveys \cite{Erd}, \cite{Alice}, \cite{schlein}, \cite{TV-survey2} for surveys for more details.  In view of these results, it is reasonable to conjecture the moment matching assumptions in Theorem \ref{main-alt} or Corollary \ref{main} may be relaxed; see Remark \ref{breaking} for some very recent developments in this direction.} to the four moment theorems in \cite[Theorems 11, 38]{TVlocal1}.

\begin{remark} The hypothesis of independent real and imaginary parts is primarily for reasons of notational convenience, and it is likely that this hypothesis could be dropped from our results.  Note that when $M_n$ and $M'_n$ have independent real and imaginary parts, the moment matching condition \eqref{emm} simplifies to
$$
 \E \Re(\xi_{ij})^a = \E \Re(\xi'_{ij})^a
$$
and
$$
 \E \Im(\xi_{ij})^b = \E \Im(\xi'_{ij})^b
$$
for $1 \leq i,j \leq n$ and $0 \leq a,b \leq k$.

It is also likely that the exponential decay condition in Condition {\condone} could be replaced with a bound on a sufficiently high moment of the entries.  We will however not pursue these refinements here.  The vague convergence in the conclusion is natural given that the ensemble $M_n$ is permitted to be discrete (so that $\rho^{(k)}_n$ could be a discrete measure, rather than a continuous function).  In analogy with the Hermitian theory (see e.g. \cite{TVmeh}), it is reasonable to conjecture that stronger modes of convergence become available if some additional regularity hypotheses are placed on the entries, but we will not pursue such matters here.
\end{remark}

We now discuss some applications of Theorem \ref{main-alt}.  The first application concerns the asymptotic behaviour of the $k$-point correlation functions as $n \to \infty$.  In the case when $M_n$ is drawn from the complex gaussian ensemble, these asymptotics have been well understood since the work of Ginibre \cite{gin}.  To recall these asymptotics we introduce the following functions.

\begin{definition}[Asymptotic kernel]  For complex numbers $z_1,z_2,w_1,w_2$, define the kernel $K_{\infty,z_1,z_2}(w_1,w_2)$ by the following rules:
\begin{itemize}
\item[(i)]  If $z_1 \neq z_2$, then $K_{\infty,z_1,z_2}(w_1,w_2) := 0$.
\item[(ii)]  If $z_1 = z_2$ and $|z_1| > 1$, then $K_{\infty,z_1,z_2}(w_1,w_2) := 0$.
\item[(iii)]  If $z_1 = z_2$ and $|z_1| < 1$, then $K_{\infty,z_1,z_2}(w_1,w_2) := \frac{1}{\pi} e^{-|w_1|^2/2-|w_2|^2/2+w_1\overline{w_2}}$.
\item[(iv)] If $z_1=z_2$ and $|z_1|=1$, then $K_{\infty,z_1,z_2}(w_1,w_2) :=
\frac{1}{\pi} e^{-|w_1|^2/2-|w_2|^2/2+w_1\overline{w_2}} (\frac{1}{2} + \frac{1}{2} \erf( - \sqrt{2} (z_1 \overline{w_2} + w_1 \overline{z_2}) ))$.
\end{itemize}
Here
$$ \erf(z) := \frac{2}{\sqrt{\pi}} \int_0^z e^{-t^2} dt$$
is the usual error function, defined for all complex $z$, where the integral is over an arbitrary contour from $0$ to $z$.  For complex numbers $z_1,\dots,z_k,w_1,\dots,w_k$, define the correlation function
$$ \rho_{\infty,z_1,\dots,z_k}^{(k)}(w_1,\dots,w_k) := \det( K_{\infty,z_i,z_j}(w_i,w_j) )_{1 \leq i,j \leq k}.$$
\end{definition}

In the model case when $z_1,\dots,z_k$ all avoid the unit circle $\{ z \in \C: |z|=1\}$, the kernel simplifies to
$$ K_{\infty,z_i,z_j}(w_i,w_j) = 1_{z_i=z_j} 1_{|z_i| < 1} K_\infty(w_i,w_j)$$
where
$$ K_\infty(z,w) := \frac{1}{\pi} e^{-|z|^2/2-|w|^2/2 + z \overline{w}}.$$
The kernel $K_\infty$ can also be interpreted as the reproducing kernel for the orthogonal projection in $L^2(\C)$ to (the closure of) the space of functions $f(z)$ that become holomorphic after multiplication by $e^{|z|^2/2}$, or equivalently to the closed span of $z^k e^{-|z|^2/2}$ for $k=0,1,\dots$.

\begin{lemma}[Kernel asymptotics]\label{kernel}  Let $z_1,\dots,z_k,w_1,\dots,w_k$ be fixed complex numbers for some fixed $k$, and let $M_n$ be drawn from the complex gaussian ensemble.  Then we have\footnote{See Section \ref{notation-sec} for the asymptotic notational conventions we will use in this paper.}
\begin{equation}\label{potential}
 \rho^{(k)}_n(\sqrt{n} z_1 + w_1,\dots,\sqrt{n} z_k + w_k ) = \rho^{(k)}_{\infty,z_1,\dots,z_k}(w_1,\dots,w_k) + o(1).
\end{equation}
If none of the $z_1,\dots,z_k$ lie on the unit circle, then we may improve the error term $o(1)$ to $O(\exp(-\delta n))$ for some fixed $\delta>0$.

Now suppose that $z_1,\dots,z_k,w_1,\dots,w_k$ are allowed to vary in $n$, but that the $z_1,\dots,w_1,\dots,w_k$ remain bounded (i.e. $|z_i|,|w_i| \leq C$ for some fixed $C$ and all $1 \leq i \leq k$) and the $z_1,\ldots,z_k$ stay bounded away from the unit circle (i.e. $||z_i|-1| \geq \eps$ for some fixed $\eps>0$ and all $1 \leq i \leq k$).  Then one still has the asymptotic \eqref{potential}.  In other words, the decay rate of the error term $o(1)$ in \eqref{potential} is uniform across all choices of $z_1,\ldots,z_k,w_1,\ldots,w_k$ in the ranges specified above.
\end{lemma}

\begin{proof} This is a well-known asymptotic (see e.g. \cite{Meh}, \cite{nourdin}, or \cite{borodin}).  For sake of completeness, we have written a proof of these standard facts at Appendix B of the copy of this paper at {\tt arXiv:1206.1893v3}.
\end{proof}

From this lemma we conclude in particular that $\rho^{(k)}_{\infty,z_1,\dots,z_k}(w_1,\ldots,w_k) \geq 0$ for all $k,z_1,\dots,z_k,w_1,\dots,w_k$, which (when combined with \eqref{knzwp}) yields the uniform bound
$$ |K_{\infty,z_1,z_2}(w_1,w_2)| \leq \frac{1}{\pi}$$
for all $z_1,z_2,w_1,w_2 \in \C$.  In particular, we have
\begin{equation}\label{rhok}
0 \leq \rho^{(k)}_{\infty,z_1,\dots,z_k}(w_1,\ldots,w_k) \leq C_k
\end{equation}
for all $w_1,\ldots,w_k \in \C$ and some constant $C_k$ depending only on $k$.

Using Theorem \ref{main-alt}, we may extend the above asymptotics for complex gaussian matrices to more general ensembles (including some discrete ensembles), as follows.

\begin{corollary}[Universality for complex matrices]\label{main}  Let $M_n$ be an independent-entry matrix ensemble with independent real and imaginary parts, obeying Condition \condone, and which matches moments with the complex gaussian matrix ensemble to fourth order.  Then for any fixed (i.e. independent of $n$), fixed $k \geq 1$ and \emph{fixed} $z_1,\dots,z_k \in \C$, and any fixed continuous, compactly supported function $F: \C^k \to \C$, one has
\begin{align*}
& \int_{\C^k} F(w_1,\dots,w_k) \rho^{(k)}_n(\sqrt{n} z_1 + w_1,\dots,\sqrt{n} z_k + w_k)\ dw_1 \dots dw_k \\
&\quad =
\int_{\C^k} F(w_1,\dots,w_k) \rho^{(k)}_{\infty,z_1,\dots,z_k}(w_1,\dots,w_k)\ dw_1 \dots dw_k + o(1).
\end{align*}
In  other words, the asymptotic \eqref{potential} is valid in the vague topology for this ensemble.  If $F$ is furthermore assumed to be smooth, then we may improve the $o(1)$ error term here to $O(n^{-c})$ for some fixed $c>0$.
\end{corollary}

\begin{proof}  From Theorem \ref{main-alt} and Lemma \ref{kernel}, we obtain Corollary \ref{main} in the case when $F$ admits a decomposition of the form given in Theorem \ref{main-alt} (and in this case the $o(1)$ error can be improved to $O(n^{-c})$).  The more general case of continuous, compactly supported $F$ can then be deduced by using the Stone-Weierstrass theorem to approximate a continuous $F$ by an approximant $\tilde F$ of the form \eqref{factor} (and by using a further function of the form in Theorem \ref{main-alt} and \eqref{rhok} to upper bound the error).  When $F$ is smooth, one can replace the use of the Stone-Weierstrass theorem by a more quantitative partial Fourier series expansion of $F$ (extended periodically in a suitable fashion), followed by a multiplication by a smooth cutoff function, taking advantage of the rapid decrease of the Fourier coefficients in the smooth case; we omit the standard details.
\end{proof}

\begin{remark}  Note that in contrast to the situation in Theorem \ref{main-alt}, the parameters $z_1,\ldots,z_k$ in Corollary \ref{main} are required to be fixed in $n$, as opposed to being allowed to vary in $n$.  Related to this, the error term $o(1)$ in Corollary \ref{main} is not asserted to be uniform in the choice of $z_1,\ldots,z_k$, in contrast to the uniformity in Theorem \ref{main-alt}.  Indeed, given that the limiting correlation function $\rho^{(k)}_{\infty,z_1,\ldots,z_k}$ behaves discontinuously in $z_1,\ldots,z_k$ whenever two of the $z_i$ collide, or when one of the $z_i$ crosses the unit circle, one would not expect such uniformity in Corollary \ref{main}.  Thus, while Corollary \ref{main} describes more explicitly the limiting behavior (in certain regimes) of the correlation functions $\rho^{(k)}$, we regard Theorem \ref{main-alt} as the more precise statement regarding the asymptotics of these functions.
\end{remark}

In the Hermitian case, Four Moment Theorems can be used to extend various facts about the asymptotic spectral distribution of special matrix ensembles (such as the gaussian unitary ensemble) to other matrix ensembles which obey appropriate moment matching conditions.  Similarly, by using Theorem \ref{main-alt}, one may extend some facts about eigenvalues of complex gaussian matrices can now be extended to iid matrix models that match the complex gaussian ensemble to fourth order, although in some ``global'' cases the extension is only partial in nature due to the ``local'' nature of the four moment theorem.  Rather than provide an exhaustive list of such applications, we will present just one representative such application, namely that of (partially) extending the following central limit theorem of Rider \cite{rider}:

\begin{theorem}[Central limit theorem, gaussian case]\label{clt-gauss}  Let $M_n$ be drawn from the complex gaussian ensemble.  Let $r>0$ be a real number (depending on $n$) such that $1/r, r/n^{1/2} = o(1)$.  Let $z_0$ be a complex number (also depending on $n$) such that $|z_0| \leq (1-\eps) \sqrt{n}$ for some fixed $\eps>0$. Let $N_{B(z_0,r)}$ be the number of eigenvalues of $M_n$ in the ball $B(z_0,r) := \{ z \in \C: |z-z_0| < r \}$.  Then we have
$$
\frac{N_{B(z_0,r)} - r^2}{r^{1/2} \pi^{-1/4}} \to N(0,1)_\R$$
in the sense of distributions.  In fact, we have the slightly stronger statement that
\begin{equation}\label{Carleman}
 \E \left(\frac{N_{B(z_0,r)} - r^2}{r^{1/2} \pi^{-1/4}}\right)^k \to \E N(0,1)_\R^k
\end{equation}
for all fixed natural numbers $k \geq 0$.
\end{theorem}

\begin{proof}  From the general Costin-Lebowitz central limit theorem for determinantal point processes \cite{costin}, \cite{sos}, \cite{sos2} we know that
$$
\frac{N_{B(z_0,r)} - \E N_{B(z_0,r)}}{(\Var N_{B(z_0,r)})^{1/2}} \to N(0,1)_\R$$
provided that $\Var N_{B(z_0,r)} \to \infty$; indeed, an inspection of the proof in \cite{sos2} gives the slightly stronger assertion that
$$
\E \left(\frac{N_{B(z_0,r)} - \E N_{B(z_0,r)}}{(\Var N_{B(z_0,r)})^{1/2}}\right)^k \to \E N(0,1)_\R^k$$
for any fixed $k \geq 0$.  Thus it will suffice to establish the asymptotics
$$ \E N_{B(z_0,r)} = (1+o(1)) r^2$$
and
$$ \Var N_{B(z_0,r)} = (1+o(1)) \pi^{-1/2} r.$$
Using \eqref{ck}, \eqref{Det}, one can write the left-hand sides here as
$$ \int_{B(z_0,r)} K_n(z,z)\ dz$$
and
$$ \int_{B(z_0,r)} K_n(z,z)\ dz - \int_{B(z_0,r)} \int_{B(z_0,r)} |K_n(z,w)|^2\ dz dw$$
respectively.  By Lemma \ref{kernel}, the former expression converges to $\int_{B(z_0,r)} \frac{1}{\pi} \ dz = r^2$.  Lemma \ref{kernel} also reveals that the second expression is asymptotically independent of $z_0$, and so one may without loss of generality take $z_0=0$.  But then the required asymptotic follows from \cite[Theorem 1.6]{rider} (after allowing for the different normalisation for $M_n$ in that paper).
\end{proof}

Using Theorem \ref{main-alt}, we may extend this result to more general ensembles, at least in the small radius case:

\begin{corollary}[Central limit theorem, general case]\label{clt-2}  Let $M_n$ be an independent-entry matrix ensemble with independent real and imaginary parts, obeying Condition \condone, such that $M_n$ matches moments with the complex gaussian matrix ensemble to fourth order.  Then the conclusion of
Theorem \ref{clt-gauss} for $M_n$ holds provided that one has the additional assumption $r \leq n^{o(1)}$.
\end{corollary}

We prove this result in Section \ref{real-sec}.  The restriction to small radii $r \leq n^{o(1)}$ appears to be a largely technical restriction, relating to the need to take arbitrarily high moments in order to establish a central limit theorem; see for instance Figure \ref{figure:circle} for some numerical evidence that the central limit theorem should in fact hold for larger radii as well (and for real matrices as well as complex ones).  It seems likely that one can also obtain extensions of many of the other results in \cite{rider} (or related papers, such as \cite{kostlan}, \cite{rider-2}) on gaussian fluctuations from the circular law from the complex gaussian ensemble to other ensembles that match the complex gaussian ensemble to a sufficiently large number of moments, but we will not pursue such results here.  We remark that for macroscopic statistics $\frac{1}{n} \sum_{i=1}^n F(\lambda_i/\sqrt{n})$ with $F$ fixed and analytic, such extensions (without the need for matching moments beyond the second moment) were already established in \cite{rider-silverstein}.

\begin{figure}
\begin{center}
\scalebox{.3}{\includegraphics{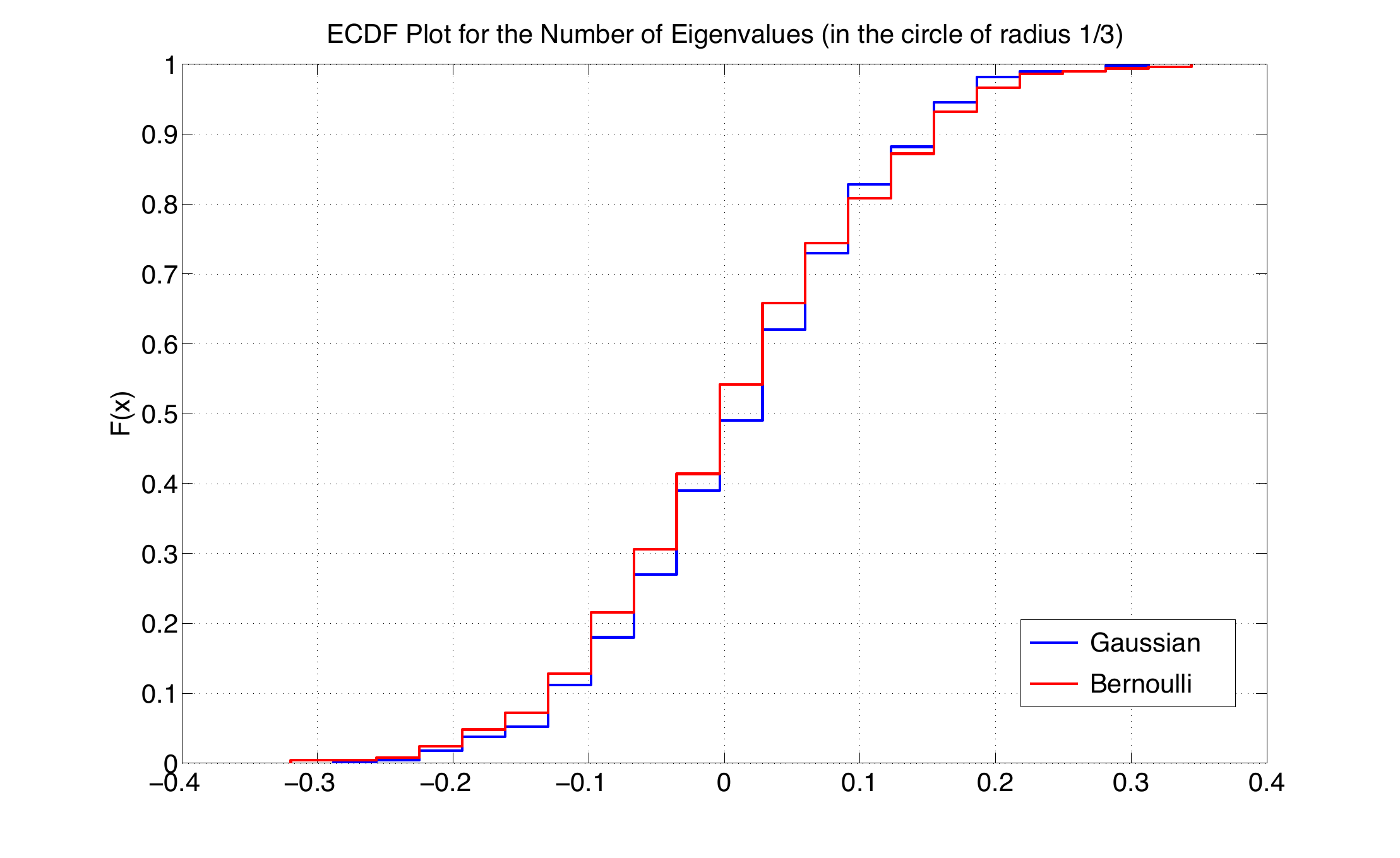}}
\end{center}
\caption{The cumulative distribution function for the number of eigenvalues in the disk $B(0,\sqrt{n}/3)$ of real gaussian and real Bernoulli matrices of size $10,000 \times 10,000$, after normalizing the mean by $n/9$ and variance by $\sqrt{n}$.  Thanks to Ke Wang for the data and figure.}
\label{figure:circle}
\end{figure}

\subsection{The real case and applications}

There is a (more complicated) analogue of Theorem \ref{main-alt} in which the complex entries are  replaced by real ones.  This has the effect of forcing the spectrum $\lambda_1(M_n),\dots,\lambda_n(M_n)$ to split into some number $\lambda_{1,\R}(M_n), \dots, \lambda_{N_\R[M_n],\R}(M_n)$ of real eigenvalues, together with some number $\lambda_{1,\C_+}(M_n),\dots,\lambda_{N_{\C_+}[M_n],\C_+}(M_n)$ of complex eigenvalues in the upper half-plane $\C_+ := \{ z \in \C: \Im(z) > 0 \}$, as well as their complex conjugates $\overline{\lambda_{1,\C_+}(M_n)},\dots,\overline{\lambda_{b,\C_+}(M_n)}$, where $N_\R[M_n], N_{\C_+}[M_n]$ denote the number of real eigenvalues of $M_n$ and the number of eigenvalues of $M_n$ in the upper half-plane respectively (so in particular, $N_\R[M_n] + 2 N_{\C_+}[M_n] = n$ almost surely).  Because of this additional structure of the eigenvalues, it is no longer convenient to consider the correlation functions $\rho^{(k)}_n: \C^k \to \R^+$ as defined in \eqref{ck}, since they become singular when one or more of the variables is real.  Instead, it is more convenient to work with the correlation functions $\rho^{(k,l)}_n: \R^k \times \C_+^l \to \R^+$, defined for $k,l \geq 0$ by the formula
\begin{equation}\label{ck-r}
\begin{split}
& \int_{\R^k}\int_{\C_+^l} F(x_1,\dots,x_k,z_1,\dots,z_l) \rho^{(k,l)}_n(x_1,\dots,x_k,z_1,\dots,z_l)\ dx_1 \dots dx_k dz_1 \dots dz_l  \\
&\quad=  \E \sum_{1 \leq i_1 < \dots < i_k \leq N_\R[M_n]}
\sum_{1 \leq j_1 < \dots < j_l \leq N_{\C_+}[M_n]} \\
&\quad\quad\quad\quad
F( \lambda_{i_1,\R}(M_n), \dots, \lambda_{i_k,\R}(M_n), \lambda_{j_1,\C_+}(M_n), \dots, \lambda_{j_l,\C_+}(M_n)).
\end{split}
\end{equation}
Again, the exact ordering of the eigenvalues here is unimportant.  When the law of $M_n$ has a continuous density with respect to Lebesgue measure on real matrices (which is for instance the case with the real gaussian ensemble), one can interpret $\rho^{(k,l)}_n(x_1,\dots,x_k,z_1,\dots,z_l)$ for distinct $x_1,\ldots,x_k \in \R$ and $z_1,\ldots,z_l \in \C_+$ as the unique real number such that, as $\eps \to 0$, the probability of simultaneously having an eigenvalue of $M_n$ in each of the intervals $(x_i-\eps,x_i+\eps)$ for $i=1,\ldots,k$ and in each of the disks $B(z_j,\eps)$ for $j=1,\ldots,l$ is equal to
$$(1+o(1)) \rho^{(k,l)}_n(x_1,\dots,x_k,z_1,\dots,z_l) (2\eps)^k (\pi \eps^2)^l$$
in the limit as $\eps \to 0$.

Define $\C_- := \{z \in \C: \Im(z) < 0 \}$ and $\C_* := \C_+ \cup \C_- = \C \backslash \R$.
We extend the correlation functions $\rho^{(k,l)}_n$ from $\R^k \times \C_+^l$ to $\R^k \times \C_*^l$ by requiring that the functions be invariant with respect to conjugations of any of the $l$ coefficients of $\C^l$.  We then extend $\rho^{(k,l)}_n$ by zero from $\R^k \times \C^l_*$ to $\R^k \times \C^l$.

When $M_n$ is given by the real gaussian ensemble, the correlation functions $\rho^{(k,l)}_n$ were computed by a variety of methods, for both odd and even $n$, in \cite{sommers}, \cite{sinclair}, \cite{borodin}, \cite{borodin-0}, \cite{ake}, \cite{kanz}, \cite{mays} (with the $(k,l)=(1,0),(0,1)$ cases worked out previously in \cite{leh}, \cite{eks}, \cite{edel}, building in turn on the foundational work of Ginibre \cite{gin}).  The precise formulae for these correlation functions are somewhat complicated and involve Pfaffians of a certain $2 \times 2$ matrix kernel; see Appendix \ref{real-app} for the formulae when $n$ is even, and \cite{sinclair}, \cite{mays} for the case when $n$ is odd.  To avoid some technical issues we shall restrict attention to the case when $n$ is even, although it is virtually certain that the results here should also extend to the odd $n$ case.

For technical reasons, we will need the following variant of \eqref{rhok-2}:

\begin{lemma}[Uniform bound on correlation functions]\label{corf}  Let $k,l \geq 0$ be fixed natural numbers, let $n$ be even, and let $M_n$ be drawn from the real gaussian ensemble.  Then for all $x_1,\dots,x_k \in \R$ and $z_1,\dots,z_l \leq \C$ one has
$$0 \leq \rho^{(k,l)}_n(x_1,\dots,x_k,z_1,\dots,z_l) \leq C_{k,l}$$
for some fixed $C_{k,l}$ depending only on $k,l$.
\end{lemma}

This lemma follows fairly easily from the computations in \cite{borodin}; we give the details in Appendix \ref{real-app}.  We will need this lemma in order to control the event of having two real eigenvalues that are very close to each other, or a complex eigenvalue very close to the real axis, as in those cases, one is close to a transition in which two real eigenvalues become complex or vice versa, creating a potential instability in the correlation functions $\rho^{(k,l)}_n$.  One can in fact establish stronger \emph{level repulsion}
 estimates which provide some decay on $\rho^{(k,l)}_n(x_1,\dots,x_k,z_1,\dots,z_l)$ as two of the $x_1,\ldots,x_k,z_1,\ldots,z_l$ get close to each other, or as one of the $z_i$ gets close to the real axis, but we will not need such estimates here.

We then have the following analogue of Theorem \ref{main-alt}, which is the second main result of this paper:

\begin{theorem}[Four Moment Theorem for real  matrices]\label{main-alt-2}  Let $M_n, \tilde M_n$ be independent-entry matrix ensembles with real coefficients, obeying Condition \condone, such that $M_n$ and $\tilde M_n$ both match moments with the real gaussian matrix ensemble to fourth order.  Let $k,l \geq 0$ be fixed integers, and let
let $x_1,\dots,x_k$ and $z_1,\dots,z_l \in \C$ be bounded.  Assume that $n$ is even. Let $F: \R^k \times \C^l \to \C$ be a smooth function which admits a decomposition of the form
$$ F(y_1,\dots,y_k,w_1,\ldots,w_l) = \sum_{i=1}^m G_{i,1}(y_1) \ldots G_{i,k}(y_k) F_{i,1}(w_1) \ldots F_{i,l}(w_l)$$
for some fixed $m$ and some smooth functions $G_{i,p}: \R \to \C$ and $F_{i,j}: \C \to \C$ for $i=1,\ldots,m$, $p=1,\dots,k$ and $j=1,\ldots,l$ supported on the interval $\{ y \in \R: |y| \leq C\}$ and disk $\{ w \in \C: |w| \leq C \}$ respectively, obeying the derivative bounds
$$ |\nabla^a G_{i,p}(y)|, |\nabla^a F_{i,j}(w)| \leq C$$
for all $0 \leq a \leq 5$, $i=1,\dots,m$, $p=1,\dots,k$, $j=1,\dots,l$, $y \in \R$, and $w \in \C$, and some fixed $C$.  Let $\rho^{(k,l)}_n, \tilde \rho^{(k,l)}_n$ be the correlation functions for $M_n, \tilde M_n$ respectively.  Then
\begin{align*}
& \int_{\R^k} \int_{\C^l} F(y_1,\dots,y_k,w_1,\dots,w_l) \rho^{(k,l)}_n(\sqrt{n} x_1 + y_1,\dots,\sqrt{n} x_k + y_k,\\
&\quad\quad \sqrt{n} z_1 + w_1,\dots,\sqrt{n} z_l + w_l)\ dw_1 \dots dw_l dy_1 \dots dy_k \\
&\quad =
\int_{\R^k} \int_{\C^l} F(y_1,\dots,y_k,w_1,\dots,w_l) \tilde \rho^{(k,l)}_n(\sqrt{n} x_1 + y_1,\dots,\sqrt{n} x_k + y_k, \\
&\quad\quad \sqrt{n} z_1 + w_1,\dots,\sqrt{n} z_l + w_l)\ dw_1 \dots dw_l dy_1 \dots dy_k + O(n^{-c}).
\end{align*}
for some absolute constant $c>0$ (independent of $k,l$).  Furthermore, the implicit constant in the $O(n^{-c})$ notation is uniform over all $x_1,\dots,x_k$ and $z_1,\ldots,z_l$ in the bounded regions $\{ x \in \R: |x| \leq C\}$ and $\{ z \in \C: |z| \leq C \}$ respectively.
\end{theorem}

As will be seen in Section \ref{four-real}, the proof of Theorem \ref{main-alt-2} proceeds along the same lines as Theorem \ref{main-alt}, but with some additional arguments involving Lemma \ref{corf} required to prevent pairs of eigenvalues from escaping or entering the real axis due to collisions.  It is because of these additional arguments that matching to fourth order, rather than third order, is required.  It is however expected that the moment conditions should be relaxed; see for instance Figures \ref{figure:cl-1}, \ref{figure:cl-2} for the close resemblance in spectral statistics between real gaussian and Bernoulli matrices, which only match to third order rather than to fourth order.

\begin{remark} In \cite{sinclair}, some explicit formulae for the correlation functions of real gaussian matrices in the case of odd $n$ were given, while in \cite{mays} a relationship between the correlation functions for odd and even $n$ is established.  In principle, one could use either of these two results to extend Lemma \ref{corf} to the odd $n$ case. Once the odd case of Lemma \ref{corf}
is obtained,   Theorem \ref{main-alt-2} extends automatically  to this case.  Due to space limitations, we do not attempt to execute this calculation here.
\end{remark}

\begin{figure}
\begin{center}
\scalebox{.3}{\includegraphics{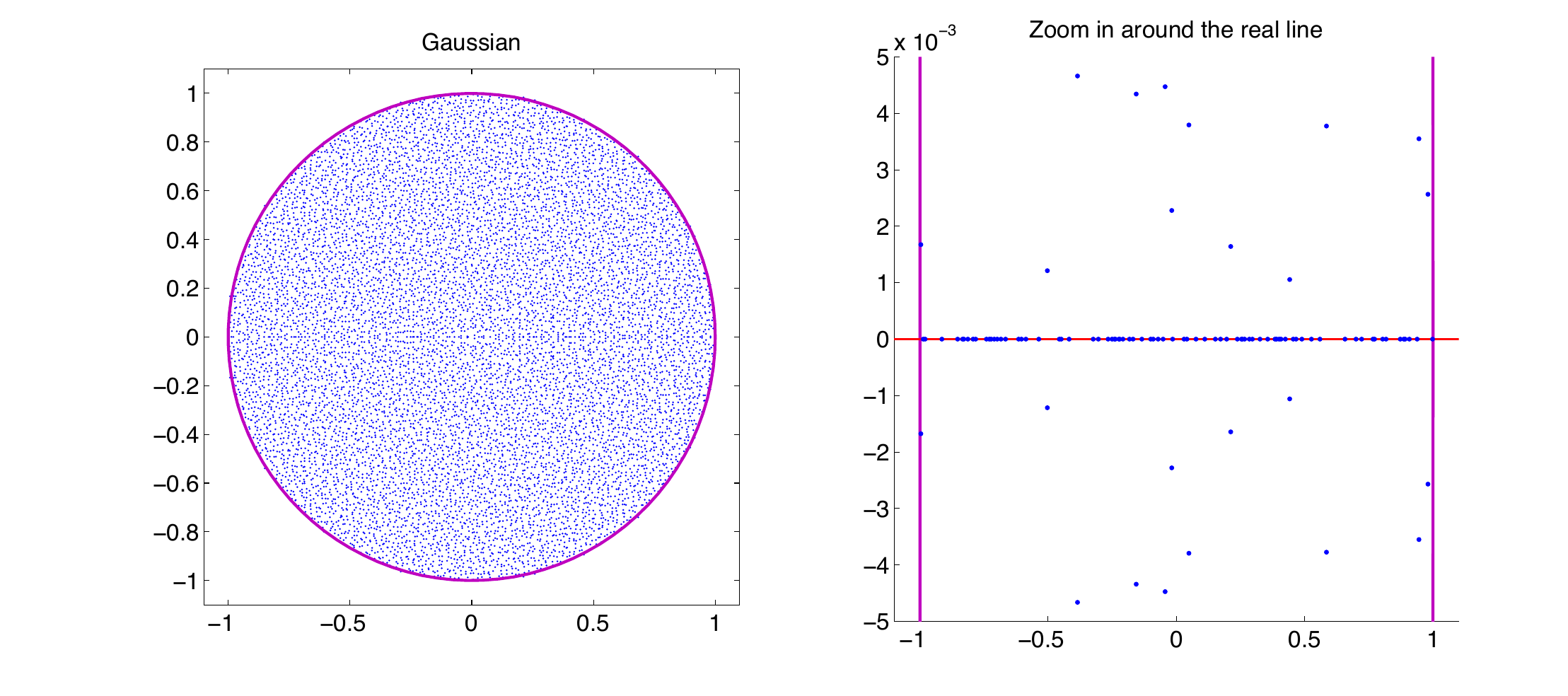}}
\end{center}
\caption{The spectrum of a random real gaussian $10,000 \times 10,000$ matrix, with additional detail near the origin to show the concentration on the real axis.  Thanks to Ke Wang for the data and figure.}
\label{figure:cl-1}
\end{figure}

\begin{figure}
\begin{center}
\scalebox{.3}{\includegraphics{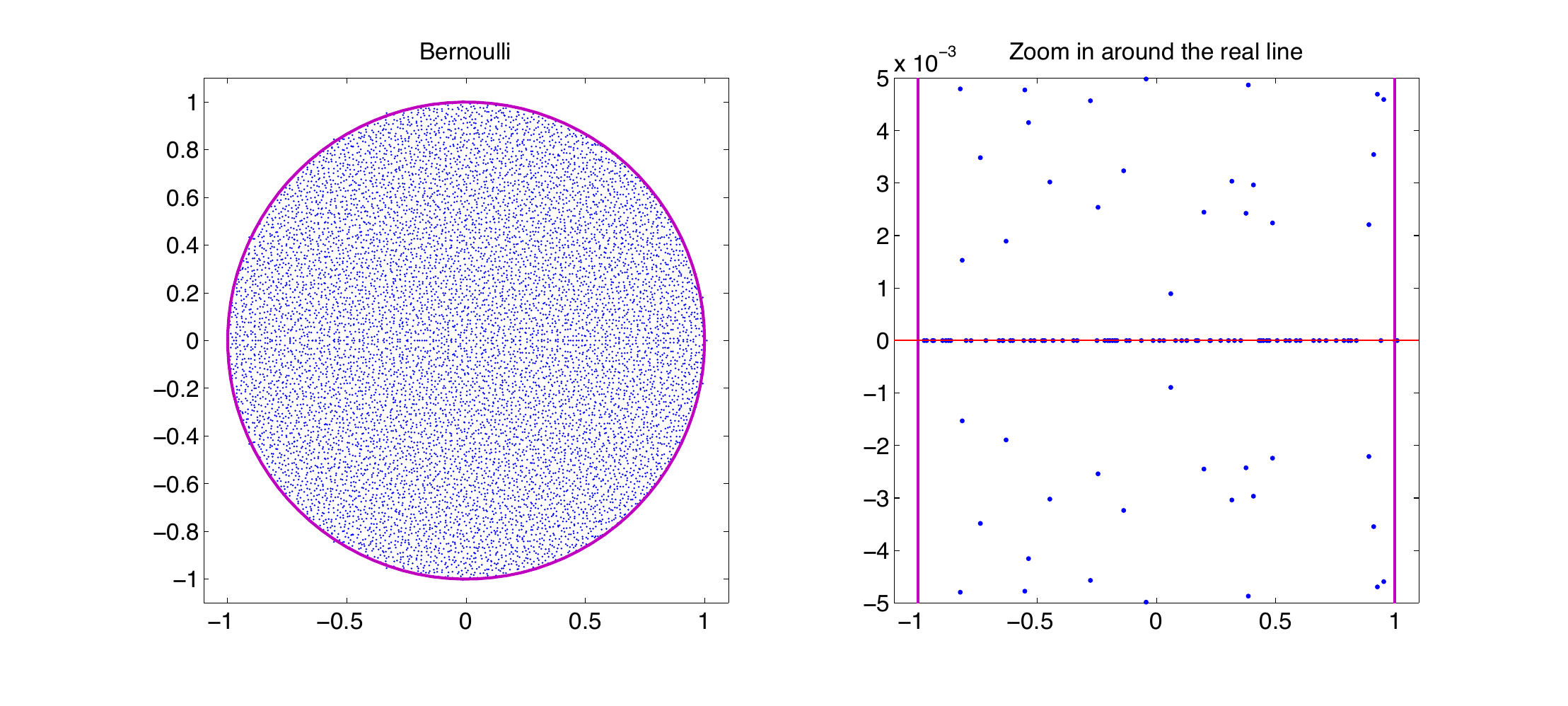}}
\end{center}
\caption{The spectrum of a random real Bernoulli $10,000 \times 10,000$ matrix, with additional detail near the origin.  Thanks to Ke Wang for the data and figure.}
\label{figure:cl-2}
\end{figure}

We now turn to applications of Theorem \ref{main-alt-2}.  In the complex case, the asymptotics for complex gaussian matrices given in Lemma \ref{kernel} could be extended to other independent entry matrices using Theorem \ref{main-alt}, yielding Corollary \ref{main}.  We now develop some analogous results in the real gaussian case.  We first recall the following result of Borodin and Sinclair \cite{borodin}:

\begin{lemma}[Kernel asymptotics, real case]\label{realkernel}  Let $k,l \geq 0$ be fixed natural numbers, and let $z$ be a fixed complex number.  Assume either that $k=0$, or that $z$ is real.  Then there is a function $\rho^{(k,l)}_{\infty,z}: \R^k \times \C^l \to \R^+$ with the property that one has the pointwise convergence
$$ \rho^{(k,l)}_n( \sqrt{n} z + y_1,\ldots, \sqrt{n} z + y_k, \sqrt{n} z + w_1,\ldots,\sqrt{n} z + w_l ) \to \rho^{(k,l)}_{\infty,z}(y_1,\ldots,y_k,w_1,\ldots,w_l)
$$
as $n \to \infty$, provided that $M_n$ is drawn from the real gaussian ensemble and $n$ is restricted to be even.
\end{lemma}

\begin{proof}  See \cite[Section 7]{borodin-0} or \cite[Section 8]{borodin}.  The limit $\rho^{(k,l)}_{\infty,z}$ is explicitly computed in these references, although when $z$ is real the limit is quite complicated, being given in terms of a Pfaffian of a moderately complicated matrix kernel involving the error function $\operatorname{erf}$.  However, when $z$ is strictly complex the limit is the same as in the complex gaussian case, thus $\rho^{(0,l)}_{\infty,z} = \rho^{(l)}_{\infty,z,\ldots,z}$; see \cite{borodin} for further details.  It is likely that the same asymptotic also holds for odd $n$, by using the explicit formulae in \cite{sinclair} or the relation between the odd and even $n$ correlation functions given in \cite{mays}; if the restriction to even $n$ is similarly dropped from Lemma \ref{corf}, then Corollary \ref{main-2} below can be extended to the odd $n$ case.  However, we will not pursue this matter here.
\end{proof}

We can then obtain the following universality theorem for the correlation functions of real matrices:

\begin{corollary}[Universality for real  matrices]\label{main-2}  Let $M_n$ be an independent-entry matrix ensemble with real coefficients obeying Condition \condone, and which matches moments with the real gaussian matrix ensemble to fourth order.  Assume $n$ is even.
Let $k,l \geq 0$ be fixed natural numbers, and let $z$ be a fixed complex number.  Assume either that $k=0$, or that $z$ is real.  Let $F: \R^k \times \C^l \to \R^+$ be a fixed continuous, compactly supported function.  Then
\begin{align*}
& \int_{\R^k} \int_{\C^l_*} F(y_1,\dots,y_k,w_1,\dots,w_l)  \\
&\quad\quad \rho^{(k,l)}_n(\sqrt{n} z + y_1,\dots,\sqrt{n} z + y_k, \sqrt{n} z + w_1,\dots,\sqrt{n} z + w_l )\\
 &\quad\quad\quad\quad\ dw_1 \ldots dw_l dy_1 \dots d y_k \\
&\quad \to
\int_{\R^k} \int_{\C^l_*} F(y_1,\dots,y_k,w_1,\dots,w_l) \\
&\quad\quad \rho^{(k,l)}_{\infty,z}(y_1,\dots,y_k,w_1,\dots,w_l)\\
 &\quad\quad\quad\quad\ dw_1 \ldots dw_l dy_1 \dots d y_k,
\end{align*}
where $\rho^{(k,l)}_{\infty,x_1,\dots,x_k,z_1,\dots,z_l}$ is as in Lemma \ref{realkernel}.
\end{corollary}

\begin{proof}  In the case when $M_n$ is drawn from the real gaussian ensemble, this follows from Lemma \ref{realkernel}, Lemma \ref{corf}, and the dominated convergence theorem.  The extension to more general independent-entry matrices then follows from Theorem \ref{main-alt-2} by repeating the arguments used to prove Corollary \ref{main}.
\end{proof}

As in the complex case, Theorem \ref{main-alt-2} can be used to (partially) extend various known facts about the distribution of the eigenvalues of a real gaussian matrices to other real independent entry matrices.  Rather than giving an exhaustive list of such extensions, we illustrate this with two sample applications.  Let $N_\R(M_n)$ denote the number of real zeroes of a random matrix $M_n$.  Thanks to earlier results \cite{eks, FN}, we have the following asymptotics:

\begin{figure}
\begin{center}
\scalebox{.3}{\includegraphics{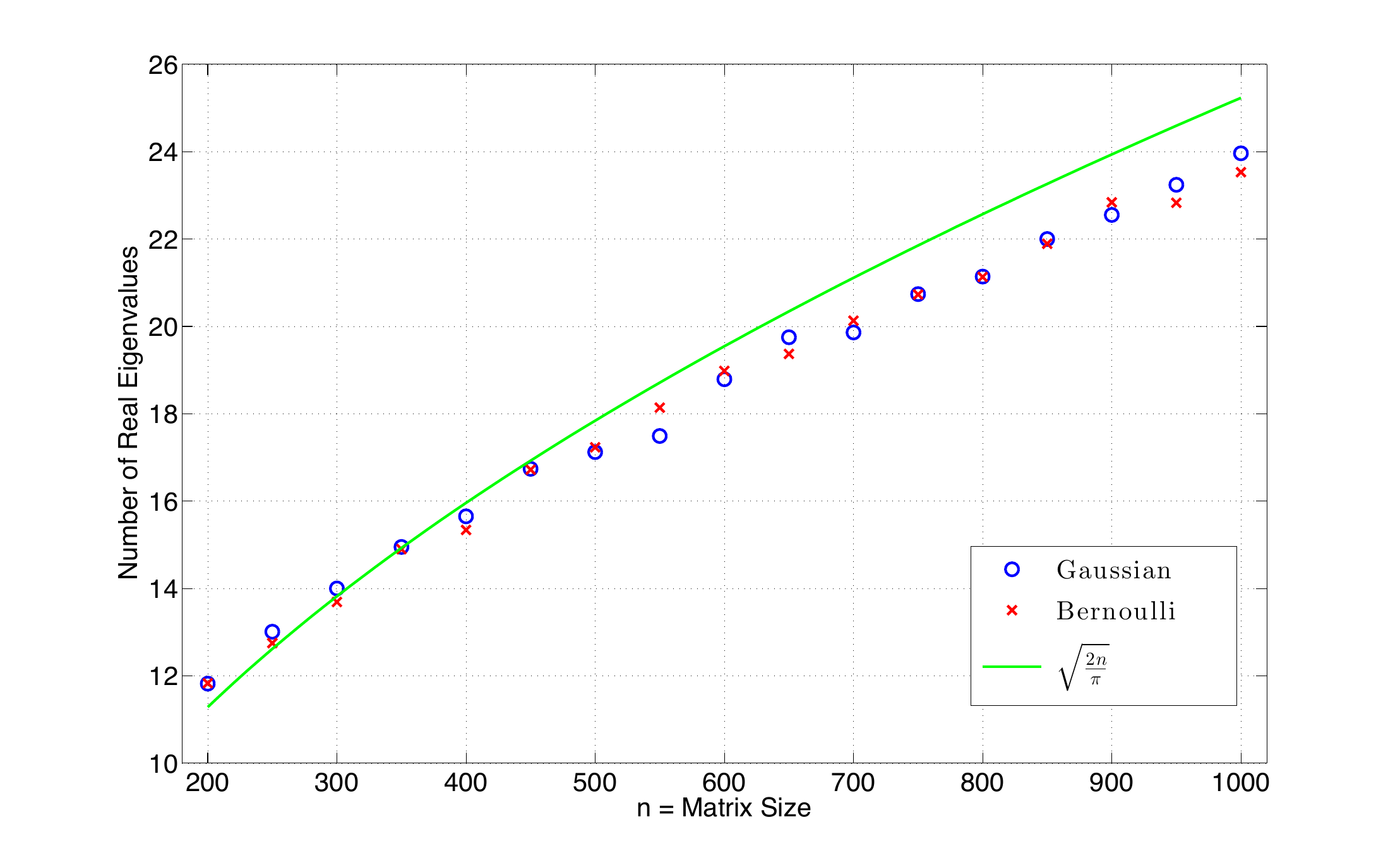}}
\end{center}
\caption{The empirical average number of real eigenvalues of $200$ samples of real gaussian and real Bernoulli matrices of various sizes, plotted against $\sqrt{\frac{2n}{\pi}}$.  Thanks to Ke Wang for the data and figure.}
\label{figure:real-ev}
\end{figure}

\begin{theorem}[Real eigenvalues of a real gaussian matrix]\label{really} Let $M_n$ be drawn from the real gaussian ensemble.  Then
$$ \E N_\R(M_n) = \sqrt{\frac{2n}{\pi}} + O(1)$$
and
$$ \Var N_\R(M_n) = (2-\sqrt{2}) \sqrt{\frac{2n}{\pi}} + o(\sqrt{n})$$
\end{theorem}

\begin{proof} The expectation bound was established in \cite{eks}, and the variance bound in \cite{FN}.  In fact, more precise asymptotics are available for both the expectation and the variance; we refer the reader to these two papers \cite{eks}, \cite{FN} for further details.
\end{proof}

By using the above universality results, we may partially extend this result to more general ensembles:

\begin{corollary}[Real eigenvalues of a real  matrix]\label{really-2}  Let $M_n$ be an independent-entry matrix ensemble with real coefficients obeying Condition \condone, and which matches moments with the real gaussian matrix ensemble to fourth order.  Assume $n$ is even. Then
$$ \E N_\R(M_n) = \sqrt{\frac{2n}{\pi}} + O(n^{1/2-c})$$
and
$$ \Var N_\R(M_n) = O( n^{1-c} )$$
for some fixed $c>0$.  In particular, from Chebyshev's inequality, we have
$$ N_\R(M_n) = \sqrt{\frac{2n}{\pi}} + O(n^{1/2-c'})$$
with probability $1-O(n^{-c'})$, for some fixed $c'>0$.
\end{corollary}

We prove this result in Section \ref{real-sec}.

As another quick application, we can show for many ensembles that most of the eigenvalues are simple:

\begin{corollary}[Most eigenvalues simple]\label{simex}  Let $M_n$ be an independent matrix ensemble obeying Condition \condone, and which matches moments with the real or complex gaussian matrix to fourth order.  In the real case, assume $n$ is even.  Then with probability $1-O(n^{-c})$, at most $O(n^{1-c})$ of the complex eigenvalues, and $O(n^{1/2-c})$ of the real eigenvalues, are repeated, for some fixed $c>0$.
\end{corollary}

We establish this result in Section \ref{real-sec} also.  It should in fact be the case that with overwhelming probability, none of the eigenvalues are repeated, but this seems to be beyond the reach of our methods.

We thank Anthony Mays and the anonymous referees for corrections and help with the references.

\section{Key ideas and a sketch of the proof}

 The proof of the four moment theorem for (Hermitian) Wigner ensembles in \cite{TVlocal1}
 is based on the \emph{Lindeberg exchange strategy}, in which one shows that various statistics of ensembles are stable with respect to the swapping of one or two of the coefficients of that ensemble.  The original argument in \cite{TVlocal1} was based on a swapping analysis of individual eigenvalues $\lambda_i(M_n)$, which was somewhat complicated technically; but in \cite{ESYY}, \cite{knowles} it was observed that one could work instead with the  simpler swapping analysis of resolvents\footnote{Here and in the sequel we adopt the abbreviation $z$ for the scalar multiple $zI$ of the identity matrix.} (or Greens functions) $R(z) := (W_n-z)^{-1}$, particularly if one was mainly focused on obtaining a Four Moment Theorem for correlation functions, rather than for individual eigenvalues (which in any event are not natural to work with in the non-Hermitian case).  In all of these arguments for Wigner matrices, a key role was played by the \emph{local semi-circle law}, which could in turn be proven by exploiting concentration results for the Stieltjes transform $s(z) := \frac{1}{n} \tr (W_n-z)^{-1}$ of a Wigner matrix.  Again, we refer the reader to the preceding surveys for details.

Our strategy of proof of Theorem \ref{main-alt} and Theorem \ref{main-alt-2} is broadly analogous to that in the Hermitian case, in that it relies on a four moment theorem (Theorem \ref{loglower} below) and on a local circular law (Theorem \ref{local-circ} below).  However, this is highly non-trivial to execute this plan. We are going to need a number of new ideas, coming from different fields of mathematics, and a fair amount of delicate analysis using advanced
sharp concentration tools.

To start,   there is an essential
difference  between handling   non-Hermitian and Hermitian matrices, namely that the spectrum of a non-Hermitian matrix is highly unstable
(see \cite{bai} for a discussion). Due to this difficulty, even the (global) circular law, which is the non-Hermitian analogue of Wigner semi-circle law,
required several decades of effort to prove, and was solved
completely only recently (see the surveys \cite{TV-survey, chafai} for further discussion).  For this reason, it is no longer practical to make the resolvent $(M_n-z)^{-1}$ (and the closely related Stieltjes transform $\frac{1}{n} \tr (M_n - z)^{-1}$) the principal object of study.  Instead, following the foundational works of Girko \cite{girko} and Brown \cite{brown}, we shall focus on the  \emph{log-determinant}
$$ \log |\det( M_n - z)|$$
for a complex number parameter $z$.

The  log-determinant is connected to the eigenvalues of the iid matrix $M_n$ via the  obvious identity
\begin{equation}\label{log-det-1}
\log |\det( M_n - z)| = \sum_{i=1}^n \log |\lambda_i(M_n) - z|.
\end{equation}

In order to restrict to a local region, our  idea is to
use  \emph{Jensen's formula}.  Suppose that $f$ is an analytic function in a region in the complex plane which contains the closed disk $D$ of radius $r$ about the origin, $a_1, a_2, \ldots, a_n $ are the zeros of $f$ in the interior of $D $  (counting multiplicity), and $f(0) \neq 0$, then

$$ \log |f(0) |= \sum_{i=1}^k \log \frac {|a_i|}{r} + \frac{1}{2\pi} \int_0^{2\pi} \log |f (r e^{\sqrt{-1} \theta} ) | d \theta. $$

Applied  Jensen's formula to \eqref{log-det-1}, we obtain
\begin{equation}\label{jensen-det}
\begin{split}
\log |\det(M_n-z_0)| &= - \sum_{1 \leq i \leq n: \lambda_i(M_n) \in B(z_0,r)} \log \frac{r}{|\lambda_i(M_n)-z_0|} \\
&\quad + \frac{1}{2\pi} \int_0^{2\pi} \log |\det(M_n-z_0-re^{\sqrt{-1}\theta})|\ d\theta
\end{split}
\end{equation}
for any ball $B(z_0,r)$ (with the convention that both sides are equal to $-\infty$ when $z_0$ is an eigenvalue of $M_n$).

From \eqref{jensen-det}, we see (in principle, at least) that information on the (joint) distribution of the log-determinants $\log|\det(M_n-z)|$ for various values of $z$ should lead to information on the eigenvalues of $M_n$, and in particular on the $k$-point correlation functions $\rho^{(k)}_n$ of $M_n$.
As Jensen formula is a  classical tool in complex analysis, this step looks quite robust and would potentially find applications in the study of local properties
of many other random processes.

 On the other hand, we can also write the log-determinant in terms of the \emph{Hermitian} $2n \times 2n$ random matrix
\begin{equation}\label{wnz}
W_{n,z} :=
\frac{1}{\sqrt{n}} \begin{pmatrix}
0 & M_n -z \\
(M_n-z)^* & 0
\end{pmatrix}
\end{equation}
via the easily verified identity
\begin{equation}\label{manz}
\log |\det(M_n - z)| = \frac{1}{2} \log|\det W_{n,z}| + \frac{1}{2} n \log n.
\end{equation}
This observation is known as the \emph{Girko Hermitization trick}, and in principle reduces the spectral theory of non-Hermitian matrices to the spectral theory of Hermitian matrices.

The log-determinant of $W_{n,z}$ is in turn related to other spectral information of $W_{n,z}$, such as the \emph{Stieltjes transform}\footnote{We use $\sqrt{-1}$ to denote the standard imaginary unit, in order to free up the symbol $i$ to be an index of summation.}
$$ s_{W_{n,z}}(E+\sqrt{-1} \eta) := \frac{1}{2n} \tr\left( (W_{n,z} - E-\sqrt{-1} \eta)^{-1} \right)$$
of $W_{n,z}$, for instance via the identity
\begin{equation}\label{imit}
 \log|\det W_{n,z}| = \log|\det(W_{n,z}-\sqrt{-1} T)| - 2n \Im \int_0^T s_{W_{n,z}}(\sqrt{-1} \eta)\ d\eta,
\end{equation}
valid for arbitrary $T>0$.  Thus, in principle at least, information on the distribution of the Stieltjes transform $s_{W_{n,z}}$ will imply information on the log-determinant of $W_{n,z}$, and hence on $M_n-z$, which in turn gives information on the eigenvalue distribution of $M_n$.  This is the route taken, for instance, to establish the circular law for iid matrices; see \cite{TV-survey, chafai} for further discussion.  There is a non-trivial issue with the possible divergence or instability of the integral in \eqref{imit} near $\eta=0$, but it is now well understood how to control this issue via a regularisation or truncation of this integral, provided that one has adequate bounds on the least singular value of $W_{n,z}$; again, see \cite{TV-survey, chafai} for further discussion.  Fortunately, we and many
other researchers have proved such bounds in previous papers, using methods from a seemingly unrelated area of Additive Combinatorics (see Proposition \ref{lsv} below).

There is a significant  technical issue arising from the fact that formulae such as \eqref{imit} or \eqref{jensen-det} require one to control the value of various random functions, such as log-determinants or Stieltjes transforms, for an uncountable number of choices of parameters such as $z$ and $\eta$, so that one can no longer directly use union bound to control exceptional events when the expected control on these quantities fails.  To overcome this, we
appeal to  the  Monte Carlo method, frequently used in combinatorics and  theoretical compute science. This method enables us to
use random sampling arguments to replace many of these integral expressions by discrete, random, approximations,
  to which the union bound can be safely applied (see Section 5).

The application of the Monte Carlo method (Lemma \ref{sampling}), on the other hand, is far from straightforward, since in certain situations (see Section 6), the variance
is too high and so the bound implied by Lemma \ref{sampling} becomes rather weak.  We handle this situation by a {\it variance reduction} argument,
exploiting   analytical properties of the relevant functions. This step also looks  robust and may be useful for practitioners of the
Monte Carlo method in other fields.

After these steps, the rest of the proof  essentially boils down to error control,  in form of a sharp concentration inequality (Theorem \ref{loglower-gaussian}), which will be done  by analyzing a delicate (and rather unstable) random process, using recent martingale inequalities and various adhoc ideas.

\begin{remark}  For Hermitian ensembles, swapping methods (such as the Four Moment Theorem) are not the only way to obtain universality results; there is also an important class of methods (such as the local relaxation flow method) that are based on analysing the effect of a Dyson-type Brownian motion on the spectrum of a random matrix ensemble; see e.g. \cite{Erd}.  However, there is a significant obstruction to adapting such methods to the non-Hermitian setting, because the equations of the analogue to Dyson Brownian motion either\footnote{One can explain this by observing that in the Hermitian case, the eigenvalues determine the matrix up to a $U_n(\C)$ symmetry, but in the non-Hermitian case the symmetry group is now the much larger group $GL_n(\C)$.  Dyson Brownian motion is $U_n(\C)$-invariant, but is not $GL_n(\C)$-invariant, which is why this motion can be reduced to dynamics purely on eigenvalues in the Hermitian case but not in the non-Hermitian one.} couple together the eigenvectors and the eigenvalues in a complicated fashion, or need to be phrased in terms of a triangular form of the matrix, rather than a diagonal one (cf. \cite{Meh}).  We were unable to resolve these difficulties in the non-Hermitian case, and rely solely on swapping methods instead; unfortunately, this then requires us to place moment matching hypotheses on our matrix ensembles.  It seems of interest to develop further tools that are able to remove these moment matching hypotheses in non-Hermitian settings.
\end{remark}

\subsection{Key propositions}

The proof of Theorem \ref{main-alt} relies on two key facts, both of which may be of independent interest.  The first is a ``local circular law''.  Given a subset $\Omega$ of the complex plane, let
$$N_\Omega = N_\Omega[M_n] := |\{ 1\leq i \leq n: \lambda_i(M_n) \in \Omega \}$$
denote the number of eigenvalues of $M_n$ in $\Omega$.

\begin{theorem}[Local circular law]\label{local-circ}  Let $M_n = (\xi_{ij})_{1 \leq i,j \leq n}$ be an independent-entry matrix with independent real and imaginary parts obeying Condition \condone, and which matches either the real or complex gaussian matrix to third order.  Then for any fixed $C>0$, one has with overwhelming probability\footnote{See Section \ref{notation-sec} for a definition of this term, and for the definition of asymptotic notation such as $o(1)$ and $\ll$.} that
\begin{equation}\label{nbr}
 N_{B(z_0,r)} = \int_{B(z_0,r)} \frac{1}{\pi} 1_{|z| \leq \sqrt{n}}\ dz + O( n^{o(1)} r )
\end{equation}
uniformly for all $z_0 \in B(0,C \sqrt{n})$ and all $r \geq 1$.  In particular, we have
\begin{equation}\label{nbr-0}
N_{B(z_0,r)} \leq n^{o(1)} r^2
\end{equation}
with overwhelming probability, uniformly for all $z_0 \in B(0, C \sqrt{n})$ and all $r \geq 1$.
\end{theorem}

\begin{remark}\label{remark:outliers}
The bound \eqref{nbr} is probably not best possible, even if one ignores the $n^{o(1)}$ term.  In the complex gaussian case, it has been shown \cite{rider} that the variance of $N_{B(z_0,r)}$ is actually of order $r$, suggesting a fluctuation of $O(n^{o(1)} r^{1/2})$ rather than $O(n^{o(1)} r)$; the closely related results in Theorem \ref{clt-gauss} and Corollary \ref{clt-2} also support this prediction.   Also notice that we assume only three matching moments in this theorem, so the statement applies for instance to random sign matrices (which match the real gaussian ensemble to third order).   For our applications to Theorems \ref{main-alt}, \ref{main-alt-2}, we do not need the full strength \eqref{nbr} of the above theorem; the weaker bound \eqref{nbr-0} will suffice.
\end{remark}

\begin{remark}\label{breaking} Very recently, Bourgade, Yau, and Yin \cite{byy} have established a variant of Theorem \ref{local-circ} (and also Theorem \ref{loglower}) which does not require matching to third order, but with the disk $B(z_0,r)$ assumed to lie a distance at least $\eps \sqrt{n}$ from the circle $\{ |z| = \sqrt{n}\}$ for some fixed $\eps>0$.  By using the main result of \cite{byy} as a substitute for Theorem \ref{local-circ} (and also Theorem \ref{loglower}), we may similarly remove the third order matching hypotheses from Theorem \ref{main-alt}, at least in the case when $z_1,\ldots,z_k$ stay a distance $\eps \sqrt{n}$ from the circle $\{ |z| = \sqrt{n}\}$. Since the initial release of this paper, an alternate proof of Theorem \ref{local-circ} (in the case when one matches the complex gaussian ensemble to third order, as opposed to the real gaussian ensemble) which works both in the bulk and in the edge was given in \cite{byy-2}.
\end{remark}

The second key fact is a ``Four Moment Theorem'' for the log-determinants $\log |\det(M_n - z)|$:

\begin{theorem}[Four Moment Theorem for determinants]\label{four-moment}  Let $c_0>0$ be a sufficiently small absolute constant.  Let $M_n, M'_n$ be two independent random matrices with independent real and imaginary parts obeying Condition \condone, which match each other to fourth order, and which both match the real gaussian matrix (or both match the complex gaussian matrix) to third order.  Let $1 \leq k \leq n^{c_0}$, let $C > 0$ be fixed, and let $z_1,\dots,z_k \in B(0,C \sqrt{n})$.  Let $G: \R^k \to \C$ be a smooth function obeying the derivative bounds
$$ |\nabla^j G(x_1,\dots,x_k)| \ll n^{c_0}$$
for all $j=0,\dots,5$ and $x_1,\dots,x_k \in \R$, where $\nabla$ denotes the gradient in $\R^k$.  Then we have
\begin{align*}
& \E G( \log|\det(M_n-z_1)|, \dots, \log|\det(M_n-z_k)| ) \\
&\quad = \E G( \log|\det(M'_n-z_1)|, \dots, \log|\det(M'_n-z_k)| ) + O( n^{-c_0} ),
\end{align*}
with the convention that the expression $G( \log|\det(M_n-z_1)|, \dots, \log|\det(M_n-z_k)| )$ vanishes if one of the $z_1,\dots,z_k$ is an eigenvalue of $M_n$, and similarly for the expression $G( \log|\det(M'_n-z_1)|, \dots, \log|\det(M'_n-z_k)| )$.
\end{theorem}

The proof of Theorem \ref{main-alt} follows fairly easily from Theorem  \ref{local-circ} (in fact, we will only need the weaker conclusion \eqref{nbr-0})
and Theorem \ref{four-moment} (and \eqref{potential}), using the well-known connection between spectral statistics and the log-determinant which goes back to the work of Girko \cite{girko} and Brown \cite{brown}, and which was mentioned previously in this introduction; we give this implication in Section \ref{four}.  A slightly more sophisticated version of the same argument also works to give Theorem \ref{main-alt-2}; we give this implication in Section \ref{four-real}.

It remains to establish the local circular law (Theorem \ref{local-circ}) and the four moment theorem for log-determinants (Theorem \ref{four-moment}).  The key lemma in the establishment of the local circular law is  the following concentration result for the log-determinant.

\begin{definition}[Concentration]\label{cono}  Let $n > 1$ be a large parameter, and let $X$ be a real or complex random variable depending on $n$.  We say that $X$ \emph{concentrates around $M$} for some deterministic scalar $M$ (depending on $n$) if one has
$$ X = M + O(n^{o(1)} )$$
with overwhelming probability.  Equivalently, for every $\eps,A > 0$ independent of $n$, one has $X = M + O(n^\eps)$ outside of an event of probability $O(n^{-A})$.  We say that $X$ \emph{concentrates} if it concentrates around some $M$.
\end{definition}

\begin{theorem}[Concentration bound on log-determinant]\label{loglower} Let $M_n = (\xi_{ij})_{1 \leq i,j \leq n}$ be an independent-entry matrix obeying Condition \condone and matching the real or complex gaussian ensemble to third order.  Then for any fixed $C > 0$, and any $z_0 \in B(0,C)$, $\log |\det(M_n - z_0 \sqrt{n})|$ concentrates around $\frac{1}{2} n \log n + \frac{1}{2} n (|z_0|^2 - 1)$ for $|z_0| \leq 1$ and around $\frac{1}{2} n \log n + n \log |z_0|$ for $|z_0| \geq 1$, uniformly in $z_0$.
\end{theorem}

\begin{remark} \label{remark:whythree}
The reason we require only three moments in this theorem instead of four (as in the previous theorem) is that in this theorem the error
in Definition \ref{cono} is allowed to be a positive power of $n$ while in the previous one it needs to be a negative power.  We remark that this theorem is consistent with \eqref{log-det-1} and the circular law; indeed, the quantity $\int_{B(0,1)} \frac{1}{\pi} \log |z-z_0|\ dz$ can be computed to be equal to $\frac{1}{2} (|z_0|^2-1)$ when $|z_0| \leq 1$ and $\log |z_0|$ when $|z_0| \geq 1$.  As in Remark \ref{breaking}, a variant of Theorem \ref{loglower} without the third order hypothesis, but requiring $z_0$ bounded away from the circle $\{|z|=1\}$, was recently given in \cite{byy}.
\end{remark}

We give the derivation of Theorem \ref{local-circ} from Theorem \ref{loglower} in Section \ref{circular-sec}.  The main tools are Jensen's formula \eqref{jensen-det} and a random sampling argument to approximate the integral in \eqref{jensen-det} by a Monte Carlo type sum, which can then be estimated by Theorem \ref{loglower}.

It remains to establish Theorem \ref{four-moment} and Theorem \ref{loglower}.  For both of these theorems, we will work with the Hermitian matrix $W_{n,z}$ defined in \eqref{wnz}, taking advantage of the identity \eqref{manz}.  In order to manipulate quantities such as the log-determinant of $W_{n,z}$ efficiently, we will need some basic estimates on the spectrum of this operator (as well as on related objects, such as resolvent coefficients).  We first need a lower bound on the least singular value that is already in the literature:

\begin{proposition}[Least singular value]\label{lsv} Let $M_n$ be an independent-entry matrix ensemble with independent real and imaginary parts, obeying Condition \condone, and let $z_0 \in B(0,C \sqrt{n})$ for some fixed $C>0$.  Then with overwhelming probability, one has
$$ \inf_{1 \leq i \leq n} |\lambda_i(W_{n,z})| \geq n^{-\log n}.$$
Furthermore, for any fixed $c_0>0$ one has
$$ {\bf P}( \inf_{1 \leq i \leq n} |\lambda_i(W_{n,z})| \leq n^{-1/2-c_0} ) \ll n^{-c_0/2}.$$
The bounds in the tail probability are uniform in $z_0$.
\end{proposition}

\begin{proof} Note from \eqref{wnz} that $\inf_{1 \leq i \leq n} |\lambda_i(W_{n,z})|$ is the least singular value of $\frac{1}{\sqrt{n}} (M_n-z)$.
The first bound then follows from \cite[Theorem 2.5]{TV-cond} (and can also be deduced from the second bound).  The lower bound $n^{-\log n}$ can be improved to any bound decaying faster than a polynomial, but for our applications any lower bound of the form $\exp(-n^{o(1)})$ will suffice.  The second bound follows from \cite[Theorem 3.2]{TVsmooth} (and can also be essentially derived from the results in \cite{RV}, after adapting those results to the case of random matrices whose entries are uncentered (i.e. can have non-zero mean)).  We remark that in the $z_0$ case, significantly sharper bounds can be obtained; see \cite{RV} for details.
\end{proof}

\begin{remark}
The proof of this bound relies heavily on the so-called \emph{inverse Littlewood-Offord theory} introduced
by the authors in \cite{TVannals}, which was motivated by Additive Combinatorics
(see \cite[Chapter 7]{TVbook}), a seemingly unrelated area. Interestingly, this is, at this point,
the only way to obtain good lower bound on the least singular values of random matrices when the ensemble is discrete (see also \cite{RV, RV1, TV-survey} for more  results and discussion).
\end{remark}

Next, we establish some bounds on the counting function
$$ N_I := |\{ 1 \leq i \leq n: \lambda_i(W_{n,z}) \in I \}|,$$
and on coefficients $R(\sqrt{-1}\eta)_{ij}$ of the resolvents $R(\sqrt{-1} \eta) := (W_{n,z} - \sqrt{-1}\eta)^{-1}$ on the imaginary axis.

\begin{proposition}[Crude upper bound on $N_I$]\label{ni} Let $M_n$ be an independent-entry matrix ensemble with independent real and imaginary parts, obeying Condition \condone.  Let $C > 0$ be fixed, and let $z_0 \in B(0,C\sqrt{n})$. Then with overwhelming probability, one has
$$ N_I \ll n^{o(1)} (1 + n|I|)$$
for all intervals $I$.
The bounds in the tail probability (and in the $o(1)$ exponent) are uniform in $z_0$.
\end{proposition}

\begin{remark} It is likely that one can strengthen Proposition \ref{ni} to a ``local distorted quarter-circular law'' that gives more accurate upper and lower bounds on $N_I$, analogous to the local semi-circular law from \cite{ESY1}, \cite{ESY2}, \cite{ESY3} (or, for that matter, the local circular law given by Theorem \ref{local-circ}).  However, we will not need such improvements here.
\end{remark}

\begin{proposition}[Resolvent bounds]\label{Resolv} Let $M_n$ be an independent-entry matrix ensemble with independent real and imaginary parts, obeying Condition \condone.  Let $C > 0$ be fixed, and let $z_0 \in B(0,C\sqrt{n})$. Then with overwhelming probability, one has
$$ |R(\sqrt{-1}\eta)_{ij}| \ll n^{o(1)} \left(1 + \frac{1}{n\eta}\right)$$
for all $\eta > 0$ and all $1 \leq i,j \leq n$.
\end{proposition}

\begin{remark}  One can also establish similar bounds on the resolvent (as well as closely related \emph{delocalization} bounds on eigenvectors) for more general spectral parameters $E+\sqrt{-1}\eta$.  However, in our application we will only need the resolvent bounds in the $E=0$ case.
\end{remark}

Propositions \ref{ni} and \ref{Resolv} are proven by standard Stieltjes transform techniques, based on analysis of the self-consistent equation of $W_{n,z}$ as studied for instance by Bai \cite{bai}, combined with concentration of measure results on quadratic forms.  The arguments are well established in the literature; indeed, the $z=0$ case of these theorems essentially appeared in \cite{TVlocal3}, \cite{ESYY}, while the analogous estimates for Wigner matrices appeared in \cite{ESY1}, \cite{ESY2}, \cite{ESY3}, \cite{TVlocal1}.  As the proofs of these results are fairly routine modifications of existing arguments in the literature, we will place the proof of these propositions in Appendix \ref{conc}.  We remark that in the very recent paper \cite{byy}, some stronger eigenvalue rigidity estimates for $W_{n,z}$ are obtained (at least for $z$ staying away from the unit circle $\{|z|=1\}$), which among other things allows one to prove variants of Theorem \ref{loglower} and Theorem \ref{local-circ} without the moment matching hypothesis, and without the need to study the gaussian case separately (see Theorem \ref{loglower-gaussian} below).

One can use Propositions \ref{lsv}, \ref{ni}, \ref{Resolv} to regularise the log-determinant of $W_{n,z}$, and then show that this log-determinant is quite stable with respect to swapping (real and imaginary parts of) individual entries of the $M_{n,z}$, so long as one keeps the matching moments assumption.   In particular, one can now establish Theorem \ref{four-moment} without much difficulty, using standard resolvent perturbation arguments; see
Section \ref{4mt-sec}.  A similar argument, which we give in Section \ref{lower-sec2}, reduces Theorem \ref{loglower} to the gaussian case.
Thus, after all these works, the remaining task is to prove

\begin{theorem}\label{loglower-gaussian}  Theorem \ref{loglower} holds when $M_n$ is drawn from the real or complex gaussian ensemble.
\end{theorem}

We prove this theorem in Section \ref{lower-sec}.
This section is the most technically involved part of the paper.
The starting point is to use an idea from our previous paper \cite{TV-determinant}, which studied the limiting distribution of the log-determinant of a shifted GUE matrix.  In that  paper, the first step was to conjugate the GUE matrix into the Trotter tridiagonal form \cite{trotter}, so that the log-determinant could be computed in terms of the solution to a certain linear stochastic difference equation.  In the case in this paper, the analogue of the Trotter tridiagonal form is a Hessenberg matrix form (that is, a matrix form which vanishes above the upper diagonal), which (after some  linear algebraic transformations) can be used to express the log-determinant $\log |\det(M_n - z_0 \sqrt{n})|$ in terms of the solution to a certain \emph{nonlinear} stochastic difference equation.  This Hessenberg form of the complex gaussian ensemble was introduced in \cite{kv}, although the difference equation we derive is different from the one used in that paper.  To obtain the desired level of concentration in the log-determinant, the main difficulty is then to satisfactorily control the interplay between the diffusive components of this stochastic difference equation, and the stable and unstable equilibria of the nonlinearity, and in particular to show that the deviation of the solution from the stable equilibrium behaves like a martingale.  This then allows us to deduce the desired concentration from a martingale concentration result  (see Proposition \ref{azuma} below).

\section{Notation}\label{notation-sec}

Throughout this paper, $n$ is a natural number parameter going off to infinity.  A quantity is said to be \emph{fixed} if it does not depend on $n$.  We write $X = O(Y)$, $X \ll Y$, $Y = \Omega(X)$, or $Y \gg X$ if one has $|X| \leq CY$ for some fixed $C$, and $X = o(Y)$ if one has $X/Y \to 0$ as $n \to \infty$.  Absolute constants such as $C_0$ or $c_0$ are always understood to be fixed.

We say that an event $E$ occurs with \emph{overwhelming probability} if it occurs with probability $1-O(n^{-A})$ for all fixed $A>0$. We use $1_E$ to denote the indicator of $E$, thus $1_E$ equals $1$ when $E$ is true and $0$ when $E$ is false.  We also write $1_\Omega(x)$ for $1_{x \in \Omega}$.

As we will be using two-dimensional integration on the complex plane $\C := \{ z = x+\sqrt{-1} y: x,y \in \R \}$ far more often than we will be using contour integration, we use $dz = dx dy$ to denote Lebesgue measure on the complex numbers, rather than the complex line element $dx + \sqrt{-1} dy$.

We use $N(\mu,\sigma^2)_\R$ to denote a real gaussian distribution of mean $\mu$ and variance $\sigma^2$, so that the probability distribution is given by $\frac{1}{\sqrt{2\pi \sigma^2}} e^{-(x-\mu)^2/2\sigma^2}\ dx$.  Similarly, we let $N(\mu,\sigma^2)_\C$ denote the complex gaussian distribution of $\mu$ and variance $\sigma^2$, so that the probability distribution is given by $\frac{1}{\pi \sigma^2} e^{-|z-\mu|^2/\sigma^2}\ dz$.  Of course, the two distributions are closely related: the real and imaginary parts of $N(\mu,\sigma^2)_\C$ are independent copies of $N(\Re \mu, \sigma^2/2)_\R$ and $N(\Im \mu, \sigma^2/2)_\R$ respectively.  In a similar spirit, for any natural number, we use $\chi_{i,\R}$ to denote the real $\chi$ distribution with $i$ degrees of freedom, thus $\chi_{i,\R} \equiv \sqrt{\xi_1^2 + \dots + \xi_i^2}$ for independent copies $\xi_1,\dots,\xi_i$ of $N(0,1)_\R$.  Similarly, we use $\chi_{i,\C}$ to denote the complex $\chi$ distribution with $i$ degrees of freedom, thus $\chi_{i,\C} \equiv \sqrt{\xi_1^2 + \dots + \xi_i^2}$ for independent copies $\xi_1,\dots,\xi_i$ of $N(0,1)_\C$.  Again, the two distributions are closely related: one has $\chi_{i,\C} \equiv \frac{1}{\sqrt{2}} \chi_{2i,\R}$ for all $i$.

If $F: \C^k \to \C$ is a smooth function, we use $\nabla F(z_1,\ldots,z_k)$ to denote the $2k$-dimensional vector whose components are the partial derivatives $\frac{\partial F}{\partial \operatorname{Re} z_i}(z_1,\ldots,z_k)$, $\frac{\partial F}{\partial \operatorname{Im} z_i}(z_1,\ldots,z_k)$ for $i=1,\ldots,k$.  Iterating this, we can define $\nabla^a F(z_1,\ldots,z_k)$ for any natural number $a$ as a tensor with $(2k)^a$ coefficients, each of which is an $a$-fold partial derivative of $F$ at $z_1,\ldots,z_k$.  The magnitude $|\nabla^a F(z_1,\ldots,z_k)|$ is then defined as the $\ell^2$ norm of these coefficients.  Similarly for functions defined on $\R^k$ instead of $\C^k$.

\section{A concentration inequality}

In this section we recall a martingale type concentration inequality which will be useful in our arguments. Let $Y= Y(\xi_1, \dots, \xi_n)$ be a random variable depending on independent atom variables $\xi_i \in \C$.
For $1\le i \le n$ and $\xi =(\xi_1, \dots, \xi_n) \in \C^n$, define the martingale differences
\begin{equation}\label{mdif}
C_i( \xi) :=  | \E (Y|\xi_1, \dots, \xi_i) - \E (Y| \xi_1, \dots, \xi_{i-1} ) | .
\end{equation}

The classical Azuma's inequality (see e.g. \cite{alon}) states that if  $C_i \le \alpha_i$ with probability one, then

$$\P \left( |Y -\E Y| \ge \lambda \sqrt {\sum_{i=1}^n \alpha_i^2 } \right) = O(\exp(- \Omega (\lambda^2))). $$

In applications, the assumption that $C_i \le \alpha_i$ with probability one sometimes fails. However, we can overcome this using a trick from \cite{VuNL}. In particular,
the following is a simple variant of \cite[Lemma 3.1]{VuNL}.

\begin{proposition}\label{azuma0}  For any $\alpha_i \ge 0$ we have the inequality
$$\P \left( |Y -\E Y| \ge \lambda \sqrt {\sum_{i=1}^n \alpha_i^2 }\right) = O(\exp(- \Omega (\lambda^2))) + \sum_{i=1}^n \P( C_i(\xi) \ge \alpha_i).$$
\end{proposition}

\begin{proof}
For each $\xi$, let $i_{\xi}$ be the first index where $C_i(\xi) \ge \alpha_i$. Thus, the sets $B_i := \{\xi| i_{\xi} =i\}$ are disjoint.
Define a function $Y'(\xi)$ of $\xi$ which agrees with $Y(\xi)$ for $\xi$ in the complement of $\cup_i B_i$, with $Y'(\xi) := \E _{B_i} Y $ if $\xi \in B_i$.
It is clear that $Y'$ and $Y$ has the same mean and
$$\P (Y \neq Y') \le \sum_{i=1}^n \P( C_i(\xi) \ge \alpha_i). $$
Moreover, $Y'$ satisfies the condition of Azuma's inequality, so
$$\P \left(| Y' -\E Y'| \geq \lambda \sqrt {\sum_{i=1}^n \alpha_i^2 }\right) \ll \exp(- \Omega (\lambda^2)) $$ and the bound follows.
\end{proof}

We have the following useful corollary.

\begin{proposition}[Martingale concentration]\label{azuma}  Let $\xi_1,\dots,\xi_n$ be independent complex random variables of mean zero and
$|\xi_i| =n^{o(1)}$ with overwhelming probability for all $i$.
  Let $\alpha_1,\dots,\alpha_n > 0$ be positive real numbers, and for each $i = 1,\dots,n$, let $c_i(\xi_1,\dots,\xi_{i-1})$ be a complex random variable depending only on $\xi_1,\dots,\xi_{i-1}$ obeying the bound
$$ |c_i(\xi_1,\dots,\xi_{i-1})| \leq \alpha_i$$ with overwhelming probability.  Define $Y := \sum_{i=1}^n c_i(\xi_1,\dots,\xi_{i-1}) \xi_i$.
Then
$$ |Y| \ll n^{o(1)} (\sum_{i=1}^n \alpha_i^2)^{1/2}$$
with overwhelming probability.
\end{proposition}

\begin{proof} Let $C_i(\xi)$ be the martingale difference \eqref{mdif}.
It is easy to see that $C_i(\xi) = |c_i(\xi_1, \dots, \xi_{i-1} ) \xi_i |$. By the assumptions, $C_i (\xi) \le n^{o(1)} \alpha_i$ with overwhelming probability. Now apply Proposition \ref{azuma0} with a suitable choice of parameter $\lambda = n^{o(1)}$.
\end{proof}

\begin{remark} (Note added after publication.)  It has been pointed out to us by Heejune Sheen, and independently by Christian Borgs and Karissa Huang, that Proposition \ref{azuma0} is incorrect; the random variable $Y'$ constructed above does not, in fact, obey the hypotheses of Azuma's inequality.  A version of the proposition can be salvaged by replacing $C_i(\xi_1,\dots,\xi_i)$ with the larger quantity $\sup_\eta C_i(\xi_1,\dots,\xi_{i-1})$ (and redefining the events $B_i$ accordingly); however, for the purposes of establishing Proposition \ref{azuma} (which is the only place in this paper where Proposition \ref{azuma0} is used), it is better to proceed as follows. Firstly, one should impose an additional mild moment hypothesis such as $\E \xi_i^2 =O(n^{O(1)})$, which is true in all applications of interest. Then one can define $Y'$ to be the same as $Y$ but with each $\xi_i$ replaced by a mean zero modification $\xi'_i$ of size $O(n^{o(1)})$ almost surely that agrees with $\xi_i$ with overwhelming probability, and $c_i(\xi_1,\dots,\xi_{i-1})$ similarly replaced by $c'_i(\xi'_1,\dots,\xi'_{i-1})$ that is bounded by $\alpha_i$ almost surely and agrees with $c_i$ with overwhelming probability.
\end{remark}

\section{From log-determinant concentration to the local circular law}\label{circular-sec}

In this section we prove Theorem \ref{local-circ} using Theorem \ref{loglower}.
The first step is to deduce the crude bound \eqref{nbr-0} from Theorem \ref{loglower}.
  We first make some basic reductions.  By a covering argument and the union bound it suffices to establish the claim for $r=1$ and for a fixed $z_0 \in B(0,2C\sqrt{n})$.

The main tool will be \emph{Jensen's formula} \eqref{jensen-det}.
Applying this to the disk $B(z_0,2)$, we see in particular that
\begin{equation}\label{nous}
 N_{B(z_0,1)} \ll \frac{1}{2\pi} \int_0^{2\pi} (\log |\det(M_n-z_0-2e^{\sqrt{-1}\theta})| - \log |\det(M_n-z_0)|)\ d\theta.
 \end{equation}
Let $A \geq 1$ be an arbitrary fixed quantity.  In view of \eqref{nous}, it suffices to show that
$$ \frac{1}{2\pi} \int_0^{2\pi} (\log |\det(M_n-z_0-2e^{\sqrt{-1}\theta})| - \log |\det(M_n-z_0)|)\ d\theta = O(n^\eps)$$
with probability $1-O(n^{-A})$.

We will control this integral\footnote{One can also control this integral by a Riemann sum, using an argument similar to that used to prove Theorem \ref{local-circ} below.  On the other hand, we will use Lemma \ref{sampling} again in Section \ref{four}, and one can view the arguments below as a simplified warmup for the more complicated arguments in that section.} by a Monte Carlo sum, using the following standard sampling lemma:

\begin{lemma}[Monte Carlo sampling lemma]\label{sampling}  Let $(X,\mu)$ be a probability space, and let $F: X \to \C$ be a square-integrable function.  Let $m \geq 1$, let $x_1,\dots,x_m$ be drawn independently at random from $X$ with distribution $\mu$, and let $S$ be the empirical average
$$ S := \frac{1}{m} (F(x_1) + \dots + F(x_m)).$$
Then $S$ has mean $\int_X F\ d \mu$ and variance $\int_X (F-\int_X F\ d\mu)^2\ d\mu$.  In particular, by Chebyshev's inequality, one has
$$ \P( |S - \int_X F\ d \mu| \geq \lambda ) \leq \frac{1}{m\lambda^2} \int_X (F-\int_X F\ d\mu)^2\ d\mu$$
for any $\lambda > 0$, or equivalently, for any $\delta > 0$ one has with probability at least $1-\delta$ that
$$| S - \int_X F\ d \mu | \le \frac{1}{\sqrt{m\delta}} (\int_X (F-\int_X F\ d\mu)^2\ d\mu)^{1/2} .$$
\end{lemma}

\begin{proof} The random variables $F(x_i)$ for $i=1,\dots,m$ are jointly independent with mean $\int_X F\ d \mu$ and variance $\frac{1}{m} \int_X (F-\int_X F\ d\mu)^2\ d\mu$.  Averaging these variables, we obtain the claim.
\end{proof}

We apply this lemma to the probability space $X := [0,2\pi]$ with uniform measure $\frac{1}{2\pi}\ d\theta$, and to the function
$$ F(\theta) := \log |\det(M_n-z_0-2e^{\sqrt{-1}\theta})| - \log |\det(M_n-z_0)|.$$
Observe that for any complex number $z$, the function $\log |z - 2e^{\sqrt{-1}\theta}|$ has an $L^2(X)$ norm of $O(1)$.  Thus by the triangle inequality and \eqref{log-det-1} we have the crude bound
$$ \int_X (F-\int_X F\ d\mu)^2\ d\mu \ll n^2.$$ We set $\delta:= n^{-A}$ and $m:= n^{A+2}$.
Let $\theta_1,\dots,\theta_m$ be drawn independently uniformly at random from $X$ (and independently of $M_n$) and set $\Theta:= (\theta_1, \dots, \theta_m)$.  Let $\CE_1$ denote the event that the inequality
$$ | S - \int_X F\ d \mu | \le \frac{1}{\sqrt{m\delta}} (\int_X (F-\int_X F\ d\mu)^2\ d\mu)^{1/2}  $$
holds, and let $\CE_2$ denote the event that the inequality
$$
\Big| \log |\det(M-z_0-2e^{\sqrt{-1}\theta_j})| - \log |\det(M_n-z_0)| \Big|  \le n^\eps $$
holds for all $j=1,\dots,m$.  Call a pair $(M,\Theta)$ is {\it good} if $\CE_1$ and $\CE_2$ both hold.
It suffices to show that the probability that a pair $(M, \Theta)$ (with $M=M_n$) is good is $1- O(n^{-A})$.

By Lemma \ref{sampling}, for each fixed $M$, the probability that $\CE_1$ fails is at most $\delta =n^{-A}$. Moreover, by Theorem
\ref{loglower}, we see that for each fixed $\theta_i$, the probability that  $\Big| \log |\det(M-z_0-2e^{\sqrt{-1}\theta_j})| - \log |\det(M_n-z_0)| \Big|\le n^\eps$ fails is less
than $O(n^{-2A-2})$. Thus, by the union bound, the probability that $(M, \Theta)$ is not good (over the product space $M_n \times X^m$) is at most

$$n^{-A} + m \times O(n^{-2A -2}) =  O(n^{-A}), $$
concluding the proof of \eqref{nbr-0}.

Now we are ready to prove Theorem \ref{local-circ}. We assume $r\ge 10$ as the
claim follows trivially from Theorem \ref{loglower} otherwise.
Consider the circle $C_{z_0, r} := \{ z \in \C: |z-z_0|=r\}$. By the pigeonhole principle, there is some $0 \le j \le n$ such that
the $\frac{1}{n^3}$-neighborhood of the circle $C_j:= C_{z_0, r_j}$ with $r_j:= r- \frac{j}{n^2}$ contains no eigenvalues of $M_n$
(notice that these neighborhoods are disjoint).
If $j$ is such an index, we see from \eqref{log-det-1} that the function
$$F(\theta) := \log |\det  (M_n - z_0 - r_j e^{-\sqrt{-1} \theta}) | - \log |\det  (M_n - z_0) |$$
then has a Lipschitz norm of $O(n^{O(1)})$ on $[0,2\pi]$.  Setting $m := n^{A+2}$ for a sufficiently large constant $A$, we then see from quadrature that the Riemann sum $\frac{1}{m} \sum_{k=1}^m F(2\pi k/m)$ approximates the integral
$\frac{1}{2 \pi} \int_0^{2\pi} F( \theta ) d \theta $ within an additive error at most $n^{o(1)}$.  By  \eqref{jensen-det}, we conclude that
$$
\sum_{  |\lambda_i -z_0| < r_j } \log \frac{r_j}{|\lambda_i -z_0|} =  \frac{1}{m} \sum_{k=1}^m F(k/m) +O(n^{o(1)}).
$$
On the other hand, from Theorem \ref{loglower} (after applying rescaling by $\sqrt{n}$) and the union bound we see that with overwhelming probability, we have
$$ F(k/m) = G( z_0 +r_j  e^{\sqrt{-1} 2\pi k/m} ) - G(z_0) + O(n^{o(1)}) $$
for all $1 \leq k \leq m$, where $G(z)$ is defined as $\frac{1}{2} (|z|^2 - n)$ for $|z| \leq \sqrt{n}$, and $n \log \frac{|z|}{\sqrt{n}}$ for $|z| \geq \sqrt{n}$.  Applying quadrature again, we conclude (for $A$ large enough) that
$$
G(z_0) = - \sum_{  |\lambda_i -z_0| < r_j } \log \frac{r_j}{|\lambda_i -z_0|} + \frac{1}{2\pi} \int_0^{2\pi}
G( z_0 +r_j  e^{\sqrt{-1} \theta} )\ d\theta + O(n^{o(1)}).$$
A similar argument (replacing $r$ by $r-1$) shows that with overwhelming probability, there exists $0 \leq j' \leq n$ such that
$$
G(z_0) = - \sum_{  |\lambda_i -z_0| < r_{j'}-1 } \log \frac{r_{j'}-1}{|\lambda_i -z_0|} + \frac{1}{2\pi} \int_0^{2\pi}
G( z_0 + (r_{j'}-1)  e^{\sqrt{-1} \theta} )\ d\theta + O(n^{o(1)}).$$
Also, from \eqref{nbr-0} and a simple covering argument, we know that with overwhelming probability, there are at most $O(n^{o(1)} r)$ eigenvalues in the annular region between $C_{z_0,r_{j'}-1}$ and $C_{z_0,r}$, and in this region, the quantities $\log \frac{r_j}{|\lambda_i -z_0|}$ and $\log \frac{r_{j'}-1}{|\lambda_i -z_0|}$ have magnitude $O(1/r)$.  We may thus subtract the above two estimates and conclude that
\begin{equation}\label{soo}
\begin{split}
0 &= - N(z_0,r) \log \frac{r_j}{r'_j-1} +
\frac{1}{2\pi} \int_0^{2\pi} G( z_0 +r_j  e^{\sqrt{-1} \theta} )\ d\theta\\
&\quad - \frac{1}{2\pi} \int_0^{2\pi} G( z_0 + (r_{j'}-1)  e^{\sqrt{-1} \theta} )\ d\theta + O(n^{o(1)}).
\end{split}
\end{equation}
On the other hand, from applying Green's theorem\footnote{The function $G$ has a mild singularity on the circle $|z|=\sqrt{n}$, but one can verify that as the first derivatives of $G$ remain continuous across this circle, there is no difficulty in applying Green's theorem even when $B(z_0,r_j)$ crosses this circle.}
$$ \int_\Omega F(z) \Delta G(z) - \Delta G(z) F(z)\ dz = \int_{\partial \Omega} F(z) \frac{\partial}{\partial n} G(z) - \frac{\partial}{\partial n} F(z) G(z)$$
to the domain $\Omega := B(z_0,r_j) \backslash B(z_0,\eps)$ with $F(z) := \log \frac{r_j}{|z-z_0|}$, and then sending $\eps \to 0$, one sees that
$$ G(z_0) =
- \frac{1}{2\pi} \int_{B(z_0,r_j)} \Delta G(z) \log \frac{r_j}{|z-z_0|}\ dz
+ \frac{1}{2\pi} \int_0^{2\pi} G( z_0 +r_j  e^{\sqrt{-1} \theta} )\ d\theta,$$
where $\Delta$ is the usual Laplacian on $\C$; one easily computes that $\Delta G(z) = 2 1_{|z| \leq \sqrt{n}}$, and thus
$$ G(z_0) =
- \frac{1}{\pi} \int_{B(z_0,r_j)} 1_{|z| \leq \sqrt{n}} \log \frac{r_j}{|z-z_0|}\ dz
+ \frac{1}{2\pi} \int_0^{2\pi} G( z_0 +r_j  e^{\sqrt{-1} \theta} )\ d\theta.$$
Similarly one has
$$ G(z_0) =
- \frac{1}{\pi} \int_{B(z_0,r_{j'}-1)} 1_{|z| \leq \sqrt{n}} \log \frac{r_{j'}-1}{|z-z_0|}\ dz
+ \frac{1}{2\pi} \int_0^{2\pi} G( z_0 +(r_{j'}-1)  e^{\sqrt{-1} \theta} )\ d\theta.$$
Subtracting, and observing that the integrands $1_{|z| \leq \sqrt{n}} \log \frac{r_j}{|z-z_0|}$, $1_{|z| \leq \sqrt{n}} \log \frac{r_{j'}-1}{|z-z_0|}$ have magnitude $O(1/r)$ in the annular region between $C_{z_0,r_{j'}-1}$ and $C_{z_0,r}$, we conclude that
\begin{align*}
 0 &= - \int_{B(z_0,r)} \frac{1}{\pi} 1_{|z| \leq \sqrt{n}}\ dz \times \log \frac{r_j}{r'_j-1}+
\frac{1}{2\pi} \int_0^{2\pi} G( z_0 +r_j  e^{\sqrt{-1} \theta} )\ d\theta \\
&\quad - \frac{1}{2\pi} \int_0^{2\pi} G( z_0 + (r_{j'}-1)  e^{\sqrt{-1} \theta} )\ d\theta + O(n^{o(1)}).
\end{align*}
Comparing this with \eqref{soo}, we conclude with overwhelming probability that
$$ \left(N_{B(z_0,r)} - \int_{B(z_0,r)} \frac{1}{\pi} 1_{|z| \leq \sqrt{n}}\ dz\right) \times \log \frac{r_j}{r'_j-1} = O(n^{o(1)}).$$
Since $\log \frac{r_j}{r'_j-1}$ is comparable to $1/r$, we obtain \eqref{nbr} as desired.

\section{Reduction to the Four Moment Theorem and log-determinant concentration}\label{four}

We now begin the task of proving Theorem \ref{main-alt} and Theorem \ref{main-alt-2}, by reducing it the Four Moment Theorem for determinants (Theorem \ref{four-moment}) and the local circular law (Proposition \ref{local-circ}).  In the preceding section, of course, the local circular law has been reduced in turn to the concentration of the log-determinant (Theorem \ref{loglower}).

\subsection{The complex case}

We begin with Theorem \ref{main-alt}, deferring the slightly more complicated argument for Theorem \ref{main-alt-2} to the end of this section.

Let $M_n, \tilde M_n$ be as in Theorem \ref{main-alt}.   Call a statistic $S(M_n)$ of (the law of) a random matrix $M_n$ \emph{asymptotically $(M_n,\tilde M_n)$ insensitive}, or \emph{insensitive} for short, if we have
$$ S(M_n) - S(\tilde M_n) = O(n^{-c})$$
for some fixed $c>0$.  Our objective is then to show that the statistic
\begin{equation}\label{aleo}
 \int_{\C^k} F(w_1,\dots,w_k) \rho^{(k)}_n(\sqrt{n} z_1 + w_1,\dots,\sqrt{n} z_k + w_k)\ dw_1 \dots dw_k
\end{equation}
is insensitive for all fixed $k \geq 1$ and all $F$ of the form \eqref{factor} for some fixed $m \geq 1$.

Fix $k$; we may assume inductively that the claim has already been proven for all smaller $k$.  By linearity we may take $m=1$, thus we may assume that $F$ takes the tensor product form
\begin{equation}\label{tensor}
 F(w_1,\dots,w_k) = F_1(w_1) \dots F_k(w_k)
\end{equation}
for some smooth, compactly supported $F_1,\dots,F_k: \C \to \C$ supported on a fixed ball, with bounds on derivatives up to second order.

Henceforth we assume that $F$ is in tensor product form \eqref{tensor}.  By \eqref{ck} and the inclusion-exclusion formula, we may thus write \eqref{aleo} in this case as
\begin{equation}\label{xzf}
 \E \prod_{j=1}^k X_{z_j,F_j}
 \end{equation}
plus a fixed finite number of lower order terms that are of the form \eqref{aleo} for a smaller value of $k$ (and a different choice of $F_j$), where $X_{z_j,F_j}$ is the linear statistic
$$ X_{z_j,F_j} := \sum_{i=1}^n F_j(\lambda_i(M_n) - \sqrt{n} z_j).$$
By the induction hypothesis, it thus suffices to show that the expression \eqref{xzf} is insensitive.

Using the local circular law (Proposition \ref{local-circ}), we see that for any $1 \leq j \leq k$, one has $X_{z_j,F_j} = O(n^{o(1)})$ with overwhelming probability.  Thus, one can truncate the product function $\zeta_1,\dots,\zeta_k \mapsto \zeta_1 \dots \zeta_k$ and write
$$  \E \prod_{j=1}^k X_{z_j,F_j} =  \E G( X_{z_1,F_1}, \dots, X_{z_k,F_k} ) + O(n^{-B})$$
for any fixed $B$, where $G$ is a smooth truncation of the product function $\zeta_1,\dots,\zeta_k \mapsto \zeta_1 \dots \zeta_k$ to the region $\zeta_1,\dots,\zeta_k = n^{o(1)}$.   Thus, it suffices to show that the quantity
\begin{equation}\label{egg}
 \E G( X_{z_1,F_1}, \dots, X_{z_k,F_k} )
\end{equation}
is insensitive whenever $G: \C^k \to \C$ is a smooth function obeying the bounds
\begin{equation}\label{g-bound}
|\nabla^j G(\zeta_1,\dots,\zeta_k)| \leq n^{o(1)}
\end{equation}
for all fixed $j$ and all $\zeta_1,\dots,\zeta_k \in \C$.

Fix $G$. As is standard in the spectral theory of random non-Hermitian matrices (cf. \cite{girko}, \cite{brown}), we now express the linear statistics $X_{z_j,F_j}$ in terms of the log-determinant \eqref{log-det-1}.  By Green's theorem we have
\begin{equation}\label{zafj}
 X_{z_j,F_j} = \int_\C \log|\det(M_n-z)| H_j(z)\ dz
\end{equation}
where $H_j: \C \to \C$ is the function
$$ H_j(z) := -\frac{1}{2\pi} \Delta F_j(z - \sqrt{n} z_j),$$
and $\Delta$ is the Laplacian on $\C$. From the derivative and support bounds on $F_j$, we see that $H_j$ is supported on $B(\sqrt{n} z_j, C)$ and is bounded.

Naively, to control \eqref{zafj}, one would apply Lemma \ref{sampling} with the function $\log|\det(M_n-z)| H_j(z)$.  Unfortunately, the variance of this expression is too large, due to the contributions of the eigenvalues far away from $\sqrt{n} z_j$.  To cancel\footnote{It is natural to expect that these non-local contributions can be canceled, since the statistics $X_{z_i,F_i}$ are clearly local in nature.} off these contributions, we exploit the fact that $H_j(z)$, being the Laplacian of a smooth compactly supported function, is orthogonal to all harmonic functions, and in particular to all (real-)linear functions:
$$ \int_\C (a + b \Re(z) + c \Im(z)) H_j(z)\ dz = 0.$$
(Recall that we use $dz$ to denote Lebesgue measure on $\C$.)  We will need a reference element $w_{j,0}$ drawn uniformly at random from $B(\sqrt{n} z_j, 1)$ (independently of $M_n$ and the $w_{j,i}$), and let $L(z) = L_j(z)$ denote the random linear function which equals $\log|\det(M_n-z)|$ for $z = w_{j,0}, w_{j,0}+1, w_{j,0}+\sqrt{-1}$.  More explicitly, one has
\begin{equation}\label{ldef}
\begin{split}
L(z) &:= \log|\det(M_n-w_{j,0})| \\
&\quad + (\log|\det(M_n-w_{j,0}-1)| - \log|\det(M_n-w_{j,0})|) \Re( z - w_{j,0} ) \\
&\quad + (\log|\det(M_n-w_{j,0}-\sqrt{-1})| - \log|\det(M_n-w_{j,0})|) \Im( z - w_{j,0} ).
\end{split}
\end{equation}

\begin{remark} There is some freedom in how to select $L(z)$; for instance, it is arguably more natural to replace the coefficients
$\log|\det(M_n-w_{j,0}-1)| - \log|\det(M_n-w_{j,0})|$ and $\log|\det(M_n-w_{j,0}-\sqrt{-1})| - \log|\det(M_n-w_{j,0})|$ in the above formula by the Taylor coefficients $\frac{d}{dt} \log|\det(M_n-w_{j,0}-t)| |_{t=0}$ and $\frac{d}{dt} \log|\det(M_n-w_{j,0}-\sqrt{-1} t)| |_{t=0}$ instead.  However this would require extending the four moment theorem for log-determinants to derivatives of log-determinants, which can be done but will not be pursued here.
\end{remark}

Subtracting off $L(z)$, we have
\begin{equation}\label{zof}
 X_{z_j,F_j} = \int_\C K_j(z)\ dz
\end{equation}
where $K_j: \C \to \C$ is the random function
\begin{equation}\label{kj}
 K_j(z) := (\log |\det(M_n-z)| - L(z)) H_j(z).
\end{equation}
Let us control the $L^2$ norm
$$ \|K_j\|_{L^2} := \left(\int_\C |K_j(z)|^2\ dz\right)^{1/2}$$
of this quantity.

\begin{lemma}  For any $\eps>0$, one has
\begin{equation}\label{kujak}
 \|K_j\|_{L^2} \ll n^{\eps+o(1)}
\end{equation}
with probability $1-O(n^{-\eps})$ and all $1 \leq j \leq k$.
\end{lemma}

\begin{proof} By the union bound, it suffices to prove the claim for a single $k$.  We can split $K_j = \sum_{i=1}^n K_{j,i}(z)$, where
$$ K_{j,i}(z) := (\log |\lambda_i(M_n)-z| - L_i(z)) H_j(z)$$
and $L_i: \C \to \C$ is the random linear function that equals $\log |\lambda_i(M_n)-z|$ when $z = w_{j,0}, w_{j,0}+1, w_{j,0}+\sqrt{-1}$.  By the triangle inequality, we thus have
$$ \|K_j\|_{L^2} \leq \sum_{i=1}^n \|K_{j,i}\|_{L^2}.$$
Thanks to Proposition \ref{local-circ}, we know with overwhelming probability that one has
\begin{equation}\label{nosy}
 N_{B(z_j\sqrt{n},r)} \ll n^{o(1)} r^2
\end{equation}
for all $r$.  Let us condition on the event that this holds, and then freeze $M_n$ (so that the only remaining source of randomness is $w_{j,0}$).  In particular, the eigenvalues $\lambda_i(M_n)$ are now deterministic.

Let $C_0>1$ be such that $H_j$ is supported in $B(z_0\sqrt{n}, C_0)$.  If $1 \leq i \leq n$ is such that $\lambda_i(M_n) \in B(z_j\sqrt{n}, 2C_0)$, then a short computation (based on the square-integrability of the logarithm function) shows that the expected value of $\|K_{j,i}\|_{L^2}$ (averaged over all choices of $w_{j,0}$) is $O(1)$.  On the other hand, if $\lambda_i(M_n) \not \in B(z_j\sqrt{n}, 2C_0)$, then the second derivatives of $\log|\lambda_i(M_n)-z|$ has size $O(1 / |\lambda_i(M_n) - z_j \sqrt{n}|^2)$ on $B(z_j\sqrt{n}, 2C_0)$.  From this and Taylor expansion, one sees that the function $\log |\lambda_i(M_n)-z| - L_i(z)$ has magnitude $O(1 / |\lambda_i(M_n) - z_j \sqrt{n}|^2)$ on this ball, and so $\|K_{j,i}\|_{L^2}$ has this size as well.  Summing, we conclude that the (conditional) expected value of $\|K_j\|_{L^2}$ is at most
\begin{equation}\label{laa}
 \ll \sum_{i=1}^n \frac{1}{1+|\lambda_i(M_n) - z_j \sqrt{n}|^2}.
\end{equation}
We claim that the summation in \eqref{laa} has magnitude $O(n^{o(1)})$ with overwhelming probability, which will give the claim from Markov's inequality.  To see this, first observe that the eigenvalues $\lambda_i(M_n)$ with $|\lambda_i(M_n)-z_j \sqrt{n}| \geq \sqrt{n}$ certainly contribute at most $O(1)$ in total to the above sum.  Next, from \eqref{nosy} we see that with overwhelming probability that there are only $O(n^{o(1)})$ eigenvalues with $|\lambda_i(M_n)-z_j \sqrt{n}| \leq 1$, giving another contribution of $O(n^{o(1)})$ to the above sum.  Similarly, for any $2^k$ between $1$ and $\sqrt{n}$, another application of \eqref{nosy} reveals that the eigenvalues with $2^k \leq |\lambda_i(M_n)-z_j \sqrt{n}| < 2^{k+1}$ contribute another term of $O(n^{o(1)})$ to the above sum with overwhelming probability.  As there are only $O( \log \sqrt{n} ) = O(n^{o(1)})$ possible choices for $k$, the claim then follows by summing all the contributions estimated above.
\end{proof}

Now let $\eps > 0$ be a sufficiently small fixed constant that will be chosen later. Set $m := \lfloor n^{10\eps} \rfloor$, and for each $1 \leq j \leq k$ let $w_{j,1},\dots,w_{j,m}$ be drawn uniformly at random from $B(\sqrt{n} z_j, C_0)$ (independently of $M_n$ and $w_{j,0}$).  By \eqref{kujak}, \eqref{zof}, and Lemma \ref{sampling}, we see that with probability $1-O(n^{-\eps})$, one has
$$
X_{z_j,F_j} = \frac{\pi C_0^2}{m} \sum_{i=1}^m K_j( w_{j,i} ) + O( n^{-3\eps + o(1)} ).$$
In particular, from \eqref{g-bound} we see that with probability $1-O(n^{-\eps})$, one has
$$
G( X_{z_1,F_1}, \dots, X_{z_k,F_k} ) = G\left( \left(\frac{\pi C_0^2}{m} \sum_{i=1}^m K_j( w_{j,i} )\right)_{1 \leq j \leq k} \right) + O( n^{-3\eps + o(1)} )$$
and hence
$$
\E
G( X_{z_1,F_1}, \dots, X_{z_k,F_k} ) = \E G\left( \left(\frac{\pi C_0^2}{m} \sum_{i=1}^m K_j( w_{j,i} )\right)_{1 \leq j \leq k} \right) + O( n^{-\eps + o(1)} ).$$
Thus, to show that \eqref{egg} is insensitive, it suffices to show that
$$\E G\left( \left(\frac{\pi C_0^2}{m} \sum_{i=1}^m K_j( w_{j,i} )\right)_{1 \leq j \leq k} \right) $$
is insensitive, uniformly for all deterministic choices of $w_{j,0} \in B(\sqrt{n} z_j, 1)$ and $w_{j,i} \in B(\sqrt{n} z_j, C_0)$ for $1 \leq j \leq k$ and $1 \leq i \leq m$.  But this follows from the Four Moment Theorem (Theorem \ref{four-moment}), if $\eps$ is small enough; indeed, once the $w_{j,0}, w_{j,i}$ are conditioned to be deterministic, we see from \eqref{kj}, \eqref{ldef} that the quantities $K_j(w_{j,i})$ can be expressed
as deterministic linear combinations of a bouned number of log-determinants $\log|\det(M_n-z)|$, with coefficients uniformly bounded in $n$ (recall that $w_{j,i} - w_{j,0} = O(C_0)$ and that the $H_j$ are uniformly bounded).  This concludes the derivation of Theorem \ref{main-alt} from Theorem \ref{four-moment} and Proposition \ref{local-circ}.

\subsection{The real case}\label{four-real}

We now turn to the proof of Theorem \ref{main-alt-2}.
Let $M_n$ be as in Theorem \ref{main-alt-2}, and let $\tilde M_n$ be a real gaussian matrix.  Our task is to show that that the quantity
\begin{equation}\label{insens}
\begin{split}
&\int_{\R^k} \int_{\C^m} F(y_1,\dots,y_k,w_1,\dots,w_l) \\
&\quad  \rho^{(k,l)}_n(\sqrt{n} x_1 + y_1,\dots,\sqrt{n} x_k + y_k,
 \sqrt{n} z_1 + w_1,\dots,\sqrt{n} z_l + w_l )\\
 &\quad \ dw_1 \dots dw_l dy_1 \dots dy_k
 \end{split}
 \end{equation}
is insensitive whenever $k, l \geq 0$ are fixed, $x_1,\ldots,x_k \in \R$ and $z_1,\ldots,z_l \in \C$ are bounded, and $F$ decomposes as in Theorem \ref{main-alt-2}.

By induction on $k+l$, much as in the complex case, and separating the spectrum into contributions from $\R, \C_+, \C_-$, it thus suffices to show that the quantity
\begin{equation}\label{egf}
 \E
(\prod_{i=1}^k X_{x_i,F_i,\R})
(\prod_{j=1}^l X_{z_j,G_j,\C_+})
(\prod_{j'=1}^{l'} X_{z'_{j'},G'_{j'},\C_-})
\end{equation}
is insensitive, where $k,l,l'$ are fixed, $x_1,\ldots,x_k \in \R$ and $z_1,\ldots,z_l,z'_1,\ldots,z'_{l'} \in \C$ are bounded,
$$ X_{x,F,\R} := \sum_{1 \leq i \leq n: \lambda_i(M_n) \in \R} F(\lambda_i(M_n)- \sqrt{n} x)$$
and
$$ X_{z,G,\C_\pm} := \sum_{1 \leq i \leq n: \lambda_i(M_n) \in \C_\pm} G(\lambda_i(M_n)- \sqrt{n} z),$$
and the $F_i: \R \to \C$, $G_j: \C \to \C$, $G'_{j'}: \C \to \C$ are smooth functions supported on bounded sets obeying the bounds
$$ |\nabla^a F_i(x)|, |\nabla^a G_j(z)|, |\nabla^a G'_{j'}(z)| \leq C$$
for all $0 \leq a \leq 5$, $x \in \R$, $z \in \C$.  Indeed, one can express any statistic of the form \eqref{insens} as a linear combination of a bounded number of statistics of the form \eqref{egf}, plus a bounded number of additional statistics of the form \eqref{insens} with smaller values of $k+l$.

As the spectrum is symmetric around the real axis, one has
$$ X_{z,G,\C_-} = X_{\overline{z}, \tilde G, \C_+}$$
where $\tilde G(z) := G(\overline{z})$.  Thus we may concatenate the $G_j$ with the $G'_{j'}$, and assume without loss of generality that $l'=0$, thus we are now seeking to establish the insensitivity of
\begin{equation}\label{egf-2}
 \E
(\prod_{i=1}^k X_{x_i,F_i,\R})
(\prod_{j=1}^l X_{z_j,G_j,\C_+}).
\end{equation}

On the other hand, by repeating the remainder of the arguments for the complex case with essentially no changes, we can show that the quantity
\begin{equation}\label{goo}
 \E \prod_{p=1}^m X_{z_p, H_p}
\end{equation}
is insensitive for any fixed $m$, any bounded complex numbers $z_1,\dots,z_m$, and any smooth $H_p: \C \to \C$ supported in a bounded set and obeying the bounds
$$ |\nabla^a H_p(z)| \leq C$$
for all $0 \leq a \leq 5$ and $z \in \C$, where
$$ X_{z,H} := \sum_{1 \leq i \leq n} H(\lambda_i(M_n)- z).$$
Thus the remaining task is to deduce the insensitivity of \eqref{egf-2} from the insensitivity of \eqref{goo}.

Specialising \eqref{goo} to the case when $z_p = z$ is independent of $p$, and $H_p=H$ is real-valued, we see that
$$ \E X_{z,H}^m$$
is insensitive for any $m$.  In particular, we see from (the smooth version of) Urysohn's lemma and Lemma \ref{corf} that we have the bound
\begin{equation}\label{nibble}
 \E N_{B(z\sqrt{n},C)}^m \ll 1
\end{equation}
for any fixed radius $C$ and any bounded complex number $z$, where $N_\Omega = N_\Omega[M_n]$ denotes the number of eigenvalues of $M_n$ in $\Omega$.  Among other things, this implies that
\begin{equation}\label{holla}
 \E |X_{x_i,F_i,\R}|^A, \E |X_{y_j,G_j,\C_+}|^A \ll 1
\end{equation}
for any fixed $A$ and all $i,j$.

To proceed further, we need a level repulsion result.

\begin{lemma}[Weak level repulsion]\label{wlr}  Let $C>0$ be fixed, $x \in \R$ be bounded, and $\eps$ be such that $n^{-c_0} \leq \eps \leq C$ for a sufficiently small fixed $c_0>0$, and let $E_{x,C,\eps}$ be the event that there are two eigenvalues $\lambda_i(M_n), \lambda_j(M_n)$ in the strip $S_{x,C,\eps} := \{ z \in B(x \sqrt{n},C): \Im(z) \leq \eps \}$ with $i \neq j$ such that $|\lambda_i(M_n) - \lambda_j(M_n)| \leq 2\eps$.  Then $\P(E_{x,C,\eps}) \ll \eps$, where the implied constant in the $\ll$ notation is independent of $\eps$.
\end{lemma}

\begin{proof}  In this proof all implied constants in the $\ll$ notation are understood to be independent of $\eps$.  By a covering argument, it suffices to show that
$$ \P( N_{B(x\sqrt{n}+t,10\eps)} \geq 2 ) \ll \eps^2$$
uniformly for all $t = O(1)$.

If we let $H$ be a non-negative bump function supported on $B(t,20\eps)$ that equals one on $B(t,10\eps)$.  Then the expression $X_{x,H}^2 - X_{x,H^2}$ is non-negative, and is at least $2$ when $N_{B(x\sqrt{n}+t,10\eps)} \geq 2$.  Thus by Markov's inequality it suffices to show that
$$ \E X_{x,H}^2 - X_{x,H^2} \ll \eps^2.$$

By the insensitivity of \eqref{goo} and the lower bound on $\eps$, it suffices to verify the claim when $M_n$ is drawn from the real gaussian distribution.  (Note that the derivatives of $H, H^2$ can be as large as $O(\eps^{-O(1)})$, causing additional factors of $O(\eps^{-O(1)})$ to appear in the error term created when swapping $M_n$ with the real gaussian ensemble, but the $n^{-c}$ gain coming from the insensitivity will counteract this if $c_0$ is small enough.)

We split
$$ X_{x,H} = X_{x,H,\R} + 2 X_{x,H,\C_+}$$
and similarly for $H^2$. It will suffice to establish the estimates
\begin{equation}\label{r1}
\E X_{x,H,\R}^2 - X_{x,H^2,\R} \ll \eps^2,
\end{equation}
\begin{equation}\label{r3}
\E X_{x,H,\R} X_{x,H,\C_+} \ll \eps^2,
\end{equation}
and
\begin{equation}\label{r2}
\E X_{x,H,\C_+}^2 \ll \eps^2.
\end{equation}

The left-hand sides of \eqref{r1}, \eqref{r3}, \eqref{r2} may be expanded as
$$
\int_\R \int_\R \rho^{(2,0)}_n(x\sqrt{n}+y, x\sqrt{n}+y') H(y) H(y')\ dy dy',
$$
$$
\int_\R \int_{\C^+} \rho^{(1,1)}_n( x \sqrt{n}+y, x \sqrt{n}+z) H(y) H(z)\ dy dz,
$$
and
$$ \int_{\C_+} \rho^{(0,1)}_n(x\sqrt{n}+z) H^2(z)\ dz
+
2 \int_{\C_+} \int_{\C^+} \rho^{(0,2)}_n(x\sqrt{n}+z, x \sqrt{n} + w) H(z) H(w)\ dz dw
$$
respectively. Using Lemma \ref{corf}, we see that these expressions are $O(\eps^2)$ as required.
\end{proof}

\begin{remark} In fact, a closer inspection of the explicit form of the correlation functions reveals that one can gain some additional powers of $\eps$ here, giving a stronger amount of level repulsion, but for our purposes any bound that goes to zero as $\eps\to 0$ will suffice.
\end{remark}

From the symmetry of the spectrum, we observe that if $E_{x,C,\eps}$ does not hold, then there cannot be any strictly complex eigenvalue $\lambda_i(M_n)$ in the strip $S_{x,C,\eps}$, since in that case $\overline{\lambda_i(M_n)}$ would be distinct eigenvalue in the strip at a distance at most $2\eps$ from $\lambda_i(M_n)$.  In particular, we see that
\begin{equation}\label{sala}
 \P( N_{S_{x,C,\eps} \backslash [x\sqrt{n}-C,x\sqrt{n}+C]} = 0 ) = 1 - O(\eps).
\end{equation}
Informally, this estimate tells us that we can usually thicken the interval $[x\sqrt{n}-C,x\sqrt{n}+C]$ to the strip $S_{x,C,\eps}$ without encountering any additional spectrum.

Fix $\eps := n^{-c_0}$ for some sufficiently small fixed $c_0>0$.
We can use \eqref{sala} to simplify the expression \eqref{egf-2} in two ways.  Firstly, thanks to \eqref{sala}, \eqref{holla}, and H\"older's inequality, we may replace each of the $G_j$ in \eqref{egf} with a function $\tilde G_j$ that vanishes on the strip $\{ z-z_j: |\Im(z)| \leq \eps \}$, while only picking up an error of $O(\eps^c)$ for some fixed $c>0$, which will be acceptable from the choice of $\eps$.  By discarding the component of $\tilde G_j$ below the strip, we may then assume $\tilde G_j$ is supported on the half-space $\C_+ - z_j$.  In particular, we have
$$ X_{z_j,\tilde G_j,\C_+} = X_{z_j, \tilde G_j}.$$
Also, by performing a smooth truncation, we see that we have the derivative bounds $\nabla^a \tilde G_j = O(\eps^{-O(1)})$ for all $0 \leq a \leq 5$.

Secondly, by another application of \eqref{sala}, \eqref{holla}, and H\"older's inequality, we may ``thicken'' each factor $X_{x_i,F_i,\R}$ by replacing it with $X_{x_i,\tilde F_i}$, where $\tilde F_{i}: \C \to \C$ is a smooth extension of $F_i$ that is supported on the strip $\{ z: |\Im(z)| \leq \eps \}$, while only acquiring an error of $O(\eps^c)$ for some fixed $c>0$.  Again, we have the derivative bounds $\nabla^a \tilde F_i = O(\eps^{-O(1)})$ for $0 \leq a \leq 5$.  From the insensitivity of \eqref{goo} (and using the $n^{-c}$ gain coming from insensitivity to absorb all $O(\eps^{-O(1)})$ losses from the derivative bounds) we see that
\begin{equation}\label{egf-3}
 \E
(\prod_{i=1}^k X_{x_i,\tilde F_i})
(\prod_{j=1}^l X_{z_j,\tilde G_j})
\end{equation}
is insensitive, which by the preceding discussion yields (for $c_0$ small enough) that \eqref{egf-2} is insensitive also, as required.
This concludes the derivation of Theorem \ref{main-alt-2} from Theorem \ref{four-moment} and Proposition \ref{local-circ}.

\subsection{Quick applications}\label{real-sec}

As quick consequences of Theorem \ref{main-alt} and Theorem \ref{main-alt-2}, we now prove Corollaries \ref{clt-2}, \ref{really-2} and \ref{simex}.

We first prove we prove Corollary \ref{simex}.  Let $M_n$ be as in that theorem.  Set $\eps := n^{-c_0}$ for some sufficiently small $c_0>0$.  A routine modification of the proof of Lemma \ref{wlr} (or, alternatively, Theorem \ref{main-alt-2} combined with Lemma \ref{corf}) shows that for any $z \in B(0,O(\sqrt{n}))$, one has
$$ \E \binom{N_{B(z,\eps)}}{2} \ll \eps^4$$
when $|\Im z| \geq \eps$, if $c_0$ is small enough; in particular, the expected number of eigenvalues in $B(z,\eps)$ which are repeated is $O(\eps^4)$.  We then cover $B(0,3\sqrt{n})$ by $O(n/\eps^2)$ balls $B(z,\eps)$ with $|\Im z| \geq \eps$, together with the strip $\{ z: |\Im z| \leq \eps \}$.  By \eqref{sala} (or Theorem \ref{main-alt-2} and Lemma \ref{corf}) and linearity of expectation, the strip contains $O(\eps \sqrt{n})$ eigenvalues.  By \cite{baiyin}, \cite{geman}, the spectral radius of $M_n$ is known to equal $(1+o(1))\sqrt{n}$ with overwhelming probability\footnote{Actually, for this argument, the easier bound of $O(1)$ would suffice, which can be obtained by a variety of methods, e.g. by an epsilon net argument or by Talagrand's inequality \cite{Tal}.}.  We conclude that the expected number of repeated complex eigenvalues is at most
$$ O( n/\eps^2) \times O(\eps^4) + O(\eps \sqrt{n} ) + O(n^{-100}),$$
which becomes $O(n^{1-c})$ for some fixed $c>0$; a similar argument gives a bound of $O(n^{1/2-c})$ for the expected number of repeated real eigenvalues.  The claim now follows from Markov's inequality.

Now we prove Corollary \ref{really-2}.  Let $M_n$ be as in that theorem.  As mentioned previously, the spectral radius of $M_n$ is known to equal $(1+o(1))\sqrt{n}$ with overwhelming probability.  In particular, we have
$$ \E N_\R(M_n) = \E N_{[-3 \sqrt{n},3\sqrt{n}]}(M_n) + O(n^{-100})$$
(say).  By the smooth form of Urysohn's lemma, we can select fixed smooth, non-negative functions $F_-, F_+$ such that we have the pointwise bounds
$$ 1_{[-2,2]} \leq F_- \leq 1_{[-3,3]} \leq F_+ \leq 1_{[-4,4]}.$$
By definition of $\rho^{(1,0)}$, we observe that
\begin{align*}
\E N_{[-2\sqrt{n},2\sqrt{n}]}(M_n) &\leq \int_\R \rho^{(1,0)}(x) F_-(x/\sqrt{n})\ dx \\
&\leq  \E N_{[-3\sqrt{n},3\sqrt{n}]}(M_n) \\
&\leq \int_\R \rho^{(1,0)}(x) F_+(x/\sqrt{n})\ dx \\
&\leq \E N_{[-4\sqrt{n},4\sqrt{n}]}(M_n).
\end{align*}
By smoothly partitioning $F_\pm(x/\sqrt{n})$ into $O(\sqrt{n})$ pieces supported on intervals of size $O(1)$, and applying Theorem \ref{main-alt-2} to each piece, we see upon summing that the two integrals above are only modified by $O(n^{1/2-c})$ for some fixed $c>0$ if we replace $M_n$ with a real gaussian matrix $M'_n$.  On the other hand, when $M'_n$ is real gaussian we see from Theorem \ref{really} (and the spectral radius bound) that
$$ \E N_{[-2\sqrt{n},2\sqrt{n}]}(M'_n), \E N_{[-4\sqrt{n},4\sqrt{n}]}(M'_n) = \sqrt{\frac{2n}{\pi}} + O(1).$$
Putting these bounds together, we obtain the expectation claim of Corollary \ref{really-2}.  The variance claim is similar.  Indeed, we have
$$ \E N_\R(M_n)^2 = \E N_{[-3 \sqrt{n},3\sqrt{n}]}(M_n)^2 + O(n^{-90})$$
(say) and
\begin{align*}
\E N_{[-2\sqrt{n},2\sqrt{n}]}(M_n)^2 &\leq \int_\R \rho^{(1,0)}(x) F_-(x/\sqrt{n})^2\ dx + \int_\R \int_\R \rho^{(2,0)}(x,y) F_-(x/\sqrt{n}) F_-(y/\sqrt{n})\ dx dy \\
&\leq  \E N_{[-3\sqrt{n},3\sqrt{n}]}(M_n)^2 \\
&\leq \int_\R \rho^{(1,0)}(x) F_+(x/\sqrt{n})^2\ dx + \int_\R \int_\R \rho^{(2,0)}(x,y) F_+(x/\sqrt{n}) F_+(y/\sqrt{n})\ dx dy \\
&\leq \E N_{[-4\sqrt{n},4\sqrt{n}]}(M_n)^2.
\end{align*}
From Theorem \ref{main-alt-2} and smooth decomposition we see that all of the above integrals vary by $O(n^{1-c})$ at most for some fixed $c>0$ if $M_n$ is replaced with a real gaussian matrix, and then the variance claim can be deduced from Theorem \ref{really} and the spectral radius bound as before.

\begin{remark} A similar argument shows that in the complex case, the expected number of real eigenvalues is $O(n^{1/2-c})$, which can be improved to $O(n^{-A})$ for any $A>0$ if one assumes sufficiently many matching moments depending on $A$.  Of course, one expects typically in this case that there are no real eigenvalues whatsoever (and this is almost surely the case when the matrix ensemble is continuous), but this is beyond the ability of our current methods to establish in the case of discrete complex matrices.
\end{remark}

Finally, we prove Corollary \ref{clt-2}.  Let $M_n,z_0,r$ be as in that theorem, and let $\tilde M_n$ be drawn from the complex gaussian matrix ensemble.  Let $\eps = o(1)$ be a slowly decaying function of $n$ to be chosen later.  Let $R$ be any rectangle in $B(0,100\sqrt{n})$ of sidelength $1 \times n^{-\eps}$, and let $3R$ be the rectangle with the same center as $R$ but three times the sidelengths.  By the smooth form of Urysohn's lemma, we can construct a smooth function $F: \C \to \R^+$ with the pointwise bounds
$$ 1_R \leq F \leq 1_{3R}$$
such that $|\nabla^j F| \ll n^{j\eps}$ for all $0 \leq j \leq 5$.  Applying Corollary \ref{main-2} (to $n^{-5\eps} F$), we conclude that
$$ \int_\C F(z) \rho^{(1)}_n(z)\ dz = \int_\C F(z) \tilde \rho^{(1)}_n(z)\ dz + O(n^{-c+5\eps})$$
for some absolute constant $c$.  On the other hand, from \eqref{knzwp} we see that $\int_\C F(z) \tilde \rho^{(1)}_n(z)\ dz \ll n^{-\eps}$, since $3R$ has area $O(n^{-\eps})$. Since $\eps=o(1)$, we conclude that
$$ \int_\C F(z) \rho^{(1)}_n(z)\ dz \ll n^{-\eps}$$
and in particular that
\begin{equation}\label{enerman}
 \E N_R(M_n) \ll n^{-\eps}.
\end{equation}
A similar argument (with larger values of $k$) gives
\begin{equation}\label{enerman-2}
 \E N_{R_1}(M_n) \ldots N_{R_k}(M_n) \ll n^{-k\eps}.
\end{equation}
whenever $k$ is fixed and $R_1,\ldots,R_k$ are $1 \times n^{-\eps}$ rectangles (possibly overlapping) in $B(0,100\sqrt{n})$.

Now let $G: \C \to \R^+$ be a smooth function supported on $B(z_0,r+n^{-\eps})$ which equals $1$ on $B(z_0,r)$ and has the derivative bounds $|\nabla^j G| \ll n^{j\eps}$ for all $0 \leq j \leq 5$.  By covering the annulus $B(z_0,r+n^{-\eps}) \backslash B(z_0,r)$ by $O(r)$ rectangles of dimension $1 \times n^{-\eps}$, we see from \eqref{enerman} that
$$ \E N_{B(z_0,r+n^{-\eps}) \backslash B(z_0,r)}(M_n) \ll r n^{-\eps}$$
and similarly from \eqref{enerman-2} one has
$$ \E N_{B(z_0,r+n^{-\eps}) \backslash B(z_0,r)}(M_n)^k \ll r^k n^{-k\eps}$$
for any fixed $k$.  Since we are assuming $r \leq n^{o(1)}$, we conclude (if $\eps$ decays to zero sufficiently slowly) that
$$ \E N_{B(z_0,r+n^{-\eps}) \backslash B(z_0,r)}(M_n)^k = o(1)$$
for all $k$.  In particular, if we introduce the linear statistic
\begin{equation}\label{xo}
 X := \frac{\sum_{i=1}^n G(\lambda_i(M_n)) - r^2}{r^{1/2} \pi^{-1/4}}
 \end{equation}
we see from the triangle inequality that the asymptotics
$$ \E (\frac{N_{B(z_0,r)} - r^2}{r^{1/2} \pi^{-1/4}})^k \to \E N(0,1)_\R^k$$
for all fixed $k \geq 0$ are equivalent to the asymptotics
$$ \E X^k \to \E N(0,1)_\R^k.$$
Let $\tilde X$ be the analogue of $X$ for $\tilde M_n$.  From Theorem \ref{clt-gauss} and the preceding arguments we have
$$ \E \tilde X^k \to \E N(0,1)_\R^k$$
and so it will suffice to show that
$$ \E X^k - \E \tilde X^k = o(1)$$
for all fixed $k \geq 1$.
By \eqref{xo} and the hypotheses that $1 \leq r \leq n^{o(1)}$ and $\eps=o(1)$, it will suffice to show that
$$ \E (\sum_{i=1}^n G(\lambda_i(M_n)))^k - \E (\sum_{i=1}^n G(\lambda_i(\tilde M_n)))^k = O( r^{O(k)} n^{-c+O(k\eps)} )$$
for all fixed $k \geq 0$ and some fixed $c>0$ (which will in fact turn out to be uniform in $k$, although we will not need this fact).  Expanding out the $k^{\operatorname{th}}$ powers and collecting terms\footnote{The observant reader will note that this step is inverting one of the first steps in the proof of Theorem \ref{main-alt} given previously, and one could shorten the total length of the argument here if desired by skipping directly to that point of the proof of Theorem \ref{main-alt} and continuing onwards from there.} depending on the multiplicities of the $i$ indices, we see that it suffices to show that
\begin{align*}
& \E \sum_{1 \leq i_1 < \ldots < i_{k'} \leq n} G^{a_1}(\lambda_{i_1}(M_n)) \ldots G^{a_{k'}}(\lambda_{i_{k'}}(M_n)) -
G^{a_1}(\lambda_{i_1}(\tilde M_n)) \ldots G^{a_{k'}}(\lambda_{i_{k'}}(\tilde M_n))\\
&\quad = O( r^{O(k)} n^{-c+O(k\eps)} )
\end{align*}
for all fixed $k', a_1,\ldots,a_{k'} \geq 1$ and some fixed $c>0$, where $k := a_1 + \ldots + a_{k'}$.  But the left-hand side can be rewritten using \eqref{ck} as
$$ \int_{\C^k} (\prod_{j=1}^k G(z_j)^{a_j}) (\rho^{(k)}_n(z_1,\ldots,z_k) - \tilde \rho^{(k)}_n(z_1,\ldots,z_k))\ dz_1 \ldots dz_k.$$
One can smoothly decompose $(\prod_{j=1}^k G(z_j)^{a_j})$ as the sum of $O( r^{O(k)} n^{O(\eps)} )$ smooth functions supported on balls of bounded radius, whose derivatives up to fifth order are all uniformly bounded.  Applying Theorem \ref{main-alt} to each such function and summing, one obtains the claim.

\begin{remark}  The main reason why the radius $r$ was restricted to be $O(n^{o(1)})$ was because of the need to obtain asymptotics for $k^{\operatorname{th}}$ moments for arbitrary fixed $k$.  For any given $k$, the above arguments show that one obtains the right asymptotics for all $r \leq n^{c/k}$ for some absolute constant $c>0$.  If one increases the number of matching moment assumptions, one can increase the value of $k$, but we were unable to find an argument that allowed one to take $r$ as large as $n^\alpha$ for some fixed $\alpha>0$ \emph{independent} of $k$, even after assuming a large number of matching moments.
\end{remark}

\section{Resolvent swapping}\label{resolvent-sec}

In this section we recall some facts about the stability of the resolvent of Hermitian matrices with respect to permutation in just one or two entries, in order to perform swapping arguments.  Such swapping arguments were introduced to random matrix theory in \cite{Chat}, and first applied to establish universality results for local spectral statistics in \cite{TVlocal1}.  In \cite{ESYY} it was observed that the stability analysis of such swapping was particularly simple if one worked with the resolvents (or Greens function) rather than with individual eigenvalues.  Our formalisation of this analysis here is drawn from \cite{TV-determinant}.  We will use this resolvent swapping analysis twice in this paper; once to establish the Four Moment Theorem for the determinant (Theorem \ref{four-moment}) in Section \ref{4mt-sec}, and once to deduce concentration of the log-determinant for iid matrices (Theorem \ref{loglower}) from concentration for gaussian matrices (Theorem \ref{loglower-gaussian}) in Section \ref{lower-sec2}.

We will need the matrix norm
$$\|A\|_{(\infty,1)} = \sup_{1 \leq i,j \leq n} |a_{ij}|$$
and the following definition:

\begin{definition}[Elementary matrix]  An \emph{elementary matrix} is a matrix which has one of the following forms
\begin{equation}\label{vform}
 V = e_a e_a^*, e_a e_b^* + e_b e_a^*, \sqrt{-1} e_a e_b^* - \sqrt{-1} e_b e_a^*
\end{equation}
with $1 \leq a,b \leq n$ distinct, where $e_1,\dots,e_n$ is the standard basis of $\C^n$.
\end{definition}

Let $M_0$ be a Hermitian matrix, let $z=E+i\eta$ be a complex number, and let $V$ be an elementary matrix.  We then introduce, for each $t \in \R$, the Hermitian matrices
$$ M_t := M_0 + \frac{1}{\sqrt{n}} tV,$$
the resolvents
\begin{equation}\label{resolve}
R_t = R_t(E+i\eta) := (M_t - E - i\eta)^{-1}
\end{equation}
and the Stieltjes transform
$$ s_t := s_t(E+i\eta) := \frac{1}{n} \tr R_t(E + i\eta).$$

We have the following Neumann series expansion:

\begin{lemma}[Neumann series]\label{neum}  Let $M_0$ be a Hermitian $n \times n$ matrix, let $E \in \R$, $\eta > 0$, and $t \in \R$, and let $V$ be an elementary matrix.  Suppose one has
\begin{equation}\label{tro}
 |t| \| R_0\|_{(\infty,1)} = o(\sqrt{n}).
\end{equation}
Then one has the Neumann series formula
\begin{equation}\label{neumann}
 R_t = R_0 + \sum_{j=1}^\infty (-\frac{t}{\sqrt{n}})^j (R_0 V)^j R_0
\end{equation}
with the right-hand side being absolutely convergent, where $R_t$ is defined by \eqref{resolve}.  Furthermore, we have
\begin{equation}\label{tp}
 \|R_t\|_{(\infty,1)} \leq (1+o(1)) \|R_0\|_{(\infty,1)}.
\end{equation}
\end{lemma}

In practice, we will have $t = n^{O(c_0)}$ (from a decay hypothesis on the atom distribution) and $\|R_0\|_{(\infty,1)} = n^{O(c_0)}$ (from eigenvector delocalization and a level repulsion hypothesis), where $c_0>0$ is a small constant, so \eqref{tro} is quite a mild condition.

\begin{proof}   See \cite[Lemma 12]{TV-determinant}.
\end{proof}

We now can describe the dependence of $s_t$ on $t$:

\begin{proposition}[Taylor expansion of $s_t$]\label{proper}  Let the notation be as above, and suppose that \eqref{tro} holds.  Let $k \geq 1$ be fixed.   Then one has
\begin{equation}\label{expand}
 s_t = s_0 + \sum_{j=1}^k n^{-j/2} c_j t^j + O( n^{-(k+1)/2} |t|^{k+1} \|R_0\|_{(\infty,1)}^{k+1} \min( \|R_0\|_{(\infty,1)}, \frac{1}{n\eta} ) )
 \end{equation}
where the coefficients $c_j$ are independent of $t$ and obey the bounds
\begin{equation}\label{cj-bound}
|c_j| \ll \|R_0\|_{(\infty,1)}^{j} \min( \|R_0\|_{(\infty,1)}, \frac{1}{n\eta} ).
\end{equation}
for all $1 \leq j \leq k$.
\end{proposition}

\begin{proof} See \cite[Proposition 13]{TV-determinant}.
\end{proof}

\section{Proof of the Four Moment Theorem}\label{4mt-sec}

We now prove Theorem \ref{four-moment}.

We begin with some simple reductions.  Observe that each entry $\xi_{ij}$ of $M_n$ has size at most $O(n^{o(1)})$ with overwhelming probability.  Thus, by modifying the distributions of the $\xi_{ij}$ slightly (taking care to retain the moment matching property\footnote{Alternatively, one can allow the moments to deviate from each other by, say, $O(n^{-100})$, which one can verify will not affect the argument. See \cite[Chapter 2]{bai} or \cite[Appendix A]{nguyen} for details.}) and assume that all entries \emph{surely} have size $O(n^{o(1)})$.  Thus
\begin{equation}\label{manx}
\|M_n\|_{(\infty,1)}, \|M'_n\|_{(\infty,1)} \ll n^{o(1)}.
\end{equation}

We may also assume that $G$ is bounded by $1$ rather than by $n^{c_0}$, since the general claim then follows by normalising $G$ and shrinking $c_0$ as necessary; thus
\begin{equation}\label{g-bound-2}
|G(x_1,\dots,x_k)| \leq 1
\end{equation}
for all $x_1,\dots,x_k \in \R$.

Fix $M_n,M'_n$.  Recall that a statistic $S$ is \emph{asymptotically $(M_n,M'_n)$-insensitive}, or \emph{insensitive} for short, if one has
$$ |S(M_n) - S(M'_n)| \ll n^{-c}$$
for some fixed $c>0$.  By shrinking $c_0$ if necessary, our task is thus to show that the quantity
$$ \E G\left( \log|\det(M_n-z_1)|, \dots, \log|\det(M_n-z_k)| \right) $$
is insensitive.

The next step is to use \eqref{manz} to replace the log-determinants $\log |\det(M_n-z)|$ with the log-determinants $\log|\det W_{n,z}|$, where the $W_{n,z}$ are defined by \eqref{wnz}.  After translating and rescaling the function $G$, we thus see that it suffices to show that
$$ \E G\left( \log|\det(W_{n,z_1})|, \dots, \log|\det(W_{n,z_k})| \right) $$
is insensitive.

We observe the identity
$$ \log |\det(W_{n,z_j})| = \log |\det (W_{n,z_j}-\sqrt{-1}T)| - n \Im \int_0^T s_j(\sqrt{-1}\eta)\ d\eta$$
for any $T>0$ for all $1 \leq j \leq k$, where $s_j(z) := \frac{1}{n} \tr( W_{n,z_j} - z)^{-1}$ is the Stieltjes transform, as can be seen by writing everything in terms of the eigenvalues of $W_{n,z_j}$.  If we set $T := n^{100}$ then we see that
\begin{align*}
\log |\det (W_{n,z_j}-\sqrt{-1}T)| &= n \log T + \log|\det(1 - n^{-100} W_{n,z_j})| \\
&=n \log T + O(n^{-10})
\end{align*}
(say), thanks to \eqref{manx} and the hypothesis that $z_j$ lies in $B(0,(1-\delta)\sqrt{n})$.  Thus, by translating $G$ again, it suffices to show that the quantity
$$\E G\left( \left(n \Im \int_0^{n^{100}} s_j(\sqrt{-1}\eta)\ d\eta \right)_{j=1}^k \right)$$
is insensitive.

We need to truncate away from the event that $W_{n,z_j}$ has an eigenvalue too close to zero.
Let $\chi: \R \to \R$ be a smooth cutoff to the region $|x| \leq n^{3c_0}$ that equals $1$ for $|x| \leq n^{3c_0}/2$.  From Proposition \ref{lsv} and the union bound we have with probability $1-O(n^{-c_0+o(1)})$ that there are no eigenvalues of $W_{n,z_j}$ in the interval $[-n^{1-2c_0}, n^{-1-2c_0}]$ for all $1 \leq j \leq k$.  Combining this with Proposition \ref{ni} and a dyadic decomposition, we conclude that with probability $1-O(n^{-c_0+o(1)})$ one has
$$ |\Im s_j(\sqrt{-1} n^{-1-4c_0} )| \ll n^{2c_0+o(1)}$$
for all $1 \leq j \leq k$.  In particular, one has
$$\chi( \Im s_j(\sqrt{-1} n^{-1-4c_0} ) ) = 1$$
with overwhelming probability.

In view of this fact and \eqref{g-bound-2}, it suffices to show that the quantity
\begin{equation}\label{eton}
\E G\left( n \Im \int_0^{n^{100}} s_j(\sqrt{-1}\eta)\ d\eta \right) \chi\left( \left( \Im s_j(\sqrt{-1} n^{-1-4c_0} ) \right)_{j=1}^k \right)
\end{equation}
is insensitive.

Call a statistic $S$ \emph{very highly insensitive} if one has
$$ |S(M_n) - S(M'_n)| \ll n^{-2-c}$$
for some fixed $c>0$.  By swapping the real and imaginary parts of the components of $M_n$ with those of $M'_n$ one at a time, we see from telescoping series that it will suffice to show that \eqref{eton} is very highly insensitive whenever $M_n$ and $M'_n$ are identical in all but one entry, and in that entry either the real parts are identical, or the imaginary parts are identical.

Fix $M_n, M'_n$ as indicated.  Then for each $1 \leq j \leq k$, one has
\begin{align*}
W_{n,z_j} &= W_{n,z_j,0} + \frac{1}{\sqrt{n}} \xi V\\
W'_{n,z_j} &= W_{n,z_j,0} + \frac{1}{\sqrt{n}} \xi' V
\end{align*}
where $\xi, \xi'$ are real random variables that match to order $4$ and have the magnitude bound
\begin{equation}\label{xi-bound}
|\xi|, |\xi'| \ll n^{o(1)},
\end{equation}
 $V$ is an elementary matrix, and $W_{n,z_j,0}$ is a random Hermitian matrix independent of both $\xi$ and $\xi'$.  To emphasise this representation, and to bring the notation closer to that of the preceding section, we rewrite $s_j$ as $s^{(j)}_\xi$, where
$$ s^{(j)}_t(z) := \frac{1}{2n} \tr R^{(j)}_t(z)$$
and
$$ R^{(j)}_t(z) := (W_{n,z_j,0} + \frac{1}{\sqrt{n}} t V - z)^{-1}.$$
Our task is now to show that the quantity
\begin{equation}\label{lok}
\E G\left( n \Im \int_0^{n^{100}} s^{(j)}_\xi(\sqrt{-1}\eta)\ d\eta \right) \chi\left( \left( \Im s^{(j)}_\xi(\sqrt{-1} n^{-1-4c_0} ) \right)_{j=1}^k \right)
\end{equation}
only changes by $O(n^{-2-c})$ when $\xi$ is replaced by $\xi'$.

We now place some bounds on $R^{(j)}_t(z)$.

\begin{lemma}[Eigenvector delocalization]\label{laj}  Let $1 \leq j \leq k$, and suppose that we are in the event that $\chi( \Im s_j(\sqrt{-1} n^{-1-4c_0} ) )$ is non-zero.  Then with overwhelming probability, one has
\begin{equation}\label{rj}
 \sup_{\eta > 0} \| R^{(j)}_\xi(\sqrt{-1}\eta) \|_{(\infty,1)} \ll n^{O(c_0)}
\end{equation}
and hence (by Lemma \ref{neum} and \eqref{xi-bound}, swapping the roles of $\xi$ and $0$)
\begin{equation}\label{pretty}
\sup_{\eta > 0} \| R^{(j)}_0(\sqrt{-1}\eta) \|_{(\infty,1)} \ll n^{O(c_0)}.
\end{equation}
\end{lemma}

The bounds in the above lemma are similar to those from Proposition \ref{Resolv} (and Proposition \ref{Resolv} will be used in the proof of the lemma), but the point here is that the bounds remain uniform in the limit $\eta \to 0$, whereas the bounds in Proposition \ref{Resolv} blow up at that limit.

\begin{proof} By hypothesis and the support of $\chi$, one has
$$|\Im s^{(j)}_\xi(\sqrt{-1} n^{-1-4c_0} )| \ll n^{-3c_0}.$$
The left-hand side can be expanded as
$$ n^{-2-4c_0} \sum_{i=1}^n \frac{1}{\lambda_i(W_{n,z_j})^2 + n^{-2-8c_0}}$$
and so we obtain the lower bound
\begin{equation}\label{law}
 \lambda_i(W_{n,z_j}) \gg n^{-1-c_0/2}
\end{equation}
for all $i$.

From Proposition \ref{Resolv}, one already has
$$ \sup_{\eta > 1/n} \| R^{(j)}_\xi(\sqrt{-1}\eta) \|_{(\infty,1)} \ll n^{o(1)}$$
with overwhelming probability.  In particular, for each $1 \leq j \leq k$ and $\eta > 1/n$, one has
$$ \frac{\eta}{n} \sum_{i=1}^n \frac{|e_j^* u_i|^2}{\lambda_i(W_{n,z_j})^2 + \eta^2} \ll n^{o(1)}.$$
Combining this with \eqref{law}, we see that
$$ \frac{\eta}{n} \sum_{i=1}^n \frac{|e_l^* u_i|^2}{\lambda_i(W_{n,z_j})^2 + \eta^2} \ll n^{O(c_0)}.$$
for all $\eta > 0$, $1 \leq j \leq k$, and $1 \leq l \leq n$.  By dyadic summation we conclude that
$$ \sum_{i=1}^n \frac{|e_l^* u_i|^2}{(\lambda_i(W_{n,z_j})^2 + \eta^2)^{1/2}} \ll n^{O(c_0)}$$
for all $\eta > 1/n$, and thus by Cauchy-Schwarz one has
$$ |\frac{1}{n} \sum_{i=1}^n \frac{(e_l^* u_i) \overline{(e_m^* u_i)}}{\lambda_i(W_{n,z_j}) - \sqrt{-1} \eta}| \ll n^{O(c_0)}$$
for all $\eta > 0$ and $1 \leq j \leq k$, and $1 \leq l,m \leq n$.  But the left-hand side is the $lm$ coefficient of
$R^{(j)}_\xi(\sqrt{-1}\eta)$, and the claim follows.
\end{proof}

We now condition to the event that \eqref{pretty} holds for all $1 \leq j \leq k$; Lemma \ref{laj} ensures us that the error in doing so is $O_A(n^{-A})$ for any $A$.  Then by Proposition \ref{proper}, we have
$$ s^{(j)}_\xi(\sqrt{-1}\eta) = s^{(j)}_0(\sqrt{-1}\eta) + \sum_{i=1}^4 \xi^i n^{-i/2} c_i^{(j)}(\eta) + O( n^{-5/2 + O(c_0)} )  \min( 1, \frac{1}{n\eta} )$$
for each $j$ and all $\eta>0$, and similarly with $\xi$ replaced by $\tilde \xi$, where the coefficients $c_i^{(j)}$ enjoy the bounds
$$ |c_i^{(j)}| \ll n^{O(c_0)} \min( 1, \frac{1}{n\eta} ).$$
From this and Taylor expansion we see that the expression
$$ G\left( n \Im \int_0^{n^{100}} s_\xi(\sqrt{-1}\eta)\ d\eta \right) \chi\left( \Im s_\xi(E+\sqrt{-1} n^{-1-4c_0} ) \right)$$
is equal to a polynomial of degree at most $4$ in $\eta$ with coefficients independent of $\eta$, plus an error of $O( n^{-5/2 + O(c_0)} )$, which gives the claim for $c_0$ small enough.

\begin{remark}  If one assumes more than four matching moments, one can improve the final constant $c$ in the conclusion of Theorem \ref{four-moment}.  However, it appears that one cannot make $c$ arbitrarily large with this method, basically because the Taylor expansion becomes unfavorable when $c_0$ is too large.
\end{remark}

\section{Concentration of log-determinant for  gaussian matrices}\label{lower-sec}

In this section we establish Theorem \ref{loglower-gaussian}.  Fix $z_0 \in B(0,C)$; all our implied constants will be uniform in $z_0$.  Define $\alpha$ to be the quantity $\alpha := \frac{1}{2} (|z_0|^2-1)$ if $|z_0| \leq 1$, and $\alpha := \log |z_0|$ if $|z_0| \geq 1$.
Our task is to show that $\log |\det(M_n - z_0 \sqrt{n})|$ concentrates around $\frac{1}{2} n \log n + \alpha n$.

\subsection{The upper bound}

In this section, we prove that with overwhelming probability
$$\log |\det(M_n - z_0 \sqrt{n})|
\le \frac{1}{2} n \log n + \alpha n + n^{o(1)},$$
which is the upper bound of what we need. In fact, the statement (which is based on the second moment method) holds for  general
random matrices with non-gaussian entries.

\begin{proposition}[Upper bound on log-determinant]\label{logupper}
Let $M_n = (\xi_{ij})_{1 \leq i,j \leq n}$ be a random matrix with independent entries having mean zero and variance one.
Then for any $z_0 \in \C$, one has
$$ \log|\det(M_n - z_0 \sqrt{n})| \leq \frac{1}{2} n \log n +
\alpha n + O( n^{o(1)} )$$
with overwhelming probability.
\end{proposition}

The key is the following lemma.

\begin{lemma}  \label{secondmoment} Let $M_n = (\xi_{ij})_{1 \leq i,j \leq n}$ be a random matrix as above.
  Then for any $z_0 \in \C$, one has
\begin{equation}\label{vardet}
\E |\det(M_n - z_0 \sqrt{n})|^2 \leq n! \exp(|z_0|^2 n)
\end{equation}
for all $z_0$.  When $|z_0| \geq 1$, we have the variant bound
\begin{equation}\label{vardet-0}
\E |\det(M_n - z_0 \sqrt{n})|^2 \leq n^{n+1} |z_0|^{2n}.
\end{equation}
\end{lemma}

\begin{proof}
By cofactor expansion, one has
$$ \det(M_n - z_0 \sqrt{n}) = \sum_{\sigma \in S_n} \sgn(\sigma)
\prod_{i=1}^n (\xi_{i \sigma(i)} - z_0 \sqrt{n} 1_{\sigma(i) = i})$$
where $S_n$ is the set of permutations on $\{1,\dots,n\}$.  We can rewrite this expression as
$$ \sum_{A \subset \{1,\dots,n\}} \sum_{\sigma \in S_{n,A}} F_{A,\sigma}$$
where $S_{n,A}$ is the set of permutations $\sigma \in S_n$ that fix $A$, thus $\sigma(i)=i$ for all $i \in A$, and
$$ F_{A,\sigma} := (-z_0 \sqrt{n})^{|A|} \prod_{i \not \in A} \xi_{i \sigma(i)}.$$
As the $\xi_{ij}$ are jointly independent and have mean zero, we see that $\E F_{A,\sigma} \overline{F_{A',\sigma'}} = 0$ whenever $(A,\sigma) \neq (A',\sigma')$.  Also, as the $\xi_{ij}$ also have unit variance, we have $\E |F_{A,\sigma}|^2 = |z_0|^{2|A|} n^{|A|}$.  We conclude that
$$ \E |\det(M_n - z_0 \sqrt{n})|^2  = \sum_{A \subset \{1,\dots,n\}} \sum_{\sigma \in S_{n,A}} |z_0|^{2|A|} n^{|A|}.$$
Write $j=|A|$.  For each choice of $j=0,\dots,n$, there are $\frac{n!}{j!(n-j)!}$ choices for $A$, and $(n-j)!$ choices for $\sigma$.  We conclude that
$$ \E |\det(M_n - z_0 \sqrt{n})|^2 = n! \sum_{j=0}^n \frac{|z_0|^{2j} n^j}{j!}.$$
(This formula is well known in the literature; see e.g. \cite[Theorem 3.1]{edel}.)
Since
$$ \sum_{j=0}^\infty \frac{|z_0|^{2j} n^j}{j!} = \exp(|z_0|^2 n)$$
we obtain \eqref{vardet}.

Now suppose that $|z_0| \geq 1$, then the terms $\frac{|z_0|^{2j} n^j}{j!}$ are non-decreasing in $j$, and are thus each bounded by $|z_0|^{2n} n^n / n!$, and \eqref{vardet-0} follows.
\end{proof}

From Lemma \ref{secondmoment} and Stirling's formula, we see that
$$ \E |\det(M_n - z_0 \sqrt{n})|^2 \leq \exp( n \log n + 2 \alpha n + O(n^{o(1)}) )$$
and thus by Markov's inequality we see that
$$ |\det(M_n - z_0 \sqrt{n})|^2 \leq \exp( n \log n + 2 \alpha n + O(n^{o(1)}) )$$
with overwhelming probability, which gives Proposition \ref{logupper} as desired.

\subsection{Hessenberg form}

To finish the proof of Theorem \ref{loglower-gaussian}, we need to show the lower bound
$$ \log|\det(M_n - z_0 \sqrt{n})| \geq \frac{1}{2} n \log n +
\alpha n - O( n^{o(1)} )$$
with overwhelming probability. As we shall see later, the fact that we only seek a one-sided bound now instead of a two-sided one will lead to some convenient simplifications to the argument\footnote{If one really wished, one could adapt the arguments below to also give the upper bound, giving an alternate proof of Proposition \ref{logupper}, but this argument would be more complicated than the proof given in the previous section, and we will not pursue it here.}.

Now we will make essential use of the fact that the entries are gaussian.
The first step is to conjugate a complex gaussian matrix into an almost lower-triangular
form first observed in \cite{kv}, in the spirit of the tridiagonalisation of GUE matrices first observed by Trotter \cite{trotter}, as follows.

\begin{proposition}[Hessenberg matrix form]\label{hessen}\cite{kv}  Let $M_n$ be a complex gaussian matrix, and let $M'_n$ be the random matrix
$$
M'_n = \begin{pmatrix}
\xi_{11} & \chi_{n-1,\C} & 0 & 0 & \dots & 0 \\
\xi_{21} & \xi_{22} & \chi_{n-2,\C} & 0 & \dots & 0 \\
\xi_{31} & \xi_{32} & \xi_{33} & \chi_{n-3,\C} & \dots & 0 \\
\vdots & \vdots & \vdots & \vdots & \ddots & \vdots \\
\xi_{(n-1)1} & \xi_{(n-1)2} & \xi_{(n-1)3} & \xi_{(n-1)4} & \dots & \chi_{1,\C} \\
\xi_{n1} & \xi_{n2} & \xi_{n3} & \xi_{n4} & \dots & \xi_{nn}
\end{pmatrix}
$$
where $\xi_{ij}$ for $1 \leq j \leq i \leq n$ are iid copies of the complex gaussian $N(0,1)_\C$, and for each $1 \leq i \leq n-1$, $\chi_{i,\C}$ is a complex $\chi$ distribution of $i$ degrees of freedom (see Section \ref{notation-sec} for definitions), with the $\xi_{ij}$ and $\chi_{i,\C}$ being jointly independent.
Then the spectrum of $M_n$ has the same distribution as the spectrum of $M'_n$.

The same result holds when $M_n$ is a real gaussian matrix, except that $\xi_{ij}$ are now iid copies of the real gaussian $N(0,1)_\R$, and the $\chi_{i,\C}$ are replaced with real $\chi$ distribtions $\chi_{i,\R}$ with $i$ degrees of freedom.
\end{proposition}

\begin{proof}  This result appears in \cite[\S 2]{kv}, but for the convenience of the reader we supply a proof here.
We establish the complex case only, as the real case is similar, making the obvious changes (such as replacing the unitary matrices in the argument below by orthogonal matrices instead).

The idea will be to exploit the unitary invariance of complex gaussian vectors by taking a complex gaussian matrix $M_n$ and conjugating it by unitary matrices (which will depend on $M_n$) until one arrives at a matrix with the distribution of $M'_n$.

Write the first row of $M_n$ as $(\xi_{11}, \dots, \xi_{1n})$.  Then there is a unitary transformation $U_1$ that preserves the first basis vector $e_1$, and maps $(\xi_{11}, \dots, \xi_{1n})$ to $(\xi_{11}, \chi_{n-1,\C}, 0, \dots, 0)$, where $\chi_{n-1,\C}$ is a complex $\chi$ distribution with $n-1$ degrees of freedom.  If we then conjugate $M_n$ by $U_1$, and use the fact that the conjugate of a gaussian vector by a unitary matrix that is independent of that vector, remains distributed as a gaussian vector, we see that the conjugate $U_1 M_n U^*_1$ to a matrix takes the form
$$
\begin{pmatrix}
\xi_{11} & \chi_{n-1,\C} & 0 & \dots & 0 \\
\xi_{21} & \xi_{22} & \xi_{23} & \dots & \xi_{2n} \\
\vdots & \vdots & \vdots & \ddots & \vdots \\
\xi_{n1} & \xi_{n2} & \xi_{n3} & \dots & \xi_{nn}
\end{pmatrix},
$$
where the $\xi_{ij}$ coefficients appearing in this matrix are iid copies of $N(0,1)_\C$ (and are not necessarily equal to the corresponding coefficients of $M_n$), and $\chi_{n-1,\C}$ is independent of all of the $\xi_{ij}$.

We may then find another unitary transformation $U_2$ that preserves $e_1$ and $e_2$, and maps the second row $(\xi_{21},\dots,\xi_{2n})$ of $U_1 M_n U_1^*$ to $(\xi_{21},\xi_{22}, \chi_{n-2,\C},0,\dots,0)$, where $\chi_{n-2,\C}$ is distributed by the complex $\chi$ distribution with $n-2$ degrees of freedom.  Conjugating $U_1 M_n U_1^*$ by $U_2$, we arrive at a matrix of the form
$$
\begin{pmatrix}
\xi_{11} & \chi_{n-1,\C} & 0 & 0 & \dots & 0 \\
\xi_{21} & \xi_{22} & \chi_{n-2,\C} & 0 & \dots & 0 \\
\xi_{31} & \xi_{32} & \xi_{33} & \xi_{34} & \dots & \xi_{3n} \\
\vdots & \vdots & \vdots & \vdots & \ddots & \vdots \\
\xi_{n1} & \xi_{n2} & \xi_{n3} & \dots & \xi_{nn}
\end{pmatrix}
$$
where the $\xi_{ij}$ coefficients appearing in this matrix are again iid copies of $N(0,1)_\C$ (though they are not necessarily identical to their counterparts in the previous matrix $U_1 M_n U_1^*$), and $\chi_{n-1,\C}$ and $\chi_{n-2,\C}$ are independent of each other and of the $\xi_{ij}$.  Iterating this procedure a total of $n-1$ times, we obtain the claim.
\end{proof}

We now use this conjugated form of the complex gaussian matrix $M_n$ to describe the characteristic polynomial $\det(M_n - z_0 \sqrt{n})$.

\begin{proposition}\label{Detf}  Let $z_0$ be a complex number, and let $M_n$ be a complex gaussian matrix.  Let $\chi_{1,\C},\dots,\chi_{n-1,\C}$ be a sequence of independent random variables distributed according to the complex $\chi$ distributions with $1,\dots,n-1$ degrees of freedom respectively.  Let $\xi_1,\dots,\xi_n$ be another sequence of independent random variables distributed according to the complex gaussian $N(0,1)_\C$, and independent of the $\chi_i$.  Define the sequence $a_1,\dots,a_n$ of complex random variables recursively by setting
\begin{equation}\label{abo}
a_1 := \xi_1 - z_0 \sqrt{n}
\end{equation}
and
\begin{equation}\label{ab2}
 a_{i+1} := \frac{-z_0 \sqrt{n} a_i}{\sqrt{|a_i|^2 + \chi_{n-i,\C}^2}} + \xi_{i+1}
\end{equation}
for $i=1,\dots,n-1$.  (Note that the $a_i$ are almost surely well-defined.)  Then the random variable
$$ \left(\prod_{i=1}^{n-1} \sqrt{|a_i|^2 + \chi_{n-i,\C}^2}\right) a_n$$
has the same distribution as $\det(M_n-z_0\sqrt{n})$.

The same conclusions hold when $M_n$ is a real gaussian matrix, after replacing $\xi_i$ with copies of the real gaussian $N(0,1)_\C$, and replacing $\chi_{i,\C}$ with a real $\chi$ distribution $\chi_{i,\R}$ with $i$ degrees of freedom.
\end{proposition}

We remark that in \cite{kv} a slightly different stochastic equation (a Hilbert space variant of the P\'olya urn process) for the determinants $\det(M_n-z_0\sqrt{n})$ were given, in which the value of each determinant was influenced by a gaussian variable whose variance depended on all of the determinants of the top left $k \times k$ minors for $k=1,\ldots,n-1$.  In contrast, the recurrence here is more explicitly Markovian in the sense that the state $a_{i+1}$ of the recursion at time $i+1$ only depends (stochastically) on the state $a_i$ at the immediately preceding time.  We will rely heavily on the Markovian nature of the process in the subsequent analysis.

\begin{proof}  Again, we argue for the complex gaussian case only, as the real gaussian case proceeds similarly with the obvious modifications.

By Proposition \ref{hessen}, $\det(M_n-z_0\sqrt{n})$ has the same distribution as $\det(M'_n-z_0\sqrt{n})$.  The strategy is then to manipulate $M'_n-z_0\sqrt{n}$ by elementary column operations that preserve the determinant, until it becomes a lower triangular matrix whose diagonal entries have the joint distribution of $\left(\sqrt{|a_i|^2 + \chi_{n-i,\C}^2}\right)_{i=1}^{n-1}, a_n$, at which point the claim follows.

We turn to the details.  Writing $\xi_1 := \xi_{11}$, we see that $M'_n-z_0\sqrt{n}$ can be written as
$$
\begin{pmatrix}
a_1 & \chi_{n-1,\C} & 0 & 0 & \dots & 0 \\
\xi_{21} & \xi_{22}-z_0\sqrt{n} & \chi_{n-2,\C} & 0 & \dots & 0 \\
\xi_{31} & \xi_{32} & \xi_{33}-z_0\sqrt{n} & \chi_{n-3,\C} & \dots & 0 \\
\vdots & \vdots & \vdots & \vdots & \ddots & \vdots \\
\xi_{(n-1)1} & \xi_{(n-1)2} & \xi_{(n-1)3} & \xi_{(n-1)4} & \dots & \chi_{1,\C} \\
\xi_{n1} & \xi_{n2} & \xi_{n3} & \xi_{n4} & \dots & \xi_{nn}-z_0\sqrt{n}
\end{pmatrix}.
$$
Note that there is a unitary matrix $U_1$ whose action on row vectors (multiplying on the right) maps $(a_1, \chi_{n-1,\C},0,\dots,0)$ to $(\sqrt{|a_1|^2+\chi_{n-1,\C}^2}, 0,\dots,0)$, and which only modifies the first two coefficients of a row vector.  This corresponds to a column operation that modifies the first two columns of a matrix in a unitary fashion (by multiplying that matrix on the right by $U_1$).
Because complex gaussian vectors remain gaussian after unitary transformations, we see (after a brief computation) that this transformation maps the second row
$(\xi_{21}, \xi_{22}-z_0\sqrt{n}, \chi_{n-2,\C}, 0, \dots, 0)$ of the above matrix to a vector of the form
$$ \left(\ast, \frac{-z_0 \sqrt{n} a_1}{\sqrt{|a_1|^2 + \chi_{n-1,\C}^2}} + \xi_{2}, \chi_{n-2,\C}, \dots, 0\right)$$
where $\xi_2$ is a complex gaussian (formed by some combination of $\xi_{21}$ and $\xi_{22}$) and $\ast$ is a quantity whose exact value will not be relevant for us.  By \eqref{ab2}, we may denote the second coefficient of this vector by $a_2$.  The remaining rows of the matrix have their distribution unchanged by the unitary matrix $U_1$, because their first two entries form a complex gaussian vector.  Thus, after applying the $U_1$ column operation to the above matrix, we arrive at a matrix with the distribution
$$
\begin{pmatrix}
\sqrt{|a_1|^2+\chi_{n-1,\C}^2} & 0 & 0 & 0 & \dots & 0 \\
\ast & a_2 & \chi_{n-2,\C} & 0 & \dots & 0 \\
\xi_{31} & \xi_{32} & \xi_{33}-z_0\sqrt{n} & \chi_{n-3,\C} & \dots & 0 \\
\vdots & \vdots & \vdots & \vdots & \ddots & \vdots \\
\xi_{(n-1)1} & \xi_{(n-1)2} & \xi_{(n-1)3} & \xi_{(n-1)4} & \dots & \chi_{1,\C} \\
\xi_{n1} & \xi_{n2} & \xi_{n3} & \xi_{n4} & \dots & \xi_{nn}-z_0\sqrt{n}
\end{pmatrix}
$$
where the $\xi_{ij}$ here are iid copies of $N(0,1)_\C$ that are independent of $a_1$, $a_2$, and the $\chi_{i,\C}$ (and which are not necessarily identical to their counterparts in the previous matrix under consideration).  Of course, the determinant of this matrix has the same distribution as the determinant of the preceding matrix.

In a similar fashion, we may find a unitary matrix $U_2$ whose action on row vectors maps $(\ast, a_1, \chi_{n-2,\C}, 0, \dots, 0)$ to $(\ast, \sqrt{|a_2|^2+\chi_{n-2,\C}^2}, 0, \dots, 0)$, and which only modifies the second and third coefficients of a row vector.  Applying the associated column operation, and arguing as before, we arrive at a matrix with the distribution
$$
\begin{pmatrix}
\sqrt{|a_1|^2+\chi_{n-1,\C}^2} & 0 & 0 & 0 & \dots & 0 \\
\ast & \sqrt{|a_2|^2+\chi_{n-2,\C}^2} & 0 & 0 & \dots & 0 \\
\ast & \ast & a_3 & \chi_{n-3,\C} & \dots & 0 \\
\vdots & \vdots & \vdots & \vdots & \ddots & \vdots \\
\xi_{(n-1)1} & \xi_{(n-1)2} & \xi_{(n-1)3} & \xi_{(n-1)4} & \dots & \chi_{1,\C} \\
\xi_{n1} & \xi_{n2} & \xi_{n3} & \xi_{n4} & \dots & \xi_{nn}-z_0\sqrt{n}
\end{pmatrix}
$$
where again the values of the entries marked $\ast$ are not relevant for us.  Iterating this procedure a total of $n-1$ times, we finally arrive at a lower triangular matrix whose diagonal entries have the distribution of
$$
(\sqrt{|a_1|^2+\chi_{n-1,\C}^2}, \sqrt{|a_2|^2+\chi_{n-2,\C}^2}, \dots, \sqrt{|a_{n-1}|^2+\chi_{1,\C}^2}, a_n)$$
and whose determinant has the same distribution as that of $M'_n-z_0\sqrt{n}$ or $M_n - z_0\sqrt{n}$.  The claim follows.
\end{proof}

\subsection{A nonlinear stochastic difference equation}

For the sake of exposition, we now specialize to the complex gaussian case; the case when $M_n$ is a real gaussian is similar and we will indicate at various junctures what changes need to be made.

From Proposition \ref{Detf}, we see that $\log |\det(M_n-z_0\sqrt{n})|$ has the same distribution as
\begin{equation}\label{cha}
\frac{1}{2} \sum_{i=1}^{n-1} \log (|a_i|^2 + \chi_{n-i,\C}^2) + \log |a_n|.
\end{equation}
It thus suffices to establish the lower bound
\begin{equation}\label{low}
\frac{1}{2} \sum_{i=1}^{n-1} \log (|a_i|^2 + \chi_{n-i,\C}^2) + \log |a_n| \geq \frac{1}{2} n \log n + \alpha n - n^{o(1)}
\end{equation}
with overwhelming probability.

We first note that as the distribution of $\log |\det(M_n-z_0\sqrt{n})|$ is invariant with respect to phase rotation $z_0 \mapsto z_0 e^{\sqrt{-1}\theta}$, we may assume without loss of generality that $z_0$ is real and non-positive, thus
\begin{equation}\label{aor}
 a_{i+1} := \frac{|z_0| \sqrt{n} a_i}{\sqrt{|a_i|^2 + \chi_{n-i,\C}^2}} + \xi_{i+1}.
\end{equation}

\begin{remark} In the real gaussian case, one does not have phase rotation invariance.  However, by making the change of variables $a'_i := a_i e^{-\sqrt{-1} i \theta}$ one can obtain the variant
\begin{equation}\label{aor2}
 a'_{i+1} := \frac{|z_0| \sqrt{n} a'_i}{\sqrt{|a_i|^2 + \chi_{n-i,\R}^2}} + \xi'_{i+1}
\end{equation}
to \eqref{aor}, where $\xi'_{i+1} := e^{-\sqrt{-1}i\theta} \xi_{i+1}$.  It will turn out that this recurrence is similar enough to \eqref{aor} that the arguments below used to study \eqref{aor} can be adapted to \eqref{aor2}; the $\xi'_i$ are no longer identically distributed, but they still have mean zero, variance one, and are jointly independent, and this is all that is needed in the arguments that follow.
\end{remark}

The random variable $\chi_{n-i,\C}^2$ has mean $n-i$ and variance $n-i$.  As such, it is natural to make the change of variables
$$ \chi_{n-i,\C} =: n-i + \sqrt{n-i} \eta_{n-i}$$
where the $\eta_1,\dots,\eta_{n-1}$ have mean zero, variance one, and are independent of each other and of the $\xi_i$.

\begin{remark} For real gaussian matrices, the situation is very similar, except that the error terms $\eta_{n-i}$ now have variance two instead of one.  However, this will not significantly affect the concentration results for the log-determinant in this paper.  (This will however presumably affect any central limit theorems one could establish for the log-determinant, in analogy with \cite{TV-determinant}, though we will not pursue such theorems here.)
\end{remark}

We now pause to perform a technical truncation.  As the $\xi_i$ are distributed in a gaussian fashion, we know that
\begin{equation}\label{wi-1}
 \sup_{1 \leq i \leq n} |\xi_i| \leq n^{o(1)}
\end{equation}
with overwhelming probability.  Similarly, standard asymptotics for chi-square distributions also give the bound
\begin{equation}\label{wi-2}
 \sup_{1 \leq i < n} |\eta_i| \leq n^{o(1)},
\end{equation}
with overwhelming probability (this bound also follows from Proposition \ref{azuma}).

We may now condition on the event that \eqref{wi-1}, \eqref{wi-2} hold (for a suitable choice of the $o(1)$ decay exponent).  Importantly, the joint independence of the $\xi_1,\dots,\xi_n,\eta_1,\dots,\eta_{n-1}$ remain unchanged by this conditioning.  Of course, the distribution of the $\xi_i$ and $\eta_i$ will be slightly distorted by this conditioning, but this will not cause a difficulty in practice, as the mean, variances, and higher moments of these variables are only modified by $O(n^{-100})$ (say) at most, and also we will at key junctures in the proof be able to undo the conditioning (after accepting an event of negligible probability) in order to restore the original distributions of $\xi_i$ and $\eta_i$ if needed.

We return to the task of proving \eqref{low}.  We write \eqref{aor} as
\begin{equation}\label{aor-2}
 a_{i+1} := \frac{|z_0| \sqrt{n} a_i}{\sqrt{|a_i|^2 + n-i + \sqrt{n-i} \eta_{n-i}}} + \xi_{i+1}.
\end{equation}

We will treat this as a nonlinear stochastic difference equation in the $a_i$.  If we ignore the diffusion terms $\eta_{n-i}, \xi_{i+1}$, we see that \eqref{aor-2} is governed by the dynamics of the maps
\begin{equation}\label{amap}
a \mapsto \frac{|z_0| \sqrt{n} a}{\sqrt{|a|^2 + n-i}}
\end{equation}
as $i$ increases from $1$ to $n-1$.  In the regime $i < (1-|z_0|^2) n$, we see that this map has a stable fixed point at zero, while in the regime $i > (1-|z_0|^2) n$, this map has an unstable fixed point at zero and a fixed circle at $|a| = \sqrt{|z_0|^2 n - (n-i)}$.  This suggests that $|a_i|$ should concentrate somehow around $0$ for $i \leq (1-|z_0|^2) n$ and around $\sqrt{|z_0|^2 n - (n-i)}$ for $i \geq (1-|z_0|^2) n$.  In particular, this leads to the heuristic
$$ |a_i|^2 + \chi_{n-i,\C}^2 \approx \max( n-i, |z_0|^2 n ).$$
Note from the integral test that
\begin{align}
\frac{1}{2} \sum_{i=1}^{n-1} \log \max( n-i, |z_0|^2 n ) &= \frac{1}{2} \int_1^n \log \max(n-t, |z_0|^2 n)\ dt + O(n^{o(1)} ) \nonumber\\
&= \frac{1}{2} n \log n + \alpha n + O(n^{o(1)} )\label{integral-1},
\end{align}
where the second identity follows from a routine integration (treating the cases $|z_0| \leq 1$ and $|z_0| \geq 1$ separately).
This gives heuristic support for the desired bound \eqref{low}.

We now make the above analysis rigorous.  Because we are only seeking a lower bound \eqref{low}, the main task will be to obtain lower bounds that are roughly of the form
$$ |a_i|^2 + \chi_{n-i,\C}^2 \gtrapprox \max( n-i, |z_0|^2 n )$$
with overwhelming probability.
In the ``early regime'' $i \leq (1-|z_0|^2) n$, we will be able to achieve this easily from the trivial bound $|a_i| \geq 0$.  In the ``late regime''
$i \geq (1-|z_0|^2) n$, the main difficulty is then to show (with overwhelming probability) that $a_i$ avoids the unstable fixed point at zero, and instead is essentially at least as far away from the origin as the fixed circle $|a| = \sqrt{|z_0|^2 n - (n-i)}$.

We turn to the details.  We begin with a crude bound on the magnitude of the quantities $a_i$.

\begin{lemma}[Crude lower bound]\label{cub}  Almost surely (after conditioning to \eqref{wi-1}, \eqref{wi-2}), one has
\begin{equation}\label{up}
 \sup_{1 \leq i \leq n} |a_i| \leq (1+|z_0|) \sqrt{n}
\end{equation}
and with overwhelming probability
\begin{equation}\label{down}
 \inf_{1 \leq i \leq n} |a_i| \geq \exp(-n^{o(1)}).
\end{equation}
\end{lemma}

\begin{proof}
From \eqref{abo}, \eqref{wi-1} we see that we have
$$ |a_1| \leq 2 \sqrt{n}.$$
From \eqref{aor} (trivially bounding $\chi_{n-i}$ from below by zero) we have
$$ |a_{i+1}| \leq |z_0| \sqrt{n} + |\xi_{i+1}|$$
and so the bound \eqref{up} follows from \eqref{wi-1} and the assumption that $|z_0| \le 1$.

Now we prove \eqref{down}.  Let $A \geq 0$ be fixed.  Observe that $\xi_1$ has a bounded density function (even after conditioning on \eqref{wi-1}), so from \eqref{abo} we have
$$ |a_1| \geq n^{-A}$$
with probability\footnote{In the real gaussian case, the $n^{-2A}$ factor worsens to $n^{-A}$, but this does not impact the final conclusion.} $1-O(n^{-2A})$.  In a similar spirit, for any $i=1,\dots,n-1$, $\xi_{i+1}$ has a bounded density function, so from \eqref{aor} or \eqref{aor-2} (after temporarily conditioning $a_i$ and $\eta_{n-i}$ to be fixed) that
$$ |a_{i+1}| \geq n^{-A}$$
with probability $1-O(n^{-2A})$.  By the union bound, we conclude that
$$ \inf_{1 \leq i \leq n} |a_i| \geq n^{-A}$$
 with probability $1-O(n^{-2A+1})$.  Diagonalising in $A$, we obtain the claim.
\end{proof}

From this lemma, we conclude that
\begin{equation}\label{logo}
\log |a_i| = n^{o(1)}
\end{equation}
with overwhelming probability for each $1 \leq i \leq n$.  To show \eqref{low}, it thus suffices to establish, for each fixed $\eps>0$, that
$$
\frac{1}{2} \sum_{i=1}^{n-1} \log (|a_i|^2 + \chi_{n-i,\C}^2) \ge \frac{1}{2} n \log n + \alpha n - O(n^{O(\eps)}) $$
with overwhelming probability, where the implied constant in the $O(\eps)$ notation is understood to be independent of $\eps$ of course.

In view of \eqref{integral-1},  it will suffice to show that
\begin{equation}\label{ndel}
\sum_{n^\eps < i \leq n-n^\eps} \Big( \log (|a_i|^2 + \chi_{n-i,\C}^2) - \log \max( n-i, |z_0|^2 n )  \Big) \ge  - O(n^{O(\eps)})
\end{equation}
with overwhelming probability, as the contributions of the $i$ within $n^\eps$ of $1$ or $n$ can be controlled by $O(n^{\eps+o(1)})$ thanks to Lemma \ref{cub}.

\subsection{Lower bound  at early times}

We partition $\sum_{n^\eps < i \leq n-n^\eps} \Big( \log (|a_i|^2 + \chi_{n-i,\C}^2) - \log \max( n-i, |z_0|^2 n )  \Big)$ into two parts, according to the  heuristics following \eqref{amap}. The following simple lemma handles the first part of the partition.

\begin{lemma}[Concentration at early times]\label{early}  One has
$$
\sum_{n^\eps < i \leq \min((1-|z_0|^2)n + |z_0| n^{1/2+\eps}, n-n^\eps)} \log (|a_i|^2 + \chi_{n-i,\C}^2) - \log\max( n-i, |z_0|^2 n ) \geq -O(n^{O(\eps)})$$
with overwhelming probability.
\end{lemma}

\begin{proof}  We abbreviate the summation as $\sum_i$.
The key observation here is that we need only a lower bound, so we can use the trivial inequality

$$\log (|a_i|^2 + \chi_{n-i,\C} ) \ge \log \chi_{n-i,\C}. $$

It suffices to show that
\begin{equation}\label{chii-0}
\sum_i |\log (n-i) - \log\max(n-i,|z_0|^2 n)| = O( n^{O(\eps)} )
\end{equation}
and
\begin{equation}\label{chii}
\sum_i \log \chi_{n-i,\C}^2 - \log(n-i)= O( n^{O(\eps)} )
\end{equation}
with overwhelming probability.

We first verify \eqref{chii-0}.  The summand is only non-zero when $i = (1-|z_0|^2)n + j$ for some $0 < j \leq \min(|z_0| n^{1/2+\eps}, |z_0|^2 n - n^\eps)$, and so one can bound the left-hand side of \eqref{chii-0} by
$$ \sum_{0 < j \leq \min(|z_0| n^{1/2+\eps}, |z_0|^2 n - n^\eps)} |\log(|z_0|^2 n - j) - \log(|z_0|^2 n)|.$$
When $j \leq |z_0|^2 n -n^\eps$, we may bound
$$  |\log(|z_0|^2 n - j) - \log(|z_0|^2 n)| \ll n^{o(1)} \frac{j}{|z_0|^2 n},$$
and the claim then follows by summing over all $0 <j \leq |z_0| n^{1/2+\eps}$.

Now we verify \eqref{chii}, which is quite standard.  Writing $\chi_{n-i,\C}^2 = n-i + \sqrt{n-i} \eta_{n-i}$, we can write the left-hand side of \eqref{chii} as
$$ \sum_i \log (1 + \frac{\eta_{n-i}}{\sqrt{n-i}}).$$
From Taylor expansion and \eqref{wi-2} we then have
$$  \log (1 + \frac{\eta_{n-i}}{n-i}) = \frac{\eta_{n-i}}{\sqrt{n-i}} + O( \frac{n^{o(1)}}{n-i} ).$$
The sum of the error term is acceptable, so it suffices to show that
$$ \sum_i \frac{\eta_{n-i}}{\sqrt{n-i}} = O( n^{O(\eps)} )$$
with overwhelming probability.  But this follows\footnote{Strictly speaking, Proposition \ref{azuma} does not apply directly because the mean of the random variables $\eta_{n-i}$ deviates very slightly from zero when the conditioning \eqref{wi-2} is applied.  However, one can first apply Proposition \ref{azuma} to the unconditioned variables $\eta_{n-i}$, and then apply the conditioning \eqref{wi-2} that is in force elsewhere in this argument, noting that such conditioning does not affect the property of an event occuring with overwhelming probability.}
 from Proposition \ref{azuma}.
\end{proof}

\begin{remark}
Following the heuristics after \eqref{amap}, it would be more natural to consider
$n^{\eps} \le i \le (1-|z_0|^2) n$. The extra term $|z_0| n^{1/2+ \eps}$ in the upper bound of $i$ is needed for a technical reason
which will be clear in the analysis of larger $i$ (see Lemma \ref{large-init}).
\end{remark}

\subsection{Concentration at late times}

Define
\begin{equation}\label{i0-def}
i_0 := \max( n^\eps, (1-|z_0|^2)n + |z_0| n^{1/2+\eps}).
\end{equation}
In view of Lemma  \ref{early}, we see that to prove \eqref{ndel} it now suffices to establish the lower bound

\begin{equation} \label{latetime-0}
\sum_{i_0 < i \leq n-n^\eps} \log (|a_i|^2 + \chi_{n-i,\C}^2) - \log( |z_0|^2 n ) = O(n^{O(\eps)}) \end{equation}
with overwhelming probability.  In fact, we only need the lower bound from \eqref{latetime-0}, but the argument given here gives the matching upper bound as well with no additional effort.

Let us first deal with the easy case when
\begin{equation}\label{zup0}
|z_0| <  n^{-1/2+400 \eps}
\end{equation}
(say).  In this case, there are only $O(n^{800\eps})$ terms in the sum, and from Lemma \ref{cub} (discarding the non-negative $\chi_{n-i,\C}^2$ term) each term is at least $-O(n^{o(1)})$, so the claim \eqref{latetime-0} follows immediately.  (Note that the summation is in fact empty unless $|z_0| \geq n^{-1/2+\eps/2}$, so the $\log(|z_0|^2 n)$ term is $O(n^{o(1)})$.)  Thus, in the arguments below we can assume that
\begin{equation}\label{zup}
|z_0|  \ge  n^{-1/2+400\eps }.
\end{equation}

Observe from \eqref{aor} that
$$
\log (|a_i|^2 + \chi_{n-i,\C}^2) - \log( |z_0|^2 n ) = \log \frac{|a_{i+1} - \xi_{i+1}|^2}{|a_i|^2}.$$
From telescoping series and \eqref{logo} we have
$$
\sum_{i_0 < i \leq n-n^\eps} \log \frac{|a_{i+1}|^2}{|a_i|^2} = O(n^{o(1)})$$
with overwhelming probability, so by the triangle inequality it suffices to show that
$$
\sum_{i_0 < i \leq n-n^\eps} \log \frac{|a_{i+1}-\xi_{i+1}|^2}{|a_{i+1}|^2} = O(n^{O(\eps)})$$
with overwhelming probability.  We can rewrite
$$ \frac{|a_{i+1}-\xi_{i+1}|^2}{|a_{i+1}|^2} = |1 + \frac{\xi_{i+1}}{a'_i}|^{-2},$$
where
\begin{equation}\label{ap}
a'_i : = a_{i+1} -\xi_{i+1}  = \frac{|z_0| \sqrt n a_i}{ \sqrt {|a_i|^2 + \chi_{n-1,\C} }}.
\end{equation}
It  suffices to show that
$$
\sum_{i_0 < i \leq n-n^\eps} \log |1 + \frac{\xi_{i+1}}{a'_i}| = O(n^{O(\eps)})$$
with overwhelming probability.

The heart of the matter will be the following lemma.

\begin{lemma} \label{lemmaa_i}
With overwhelming probability  \begin{equation}\label{lower}
|a'_i| \gg n^{-100 \eps} (i - (1-|z_0|^2)n)^{1/2}
\end{equation} holds for all $i_0 < i \leq n-n^\eps$.
\end{lemma}

Assuming this lemma for the moment, we can then use it to conclude the proof as follows.  For any $i_0 < i \leq n-n^\eps$, one has
\begin{equation}\label{ion}
(i - (1-|z_0|^2)n)^{1/2} > (i_0 - (1-|z_0|^2)n)^{1/2}  \ge(  |z_0| n^{1/2+ \eps} )^{1/2} \geq n^{200\eps}
\end{equation}
by \eqref{i0-def} and \eqref{zup}, and thus by Lemma \ref{lemmaa_i}
$$|a'_i| \gg n^{100\eps}$$
with overwhelming probability.  From this and \eqref{wi-1} we see that
$$ |\frac{\xi_{i+1}}{a'_i}| =o(1);$$
indeed, the same argument gives the more precise bound
$$|\frac{\xi_{i+1}}{a'_i}| \ll n^{O(\eps)} (i - (1-|z_0|^2)n)^{-1/2}.$$
Performing a Taylor expansion (up to the second order term), we conclude that
$$ \log |1 + \frac{\xi_{i+1}}{a'_i}| = \Re \xi_{i+1}/a'_i + O( n^{O(\eps)} (i - (1-|z_0|^2)n)^{-1} )$$
with overwhelming probability.

The error terms $ O( n^{O(\eps)} (i - (1-|z_0|^2)n)^{-1} ) $ sum to $O(n^{O(\eps)})$, so it suffices to show that
\begin{equation}\label{hot}
\sum_{i_0 < i \leq n-n^\eps} \frac{\xi_{i+1}}{a'_i} = O(n^{O(\eps)})
\end{equation}
with overwhelming probability.  But from \eqref{lower}, one has
$$ \frac{1}{a'_i} = O( n^{O(\eps)} (i - (1-|z_0|^2 n)^{-1/2})$$
with overwhelming probability.  Also, the coefficient $\frac{1}{a'_i}$ depends on $\xi_1,\dots,\xi_i$ and $\chi_{1,\C},\dots,\chi_{n,\C}$ and is independent of $\xi_{i+1},\dots,\xi_n$, so the sum in \eqref{hot} becomes a martingale sum\footnote{Again, strictly speaking one should apply Proposition \ref{azuma} to the unconditioned variables and then apply the conditioning \eqref{wi-1}, \eqref{wi-2}, as in Lemma \ref{early}.}.
The claim then follows from Proposition \ref{azuma}.

It remains to prove \eqref{lower}. From \eqref{aor}, \eqref{ap}, \eqref{wi-1} we have
$$ a'_i = a_{i+1} - \xi_{i+1} = a_{i+1} + O(n^{o(1)})$$
and so by \eqref{ion} it will suffice to establish the bound
\begin{equation}\label{lower-ai}
|a_i| \gg n^{- 99\eps} (i - (1-|z_0|^2)n)^{1/2}
\end{equation}
with overwhelming probability for each $i_0 < i \leq n-n^\eps+1$.

In order to prove \eqref{lower-ai},  let us first establish a preliminary largeness result on $a_i$, which uses the diffusive term $\xi_{i+1}$ in \eqref{aor} to push this random variable away from the unstable equilibrium $0$ of the map \eqref{amap}:

\begin{lemma}[Initial largeness]\label{large-init}  With overwhelming probability, one has
\begin{equation}\label{supi}
 \sup_{\max(i_0 - \frac{1}{2} |z_0| n^{1/2+\eps},0) \leq i \leq i_0} |a_i| > A.
\end{equation}
where $A$ is the quantity
$$ A :=  |z_0|^{1/2} n^{1/4+\eps/10}.$$
\end{lemma}

\begin{proof}  Suppose first that
$$ i_0 - \frac{1}{2} |z_0| n^{1/2+\eps} \leq 0.$$
By \eqref{i0-def}, this implies that $|z_0| \gg 1$, and then from \eqref{abo}, \eqref{wi-1} we have $|a_1| \gg n^{1/2}$, which certainly gives \eqref{supi} in this case.  Thus we may assume that
$$ i_0 - \frac{1}{2} |z_0| n^{1/2+\eps} > 0.$$

It will suffice to show that, for each integer
$$ i_0 - \frac{1}{2} |z_0| n^{1/2+\eps} \leq i_1 \leq i_0$$
 and each fixed (i.e. conditioned) choice of $\xi_1,\dots,\xi_{i_1}$ and $\chi_{n-1,\C},\dots,\chi_{n-i_1}$, one has
\begin{equation}\label{con}
 \sup_{i_1 \leq i \leq i_1 + |z_0| n^{1/2+\eps/2} } |a_i| > A
\end{equation}
with conditional probability at least $q$ for some fixed $q>0$. Indeed, we can choose in the interval
$[i_0 - \frac{1}{2} |z_0| n^{1/2+\eps}, i_0 - |z_0| n^{1/2+\eps/2}]$ at least $\frac{n^{\eps/2} }{100}$ initial points $i_1,\dots,i_m$ so that the distance between any two of them is at least $|z_0| n^{1/2 +\eps/2}$.  If we let $E_j$ for $j=1,\dots,m$ be the event that \eqref{con} holds with $i_1$ replaced by $i_j$, then the above claim asserts that after conditining on the failure of the events $E_1,\dots,E_{j-1}$, the event $E_j$ holds with conditional probability at least $q$.  Multiplying the conditional probabilities together, we  then obtain \eqref{supi} with a failure probability of at most
$$ (1 - q)^{n^{\eps/2}/4}$$
which is $O(n^{-A})$ for any fixed $A>0$ as required.

Fix $i_0 - \frac{1}{2} |z_0| n^{1/2+\eps} \leq i_1 \leq i_0$ and $\xi_1,\dots,\xi_{i_1}$ and $\chi_{n-1,\C},\dots,\chi_{n-i_1,\C}$; all probabilities in this argument are now understood to be conditioned on these choices.  The quantity $a_{i_1}$ is now deterministic, and we may of course assume that
\begin{equation}\label{aio-init}
|a_{i_1}| \leq A
\end{equation}
as the claim is trivial otherwise.  We may also condition on the event that \eqref{wi-2} hold.  Let $i_2 := \lfloor i_1 + |z_0| n^{1/2+\eps/2} \rfloor$. Our goal is to show that
$$
\P( \sup_{i_1 \leq i \leq i_2} |a_{i}| > A ) \gg 1.
$$
For technical reasons (having to do with the contractive nature of the recursion \eqref{aor} when $a_i$ becomes large), it will be convenient to replace the random process $a_i$ by a slightly truncated random process $\tilde a_i$ for $i_0 \leq i \leq i_1$, which is defined by setting $\tilde a_{i_1} := a_{i_1}$ and
\begin{equation}\label{aor-trunc}
 \tilde a_{i+1} := \frac{|z_0| \sqrt{n} \tilde a_i}{\sqrt{\min(|\tilde a_i|, A)^2 + \chi_{n-i,\C}^2}} + \xi_{i+1}
\end{equation}
for $i_1 \leq i < i_2$.  From an induction on the upper range $i_2$ of the $i$ parameter, we see that
$$ \sup_{i_1 \leq i \leq i_2} |a_{i}| \leq A \iff \sup_{i_1 \leq i \leq i_2} |\tilde a_{i}| \leq A$$
and in particular
$$ |\tilde a_{i_2}| > A \implies \sup_{i_1 \leq i \leq i_2} |a_{i}| > A.$$
Thus it will suffice to show that
\begin{equation}\label{pazn}
\P( |\tilde a_{i_2}| > A ) \gg 1.
\end{equation}
By a standard Paley-Zygmund type argument, it will suffice to obtain the lower bound
\begin{equation}\label{two}
 \E |\tilde a_{i_2}|^2 \gg |z_0| n^{1/2 + \eps/2}
\end{equation}
on the second moment, and the upper bound
\begin{equation}\label{four-moment-bound}
 \E |\tilde a_{i_2}|^4 \ll |z_0|^2 n^{1 + \eps} + |z_0| n^{1/2 + \eps/2}  \E |\tilde a_{i_2}|^2
\end{equation}
on the fourth moment.  Indeed, if $p$ denotes the probability in \eqref{pazn}, then from H\"older's inequality one has
$$  \E |\tilde a_{i_2}|^2 \ll A^2 + p^{1/2} ( \E |\tilde a_{i_2}|^4 )^{1/2}$$
and then from \eqref{four-moment-bound} and \eqref{two} (and the definition of $A$) we obtain $p \gg 1$ as required.

It remains to establish \eqref{two} and \eqref{four-moment-bound}.  For this, we will use \eqref{aor-trunc} to  track the growth of the moments $\E |\tilde a_i|^2, \E |\tilde a_i|^4$ as $i$ increases from $i_1$ to $i_2$.

Let $i_1 \leq i < i_2$. From \eqref{aor-trunc} we thus have
$$
\E |\tilde a_{i+1}|^2 = \E \left|\frac{|z_0| \sqrt{n} \tilde a_i}{\sqrt{\min(|\tilde a_i|, A)^2 + n-i + \sqrt{n-i} \eta_{n-i} }} + \xi_{i+1}\right|^2
$$
The quantity $\xi_{i+1}$ has mean $O(n^{-100})$, variance $1+O(n^{-100})$ (the $O(n^{-100})$ errors arising from our conditioning to \eqref{wi-1}), and is independent of the other random variables on the right-hand side.  Thus (using \eqref{up}) we have
$$
\E |\tilde a_{i+1}|^2 = \E \left|\frac{|z_0| \sqrt{n} \tilde a_i}{\sqrt{\min(|\tilde a_i|, A)^2 + n-i + \sqrt{n-i} \eta_{n-i} ) }}\right|^2 + 1 + O(n^{-90}).$$
Upper bounding $\min(|\tilde a_i|, A)$ by $A$ and $n-i$ by $|z_0|^2 \sqrt{n} - |z_0| n^{1/2+\eps}/2$, and using \eqref{wi-2} (which we recall that we have conditioned on), we conclude that
$$ \min(|\tilde a_i|, A)^2 + n-i + \sqrt{n-i} \eta_{n-i}  \le |z_0|^2 n. $$
 This implies that
\begin{equation}\label{aii}
\E |\tilde a_{i+1}|^2 \ge \E |\tilde a_{i}|^2 + 1 + O(n^{-90}).
\end{equation}
  Iterating this $\gg |z_0| n^{1/2+\eps/2}$ times, we obtain \eqref{two} as required.

Now we turn to \eqref{four-moment-bound}.  Again, we let $i_1 \leq i <i_2$. From \eqref{aor-trunc} we have
$$
\E |\tilde a_{i+1}|^4 = \E \left|\frac{|z_0| \sqrt{n} \tilde a_i}{\sqrt{\min(|\tilde a_i|, A)^2 + n-i + \sqrt{n-i} \eta_{n-i} }} + \xi_{i+1}\right|^4.
$$
Expanding out the left-hand side using the independence and moment properties of $\xi_{i+1}$, we can estimate the above expression as
\begin{align*}
&\E \left|\frac{|z_0| \sqrt{n} \tilde a_i}{\sqrt{\min(|\tilde a_i|, A)^2 + n-i + \sqrt{n-i} \eta_{n-i} }}\right|^4 \\
&\quad + O\left( \E \left|\frac{|z_0| \sqrt{n} \tilde a_i}{\sqrt{\min(|\tilde a_i|, A)^2 + n-i + \sqrt{n-i} \eta_{n-i} }}\right|^2 + 1 \right).
\end{align*}
Using \eqref{wi-1}, \eqref{wi-2} and the bound $n-i \geq |z_0|^2 n - O(|z_0| n^{1/2+\eps})$, and discarding the non-negative $\min(|\tilde a_i|, A)^2$ term, we then obtain the upper bound
\begin{equation}\label{times}
\E |\tilde a_{i+1}|^4 \leq (1 + O(|z_0|^{-1} n^{-1/2+\eps})) \E |\tilde a_i|^4 + O( \E |\tilde a_i|^2 + 1 ),
\end{equation}
via a routine calculation.
From \eqref{aii} we have
$$
 \E |\tilde a_{i}|^2 \ll \E |\tilde a_{i_2}|^2.$$
From \eqref{aio-init} we also have
$$
 \E |\tilde a_{i_1}|^4 \ll |z_0|^2 n^{1 + \eps };$$
if we then iterate \eqref{times} $O( |z_0| n^{1/2+\eps/2} )$ times, we obtain \eqref{four-moment-bound} as desired.
\end{proof}

Now we need to use the repulsive properties of \eqref{amap} near the origin to propagate this initial largeness to later values of $i$.  The key proposition is the following.

\begin{proposition} \label{repulsion} Let $i_0 \leq i_1 \leq i_2 \leq n-n^\eps/2$.  Let $E_{i_1,i_2}$ be the event that $|a_i| \le \frac{1}{2} \sqrt {i- (1- |z_0|^2) n }$ for all $i_1 \le i \le i_2$. Then we have with overwhelming probability that
$$ |a_{i_2}| 1_{E_{i_1,i_2}} \ge  \left(1+ \frac{c  { i_1 - (1- |z_0| ^2)n }} {|z_0|^2 n} \right)^{i_2-i_1}  \left(|a_{i_1}|   + O(n^{o(1)} \sqrt {i-i_1} )\right) 1_{E_{i_1,i_2}},$$
for some constant $c >0$.
\end{proposition}

\begin{proof}  The probability in question will be computed over the product space generated by $\xi_i, \eta_i$ with $i_1 < i \le i_2$, conditioning all the other $\xi_i,\eta_i$ to be fixed.  In particular, $a_{i_1}$ is now deterministic.

For any $i_1 \leq i < i_2$, we see from \eqref{aor-2} that
\begin{equation}\label{Aiai}
 a_{i+1} =  \beta_i a_i + \xi_{i+1}
 \end{equation}
where $\beta_i$ is the positive real number
$$ \beta_i := \frac{|z_0| \sqrt{n}}{\sqrt{|a_i|^2 + n-i + \sqrt{n-i} \eta_{n-i}}}.$$

Next, from iterating \eqref{Aiai} we have
$$ a_{i_2} = \gamma_{i_1,i_2} \left( a_{i_1} + \sum_{i_1 \leq i < i_2} \delta_{i_1,i} \xi_{i+1} \right)$$
where $\gamma_{i_1,i_2} := \beta_{i_1} \dots \beta_{i_2-1}$ and $\delta_{i_1,i} := \beta_{i_1}^{-1} \dots \beta_i^{-1}$.

As the event $E_{i_1,i}$ contains $E_{i_1,i_2}$ for $i_1 \leq i < i_2$, we have
\begin{equation}  \label{I_T1} a_{i_2} 1_{E_{i_1,i_2}} =  \gamma_{i_1,i_2} 1_{E_{i_1,i_2}} ( a_{i_1} + \sum_{ i_1  \leq i <i_2  } \delta_{i_1, i } \xi_{i+1 }  1_{E_{i_1,i}} ). \end{equation}

Notice that if $E_{i_1,i}$  holds, then
$$|a_i|^2 \le \frac{1}{4} (i- (1-|z_0|^2) n)$$ which is equivalent to
$$ |a_i|^2 + n-i \leq |z_0|^2 n - \frac{3}{4} (i - (1-|z_0|^2) n).$$

On the other hand,
since
$$ i - (1-|z_0|^2) n \geq i_1 - (1-|z_0|^2) n \geq |z_0| n^{1/2+\eps}/2$$
and $n-i \leq |z_0|^2 n$, we deduce from \eqref{wi-2} that
$$ |a_i|^2 + n-i + \sqrt{n-i} \eta_{n-i} \leq |z_0|^2 n - \frac{1}{2} (i - (1-|z_0|^2) n)$$
(say) if $n$ is large enough.  This gives a bound of the form
$$ \beta_i \ge 1 + c \frac{i - (1-|z_0|^2)n}{|z_0|^2 n} \ge 1 + c \frac{i_1 - (1-|z_0|^2)n}{|z_0|^2 n} $$
for some absolute constants $c > 0$.

 From the definition of $\gamma_i$, we conclude the lower bound
\begin{equation}  \label{I_T2}
 |\gamma_{i_1,i_2}| 1_{E_{i_1,i_2}}  \ge  \left(1 + c \frac{i_1 - (1-|z_0|^2)n}{|z_0|^2 n}\right)^{i_2-i_1} 1_{E_{i_1,i_2}}    \end{equation}

\noindent and  the upper bound

\begin{equation} \label{I_T3} |\delta_{i_1,i}|  1_{E_{i_1,i}}  \le  1_{E_{i_1,i}} \le 1. \end{equation}

Let us now make a critical observation that the random variable $\delta_{i_1,i} 1_{E_{i_1,i}} $ depends on $\xi_2,\dots,\xi_i$ (and on the $\chi_{1,\C},\dots,\chi_{n-1,\C}$) but is independent of $\xi_{i+1},\dots,\xi_n$.  This enables us to apply Proposition \ref{azuma}, from which  we can conclude that with overwhelming probability

\begin{equation} \label{IT_5}  \sum_{i_1 \leq i < i_2} \delta_{i_1,i}  1_{E_{i_1,i} } \xi_{i+1}  = O(n^{o(1)} |i_2 - i_1|^{1/2} ) = O(n^{o(1)} \sqrt{i_2-i_1} ), \end{equation} concluding
the proof.
\end{proof}

\begin{corollary} \label{repulsion1}
Assume that $|a_{i_1} | \ge n^{\epsilon/100} T^{1/2} $ where  $T:= \lfloor \frac{|z_0|^2 n}{ { i_1- (1-|z_0|^2) n} } \log^2 n \rfloor$.  Then $1_{E_{i_1,i_1+T}} =0$  holds with overwhelming probability.
\end{corollary}

\begin{proof} Assume, for contradiction, that there is a fixed $A$ such that $\P (1_{E_T}) \ge n^{-A}$.  By the previous lemma, we can assume that

$$ | a_{i_1+ T }|  1_{E_{i_1,i_1+T}} \ge  \left(1+ \frac{c   { i_1 - (1- |z_0| ^2)n }} {|z_0|^2 n} \right) ^{T}  (|a_{i_1}|   + O(n^{o(1)} \sqrt  T ) 1_{E_{i_1,i_1+T}} )$$
holds with probability
at least $1- n^{-2A}$.  Taking expectations, we conclude
$$ \E |a_{i_1+T}| \ge  \E |a_{i_1+ T }| 1_{E_{i_1,i_1+T}} \ge  \left(1+ \frac{c  { i_1 - (1- |z_0| ^2)n }} {|z_0|^2 n} \right) ^{T} \Big( \E  |a_{i_1}|   + O(n^{o(1)} \sqrt  T )
\Big) (n^{-A} - n^{-2A}).$$

Since $|a_{i_1}| \ge n^{\eps/100} T^{1/2} $ and $(1+ \frac{c  { i_1 - (1- |z_0| ^2)n }} {|z_0|^2 n} ) ^{T}  \ge \exp( c \log^2 n)$ for some fixed $c>0$ by the definition of $T$, the RHS is bounded from below by

$$ n^{ -A}  \exp( c \log^2 n )   \gg n. $$

On the other hand, from Lemma \ref{cub} we have that
$$ \E  |a_{i_1+T}|  \leq (1+|z_0|) \sqrt n \ll \sqrt{n},$$
yielding the desired contradiction.
\end{proof}

Next, we observe that $a_i$ cannot drop in magnitude too quickly once it is somewhat small (assuming the hypotheses \eqref{wi-1}, \eqref{wi-2}, of course):

\begin{lemma} \label{difference} If  $|a_{i}| \le \frac{1}{2} \sqrt {i- (1- |z_0|^2) n }$ then
$|a_{i}| \geq |a_{i-1}| - n^{o(1)}$.
\end{lemma}

\begin{proof} From \eqref{aor} we have
$$ a_{i}-\xi_i = \frac{|z_0| \sqrt{n}}{\sqrt{|a_{i-1}|^2 + \chi_{n-i+1,\C} } } a_{i-1}. $$
and hence
$$ \frac{|z_0|^2 n}{|a_{i-1}|^2 + \chi_{n-i+1,\C}} |a_{i-1}|^2 = |a_{i} - \xi_{i}|^2.$$

We can rearrange this as
$$ |a_{i-1}|^2 = \frac{\chi_{n-i+1,\C}}{|z_0|^2 n - |a_{i} - \xi_{i}|^2}  |a_{i} - \xi_{i}|^2.$$

By \eqref{wi-2} we have
$$ \chi_{n-i+1,\C} = n - i +O (\sqrt {n-i} n^{o(1)}) = n-i  + O( n^{o(1)} |z_0| \sqrt{n} ),$$
using the fact that in this range $n-i \le |z_0|^2 n$.

From the assumption of the lemma, we have that
$$ |a_{i} - \xi_{i}|^2 \leq \frac{1}{4} (i - (1-|z_0|^2) n) + O( n^{o(1)} \sqrt{i - (1-|z_0|^2) n} )$$
and thus
$$ \chi_{n-i+1,\C} - |z_0|^2 n + |a_{i} - \xi_{i}|^2
\leq -\frac{3}{4} (i - (1-|z_0|^2) n) + O( n^{o(1)} |z_0| \sqrt{n} ) + O( n^{o(1)} \sqrt{i - (1-|z_0|^2) n} ).$$

As $i -(1-|z_0|^2 n ) \ge |z_0| n^{1/2 +\eps}$,  we see that the right-hand side is negative for $n$ large enough, thus
$$ \frac{\chi_{n-i+1,\C}}{|z_0|^2 n - |a_{i} - \xi_{i}|^2} \leq 1.$$
We thus have
$$ |a_{i-1}| \leq |a_{i} - \xi_{i_1}|, $$
which implies from \eqref{wi-1} that
$|a_i| \ge |a_{i-1}| - n^{o(1)}$ as desired.
\end{proof}

We can now prove the lower bound \eqref{lower-ai} with overwhelming probability as follows.
We first condition on the event that the conclusion of Lemma \ref{large-init} holds. Now assume that there is some $i_0 < i \le n - n^{\eps}$ such that

$$|a_i| \le \frac{1}{3} \sqrt { i- (1-|z_0|^2)  n }. $$

Let $i_2$ be the first such index. In particular,
\begin{equation}\label{ai2}
|a_{i_2}| \le \frac{1}{3} \sqrt { i_2- (1-|z_0|^2)  n } \leq \frac{1}{2} \sqrt { i_2- (1-|z_0|^2)  n }.
\end{equation}
By Lemma \ref{large-init}, we can then locate an index $\max(i_0-\frac{1}{2} |z_0| n^{1/2+\eps},0) + 1 \le i_1 < i_2$ such that $|a_i| \le \frac{1}{2} \sqrt{ i -( 1-|z_0|^2) n} $ for all $i_1 \le i \le i_2$ (or in other words, $E_{i_1,i_2}$ holds) and
$$ |a_{i_1-1}| > \frac{1}{2} \sqrt{ i_1-1 -( 1-|z_0|^2) n}.$$
From Lemma \ref{difference}, this implies in particular that
\begin{equation}\label{ai1}
|a_{i_1} | \ge  .499 \sqrt{ i_1 - (1-|z_0|^2) n }.
\end{equation}

From the above discussion and the union bound, it thus suffices to show that for any given $i_0 \leq i_1 < i_2 \leq n-n^\eps$, the event that \eqref{ai2} and \eqref{ai1} and $E_{i_1,i_2}$ all simultaneously hold, is false with overwhelming probability.

Fix $i_1,i_2$.  If $i_2-i_1 > T$ then by Corollary \ref{repulsion1}, $1_{E_{i_1,i_2}} =0$ with overwhelming probability and we are done. In the other case $i_2-i_1 \leq T$, by Proposition \ref{repulsion}, we have  with overwhelming probability

\begin{equation} \label{BB}  |a_{i_2}| 1_{E_{i_1,i_2}} \ge  \left(1+ \frac{c   { i_1 - (1- |z_0| ^2)n }} {|z_0|^2 n} \right) ^{i_2-i_1}  (|a_{i_1}|   + O(n^{o(1)} \sqrt {i-i_1} )) 1_{E_{i_1,i_2}} .
\end{equation}

It now suffices to verify that if  $|a_{i_1}|  \ge .499 \sqrt {i_1 -(1-|z_0|^2 ) n }$, $E_{i_1,i_2}$ holds, and $|a_{i_2}| \le \frac{1}{3} \sqrt { i_2- (1-|z_0|^2)  n }$, then the above inequality is violated. Notice  that   since $i_2 -i_1 \le T = \frac{|z_0|^2 n}{ i_1 -(1-|z_0|^2 n)} \log^2 n$ and
$i_1 - (1-|z_0|^2)n \gg |z_0| n^{1/2 +\eps}$, we have

$$|a_{i_1}|  + O(n^{o(1)} \sqrt{i_2-i_1} \ge .499 \sqrt {i_1 -(1-|z_0|^2 ) n } - O(n^{o(1)} T^{1/2} ) \ge \frac{5}{12}  \sqrt {i_1 -(1-|z_0|^2 ) n } . $$
As $E_{i_1,i_2}$ holds, it follows that the RHS of \eqref{BB} is at least

$$  \frac{5}{12}  \sqrt {i_1 -(1-|z_0|^2 ) n }  > \frac{1}{3} \sqrt {i_2 - (1-|z_0|^2 n) } $$ again thanks to the fact that $i_2-i_1 \le T=o( i_1- (1-|z_0|^2)n)$. Our proof is complete.

\begin{remark} All the above arguments go through without difficulty in the real case, using \eqref{aor2} instead of \eqref{aor}, replacing $a_i, \xi_{i}, \chi_{i,\C}$ by $a'_i, \xi'_i, \chi_{i,\R}$ respectively; we leave the details to the interested reader.
\end{remark}

\section{Concentration of log-determinant for iid matrices}\label{lower-sec2}

Now that we have established concentration of the log-determinant in the special case of real and complex gaussian matrices (Theorem \ref{loglower-gaussian}), we are now ready to apply the resolvent swapping machinery from Section \ref{resolvent-sec} to obtain concentration for more general iid matrices (Theorem\ref{loglower}).

Fix $\delta, z_0$.  Let $W_{n,z_0}$ be defined as in \eqref{wnz}.  As in the previous section, set $\alpha$ equal to $\frac{1}{2} (|z_0|^2-1)$ if $|z_0| \leq 1$, and $\log|z_0|$ if $|z_0| \geq 1$.  It suffices to show that
$$ \log|\det(W_{n,z_0})| = 2 n \alpha + O( n^{o(1)} )$$
with overwhelming probability, uniformly in $z_0$.  We may assume without loss of generality that all entries of $M_n$ are $O(n^{o(1)})$.

We observe the identity
$$ \log |\det(W_{n,z_0})| = \log |\det (W_{n,z_0}-\sqrt{-1}T)| - n \Im \int_0^T s(\sqrt{-1}\eta)\ d\eta$$
for any $T>0$, where $s(z) := \frac{1}{n} \tr( W_{n,z_0} - z)^{-1}$ is the Stieltjes transform, as can be seen by writing everything in terms of the eigenvalues of $W_{n,z_0}$.  If we set $T := n^{100}$ then we see that
\begin{align*}
\log |\det (W_{n,z_0}-\sqrt{-1}T)| &= n \log T + \log|\det(1 - n^{-100} W_{n,z_0})| \\
&=n \log T + O(n^{-10})
\end{align*}
(say), thanks to \eqref{manx} and the hypothesis that $|z_j| \le \sqrt n $.   Thus it suffices to show that
$$ n \Im \int_0^T s(\sqrt{-1}\eta)\ d\eta = n \log T - 2 n \alpha + O( n^{o(1)} )$$
with overwhelming probability.

Now we eliminate the contribution of very small $\eta$.

\begin{lemma}  One has
$$ n \Im \int_0^{1/n} s(\sqrt{-1}\eta)\ d\eta = O(n^{o(1)})$$
with overwhelming probability.
\end{lemma}

\begin{proof} From Proposition \ref{Resolv} we see with overwhelming probability that
$$ |s(\sqrt{-1}\eta)| \ll n^{o(1)} (1 + \frac{1}{n\eta})$$
for all $\eta > 0$.  This already handles the portion of the integral where $\eta > n^{-2\log n}$ (say).  For the remaining portion when $0 < \eta \leq n^{-2\log n}$, we observe from Proposition \ref{lsv} that with overwhelming probability, all eigenvalues of $W_{n,z_0}$ are at least $n^{-\log n}$ in magnitude, which implies that $s(\sqrt{-1}\eta) = O( n^{1 + \log n} )$ for all such $\eta$, and the claim follows.
\end{proof}

Set $X:= n \Im \int_{1/n}^T s( \sqrt{-1} \eta) d \eta $ and $ X_* :=  n \log T - 2n\alpha$.
Fix arbitrary constants $A, \epsilon >0$. In view of the above lemma, it suffices to show that
$$
\P(|X-X_*| \ge n^{\epsilon}) \ll n^{-A}.$$
By Markov's inequality, it suffices to show that for  $j= 2\lfloor A/\eps \rfloor$

\begin{equation} \label{moment11} \E (X-X_*)^j = O( n^{j \epsilon/2 } ). \end{equation}

Without loss of generality we may assume $j$ to be large, e.g. $j > 5$.  By Theorem \ref{loglower-gaussian}, we know that a stronger bound

\begin{equation} \label{moment22} \E (X'-X_*)^j  \le n^{\epsilon} \end{equation}

\noindent  holds for the same range of $j$ (for $n$ sufficiently large depending on $\eps$ and $j$), where $X'$ is defined as in $X$ but with $M_n$ replaced by a random real or complex gaussian matrix $M'_n$ that matches $M_n$ to third order.

  We now execute the following swapping process. Start with the random gausian matrix $M'_n$ and in each step swap either the real or imaginary part of a gaussian entry of $M'_n$ to the associated real or imaginary part of the corresponding entry of $M_n$.  The exact order in which we perform this swapping is not important, so long as it is chosen in advance; for instance, one could use lexicographical ordering, swapping the real part and then the imaginary part for each entry in turn.  Let $M_{n}^{[k]}$, $0 \le k \le 2n^2$ be the resulting  random matrix at time
  $k$ and define $X^{[k]}$ accordingly. We will show, by induction on $k$, that

  \begin{equation} \label{moment33} \E (X^{[k]} -X_*)^j  \le \left(1+ \frac{k}{n^{2+ \eps/8j}}\right) n^{\eps} \end{equation}

  \noindent  for $n$ sufficiently large depending on $\eps$ and $j$ (but not on $k$).    Note that the base case $k=0$ of \eqref{moment33} holds thanks to \eqref{moment22}, while the case $k=2n^2$ implies \eqref{moment11} with some room to spare.

  For technical reasons, it is convenient to assume that $|\xi|, |\xi'| =n^{o(1)}$ with probability one. This can be done replacing all entries
  $\xi_{ij}$ by $\xi_{ij} \I_{|\xi_{ij}| \le \log^B n }$ and   $\xi'_{ij}$ by $\xi'_{ij} \I_{|\xi'_{ij} |\le \log^B n }$, where $B$ is a sufficiently large constant so that
  with overwhelming probability $|\xi_{ij}| + |\xi_{ij}'|  < \log^B n$ for all $i,j$. It is clear that any event that holds with overwhelming probability in the truncated model
  also holds with overwhelming probability in the original one. Thus, we can reduce to the truncated case.
  At this point we would like to
  point out that the truncation does change the moments of the entries, but by a very small amount that will only introduce negligible factors such as $O(n^{-100})$ to  the swapping argument.
  Abusing the notion slightly, from now on we still work with $\xi$ and $\xi'$ but under the extra assumption that with probability one
  $|\xi|, |\xi'| \le \log^B n=n^{o(1)}$.

Fix a step $0 \leq k < 2n^2$, and consider the difference

  \begin{equation} \label{moment0} D_k := \E (X^{[k+1]} -X_*)^j -\E (X^{[k]} -X_*)^j  = \int \E [ (X^{[k+1]} -X_*)^j-  (X^{[k]} -X_*)^j ]| M_0 ) d M_0. \end{equation}

  \noindent where $M_0$ is obtained from $X^{[k+1]} $ by putting $0$ at the swapping position (in other words, $M_0$ is the common part of $M^{[k]}$ and   $M^{[k+1]} $), and $dM_0$ is the law of $M_0$.  Once conditioned on $M_0$,  we can simplify the notation by
 replacing  $X^{[k]} $ and $X^{[k+1]}$ by  $X_{\xi}$ and $X_{\xi'}$ respectively.

 It is important to notice that since $\eta \ge 1/n$, we can bound $|s_{\xi} (\sqrt {-1} \eta)|$  crudely by $n$  with probability one
 (for any matrix $M_n^{[k]}$). As $T = n^{100}$, this implies that
 $|X^{[k]}| \ll n^{102}$  and

 \begin{equation} \label{moment00} |(X^{[k]} -X_*)^j-  (X^{[k+1]} -X_*)^j | \ll n^{102 j} \end{equation} for any $j$, with probability one.

 By Proposition \ref{Resolv}, we see with overwhelming probability that
$$ \| R_\xi(\sqrt{-1}\eta) \|_{(\infty,1)} \ll n^{o(1)} $$
for all $\eta \geq n^{-1}$. In this case,  by Lemma \ref{neum} and \eqref{xi-bound}
\begin{equation}\label{ax-1}
\| R_0(\sqrt{-1}\eta) \|_{(\infty,1)} \ll n^{o(1)}
\end{equation}
for all such $\eta$.

If \eqref{ax-1} holds, we say that $M_0$ is {\it good}. The contribution from bad $M_0$ in the RHS of \eqref{moment0} is very small. Indeed, by Proposition
\ref{Resolv}, we can assume that $M_0$ is bad with probability at most $n^{-102 j -100}$.  By the upper bound \eqref{moment00}, the integral
(in $D_k$)  over
the bad $M_0$ is at most

\begin{equation} \label{moment1} n^{-102 j-100}  n^{102 j} = n^{-100}.  \end{equation}

Let us now condition on a good $M_0$.  By Proposition \ref{proper}, we have

\begin{equation} \label{moment2} s_{\xi} (\sqrt{-1} \eta) = s_0 + \sum_{i=1}^3 \xi^i n^{-i/2} c_i (\eta) + O(n^{-2 +o(1)} \frac{1}{ n \eta}). \end{equation}

\noindent where the coefficient $c_i(\eta)$ is independent of $\xi$ and enjoys the bound $|c_i(\eta)| \ll n^{o(1)} \frac{1}{n \eta}$.

Multiplying by $n$ and taking the integral over $\eta$, we obtain,

\begin{equation} \label{moment3} X_\xi= X_0 + P(\xi)  + O(n^{-2+ o(1)}) \end{equation}

\noindent where $P= \sum_{i=1}^3  \xi^i  n^{-i/2} d_i$ is a polynomial in $\xi$ with coefficients $d_i = O(n^{o(1)})$, and $X_0$ is a quantity independent of $\xi$.  As $|\xi| =n^{o(1)}$ with probability one, it follows that
$|X_{\xi} -X_0 | = n^{-1/2+o(1)}$ with probability one.  Furthermore,

\begin{equation} \label{moment4}  X_\xi -X_* = (X_0 -X_*) + P(\xi)  + O(n^{-2+ o(1)}).\end{equation}

We raise this equation to the power $j$, focusing on those terms of order $\xi^4$ or more.  As $d_i = O(n^{o(1)})$, using the fact that $|\xi| \le n^{o(1)}$ with probability one and $j >5$,
 we have

\begin{equation}  \label{moment5} (X_\xi-X_*)^j= P_j (\xi) +  O(n^{-2+ o(1)} \sum_{l=1}^{j-1} |X_0- X_*|^{l} + n^{-5/2+o(1)} ) . \end{equation}

 \noindent where $P_j$ is a polynomial of degree at most $3$.  Therefore,

 \begin{equation}  \label{moment6}
 \E (X_\xi-X_*)^j=  \E P_j (\xi) +  O(n^{-2+ o(1)} \sum_{k=1}^{j-1}  |X_0- X_*|^{k} + n^{-5/2+o(1)} ) . \end{equation}

Similarly

 \begin{equation}  \label{moment6'}
 \E (X_{\xi'} -X_*)^j=  \E P_j (\xi') +  O(n^{-2+ o(1)} \sum_{k=1}^{j-1}  |X_0- X_*|^{k} + n^{-5/2+o(1)} ) . \end{equation}

 Here the expectations are with respect to $\xi$ and $\xi'$ (as we already conditioned on a good $M_0$.)
 It follows that

  \begin{equation}  \label{moment7}
 \E (X_\xi-X_*)^j - \E (X_{\xi'} -X_*)^j =  \E (P_j (\xi)-P_j (\xi')) +  O(n^{-2+ o(1)} \sum_{k=1}^{j-1}  |X_0- X_*|^{k} + n^{-5/2+o(1)} ) . \end{equation}

 As already pointed out, the first three moments of $\xi$ and $\xi'$ do not entirely match due to the truncation. However, by fixing $B$ large enough, we can assume that the truncation changes each moment  by at most $n^{-C}$ for some sufficiently large $C$ (we need $C$ to be larger than the absolute value of the coefficients of $P_j$, which are of size $O(n^{O(1)})$, again thanks to the fact that  $|s_{\xi} (\sqrt {-1} \eta)| \le n$  with probability one).   This yields

 \begin{equation}\label{moment8}
 \E (X_\xi-X_*)^j - \E (X_{\xi'} -X_*)^j =  O(n^{-2+ o(1)} \sum_{k=1}^{j-1}  |X_0- X_*|^{k} + n^{-5/2+o(1)} ) . \end{equation}

But  $ |X_{\xi} -X_0| \le n^{-1/2 +o(1)}$ with probability one,   so \eqref{moment0} implies

 \begin{equation}\label{moment8a}
 \E (X_\xi-X_*)^j - \E (X_{\xi'} -X_*)^j =   O(n^{-2+ o(1)} \sum_{k=1}^{j-1} \E |X_\xi - X_*|^{k} + n^{-5/2+o(1)} ).
 \end{equation}

The right-hand side of \eqref{moment8a} can be bounded as

\begin{equation} \label{moment9}  O(n^{-2 +o(1)}\min \{  \E |X_{\xi} -X^*|^j n^{-\eps/4j}, n^{\eps/2} \}) , \end{equation}

\noindent where the bound comes from considering two cases $\E |X_\xi  - X_*|^j $ being not smaller or smaller than  $n^{\eps/2}$, and the Holder inequality.

Thus, conditioned on a good $M_0$, we have

$$ | \E (X_\xi-X_*)^j - \E (X_{\xi'} -X_*)^j | \ll n^{-2 +o(1)}\min \{  |X_{\xi} -X^*|^j n^{-\eps/4j}, n^{\eps/2} \} . $$

Taking into account \eqref{moment1}, we conclude

$$D_k \ll n^{-100} + n^{-2 - \eps/4j } \E  |X_{\xi} -X_*|^j + n^{-2 + \eps/2 +o(1)} ,$$
and the desired bound \eqref{moment33} on $\E (X^{[k+1]} -X_*)^j $ follows easily by the induction hypothesis.

\appendix

\section{Spectral properties of $W_{n,z}$}\label{conc}

In this appendix we prove Proposition \ref{ni} and Proposition \ref{Resolv}.  We fix $M_n$, $C$, $z_0$ as in these propositions.  By truncation we may assume that all the coefficients of $M_n$ have magnitude $O(n^{o(1)})$.

\subsection{Crude upper bound}

We begin with Proposition \ref{ni}, which we will prove by modifying the argument from \cite[Appendix C]{TVlocal1} and \cite[Proposition 28]{TVlocal3}.  Write $I = [E-\eta,E+\eta]$.  It suffices to establish the claim in the case $1/n \leq \eta \leq 1$, as the general case then follows from this case (and from the trivial bound $N_I \leq 2n$).  By rounding $\eta$ to the nearest integer power of two, and using the union bound, it suffices to establish the claim for a single $\eta$ in this range, which we now fix.  Similarly, we may round $E$ to a multiple of $\eta$; since the claim is easy for (say) $|E| \geq n^{10}$, we see from the union bound that it suffices to establish the claim for a single $E$, which we now also fix.  By symmetry we may take $E \geq 0$.

By a diagonalisation argument, it will suffice to show for each fixed $c>0$ that one has
$$ N_{[E-\eta,E+\eta]} \leq n^{1+c} \eta$$
with overwhelming probability.  Accordingly, we assume for contradiction that
\begin{equation}\label{contra}
N_{[E-\eta,E+\eta]} > n^{1+c} \eta.
\end{equation}
We use the Stieltjes transform
$$ s(E + \sqrt{-1} \eta) = \frac{1}{2n} \tr (W_{n,z} - E - \sqrt{-1}\eta)^{-1}.$$
Then
$$
\Im s(E + \sqrt{-1} \eta) = \frac{1}{2n} \sum_{j=1}^{2n} \frac{\eta}{(\lambda_j(W_{n,z})-E)^2 + \eta^2};$$
from \eqref{contra} we thus have
$$ \Im s(E + \sqrt{-1} \eta) \gg n^c.$$
In particular, since
$$ s(E + \sqrt{-1} \eta) = \frac{1}{2n} \sum_{j=1}^{2n} R(E + \sqrt{-1} \eta)_{jj}$$
we see from the pigeonhole principle that we have
\begin{equation}\label{Rej}
 |R(E+\sqrt{-1}\eta)_{jj}| \gg n^c
\end{equation}
for some $1 \leq j \leq 2n$.  By the union bound, it suffices to show that for each $j$, the hypothesis \eqref{Rej} (combined with \eqref{contra}) leads to a contradiction with overwhelming probability.

Fix $j$; by symmetry we may take $j=2n$, thus
\begin{equation}\label{rennes}
 |R(E+\sqrt{-1}\eta)_{2n,2n}| \gg n^c.
\end{equation}
We expand $W_{n,z}$ as
$$ W_{n,z} = \begin{pmatrix} W'_{n,z} & X \\ X^* & 0 \end{pmatrix}$$
where $W'_{n,z}$ is the $2n-1 \times 2n-1$ Hermitian matrix
$$ W'_{n,z} := \begin{pmatrix}
0 & 0 & \frac{1}{\sqrt{n}} (M_{n-1} - z) \\
0 & 0 & Z \\
\frac{1}{\sqrt{n}} (M_{n-1}-z)^* & Z^* & 0
\end{pmatrix}
$$
where $M_{n-1}$ is the top left $n-1 \times n-1$ minor of $M_n$, $Z$ is the $n-1$-dimensional row vector with entries $\frac{1}{\sqrt{n}} \xi_{nj}$ for $j=1,\dots,n-1$, $X$ is the $2n$-dimensional column vector
$$ X := \begin{pmatrix}
X' \\
\frac{1}{\sqrt{n}} (\xi_{nn} - z)  \\
0
\end{pmatrix}
$$
and $X'$ is the $n-1$-dimensional column vector with entries $\frac{1}{\sqrt{n}} \xi_{jn}$ for $j=1,\dots,n-1$.

By Schur's complement, the resolvent coefficient $R(E+\sqrt{-1}\eta)_{2n,2n}$ can be expressed as
\begin{equation}\label{resolve-eq}
R(E+\sqrt{-1}\eta)_{2n,2n} = \frac{1}{-E-\sqrt{-1}\eta - Y_n}
\end{equation}
where $Y_n$ is the expression
$$ Y_n  := X^* (W'_{n,z} - E - \sqrt{-1}\eta)^{-1} X.$$
By \eqref{rennes} we conclude that
$$ |E+\sqrt{-1}\eta + Y_n| \ll n^{-c};$$
as $Y_n$ has a non-negative imaginary part, we conclude that
\begin{equation}\label{imoen}
\Im Y_n \ll n^{-c}.
\end{equation}
Next, we apply the singular value decomposition to the $n \times n-1$ matrix $\begin{pmatrix}
\frac{1}{\sqrt{n}} (M_{n-1} - z) \\
Z
\end{pmatrix}$, generating an orthonormal basis of $n$ right singular vectors $u_1,\dots,u_n$ in $\C^n$, and an orthonormal basis of $n-1$ left singular vectors in $\C^{n-1}$, associated to singular values $\sigma_1,\dots,\sigma_n$ (with $\sigma_n=0$).  Then $W'_{n,z}$ is conjugate to the direct sum
$$ W'_{n,z} \equiv \bigoplus_{j=1}^{n-1} \begin{pmatrix} 0 & \sigma_j \\ \sigma_j & 0 \end{pmatrix} \oplus \begin{pmatrix} 0 \end{pmatrix}$$
and thus
$$ (W'_{n,z}-E-\sqrt{-1}\eta)^{-1} \equiv \bigoplus_{j=1}^{n-1} \frac{1}{\sigma_j^2 - (E+\sqrt{-1}\eta)^2} \begin{pmatrix} E+\sqrt{-1}\eta & \sigma_j \\ \sigma_j & E+\sqrt{-1}\eta \end{pmatrix} \oplus \begin{pmatrix} \frac{1}{E+\sqrt{-1}\eta} \end{pmatrix}$$
and thus
\begin{align*}
 \Im Y_n &= \sum_{j=1}^{n-1} \Im \frac{E+\sqrt{-1}\eta}{\sigma_j^2 - (E+\sqrt{-1}\eta)^2} |\tilde X^* u_j|^2  \\
 &= \frac{1}{2} \sum_{j=1}^{n-1} \sum_{\epsilon = \pm 1} \frac{1}{\epsilon \sigma_j - (E+\sqrt{-1}\eta)} |\tilde X^* u_j|^2 \\
 &= \frac{\eta}{2} \sum_{j=1}^{n-1} \sum_{\epsilon = \pm 1} \frac{1}{|E-\epsilon \sigma_j|^2 + \eta^2} |\tilde X^* u_j|^2
\end{align*}
where
$$ \tilde X := \begin{pmatrix}
X' \\
\frac{1}{\sqrt{n}} (\xi_{nn} - z)
\end{pmatrix}
$$
is the top half of $X$.

By \eqref{contra} and the Cauchy interlacing law, we may find an interval $[j_-,j_+]$ of length $j_+-j_- \gg n^{1+c} \eta$ such that $|\sigma_j-E| \leq \eta$ for all $j_- \leq j \leq j_+$.  We conclude that
$$ \sum_{j_- \leq j \leq j_+} |\tilde X^* u_j|^2 \ll n^{-c} \eta.$$

At this point we will follow \cite{ESY3} and invoke a concentration estimate for quadratic forms essentially due to Hanson and Wright \cite{hanson}, \cite{wright}.

\begin{proposition}[Concentration]\label{concentration}  Let $\xi_1,\dots,\xi_n$ be iid complex random variables with mean zero, variance one, and bounded in magnitude by $K$ for some $K \geq 1$.  Let $X \in \C^n$ be a random vector of the form $Y+Z$, where
$$ Y := \frac{1}{n^{1/2}} \begin{pmatrix} \xi_1 \\ \vdots \\ \xi_n \end{pmatrix}$$
and $Z$ is a random vector independent of $Y$.
Let $A = (a_{ij})_{1 \leq i,j \leq n}$ be a random complex matrix that is also independent of $Y$.  Then with overwhelming probability one has
$$ X^* A X = \frac{1}{n} \tr A + Z^* A Z + O\left( K^2 \log^2 n (\frac{1}{n} \| A \|_F + \frac{1}{\sqrt{n}} \|AZ\| + \frac{1}{\sqrt{n}} \|A^* Z\|) \right)$$
where $\|A\|_F := (\sum_{1 \leq i,j \leq n} |a_{ij}|^2)^{1/2}$ is the Frobenius norm of $A$.
\end{proposition}

We remark that for our applications, one could also use Talagrand's concentration inequality \cite{Tal} as a substitute for this concentration inequality, at the cost of a slight degradation in the bounds; see e.g. \cite{TVlocal1}.

\begin{proof}  By conditioning we may assume that $Z, A$ are deterministic (the failure probability in our estimates will be uniform in the choice of $Z, A$).  Let $\tilde \xi_i := \xi_i/K$.  From \cite[Proposition 4.5]{ESY3} we have
$$ \sum_{1 \leq i,j \leq n} a_{ij} \tilde \xi_i \overline{\tilde \xi_j} =
\sum_{1 \leq i,j \leq n} a_{ij} \E \tilde \xi_i \overline{\tilde \xi_j} + O( \|A\|_F \log^2 n )$$
with overwhelming probability.  Multiplying by $K^2/n$ and noting that $\E \xi_i \overline{\xi_j} = 1_{i=j}$, we conclude that
$$ Y^* A Y = \frac{1}{n} \tr A  + O\left( \frac{K^2 \log^2 n}{n} \| A \|_F  \right)$$
with overwhelming probability.  Meanwhile, from the Chernoff inequality we see that
$$ Y^* A Z = O\left( \frac{K \log^2 n}{\sqrt{n}} \|AZ\| \right)$$
and similarly
$$ Z^* A Y = O\left( \frac{K \log^2 n}{\sqrt{n}} \|A^* Z\| \right)$$
with overwhelming probability.  The claim follows.
\end{proof}

Applying Proposition \ref{concentration} (with $A$ equal to the projection matrix $A := \sum_{j_- \leq j \leq j_+} u_j u_j^*$), one has
$$ \sum_{j_- \leq j \leq j_+} |\tilde X^* u_j|^2 = \frac{j_+ - j_- + 1}{n} + \| \frac{z}{\sqrt{n}} \pi(e_n) \|^2 + O( n^{-1+o(1)} (j_+-j_-+1)^{1/2} ) + O( n^{-1/2+o(1)} \|\frac{z}{\sqrt{n}} \pi(e_n) \| )$$
with overwhelming probability.  By the arithmetic mean-geometric mean inequality one has $\| \frac{z}{\sqrt{n}} \pi(e_n) \|^2 + O( n^{-1/2+o(1)} \|\frac{z}{\sqrt{n}} \pi(e_n)\| ) \geq - n^{-1+o(1)}$, and we conclude that
$$ \sum_{j_- \leq j \leq j_+} |\tilde X^* u_j|^2 \gg n^c \eta$$
with overwhelming probability (conditioning on $M_{n-1},Z$).  Undoing the conditioning, we thus obtain a contradiction with overwhelming probability, and Proposition \ref{ni} follows.

\subsection{Resolvent bounds}\label{Resbound}

We now prove Proposition \ref{Resolv}, by using a more complicated variant of the arguments above.  We first take advantage of the fact that the spectral parameter $\sqrt{-1}\eta$ is on the imaginary axis to make some minor simplifications.  Namely, we have
\begin{align*}
R(\sqrt{-1}\eta) &= (W_{n,z}-\sqrt{-1}\eta)^{-1}  \\
&= W_{n,z} (W_{n,z}^2 + \eta^2)^{-1} + \sqrt{-1} \eta (W_{n,z}^2 + \eta^2)^{-1}.
\end{align*}
Note from \eqref{wnz} that $W_{n,z}^2+\eta^2$ is block-diagonal, and thus $W_{n,z} (W_{n,z}^2 + \eta^2)^{-1}$ vanishes on the diagonal.  We conclude that $R(\sqrt{-1} \eta)_{jj}$ and $s(\sqrt{-1}\eta)$ are purely imaginary (with non-negative imaginary part) for $1 \leq j \leq n$, with
\begin{equation}\label{emo}
\Im s(\sqrt{-1}\eta) = \frac{\eta}{2n} \tr(W_{n,z}^2 + \eta^2)^{-1} = \frac{\eta}{n} \tr ((M_n-z)^* (M_n-z) + \eta^2)^{-1}.
\end{equation}

Now we observe that it suffices to verify the claim for $\eta \geq n^{-1+c}$ for each fixed $c$.  To see this, observe that
$$ \Im R(\sqrt{-1}\eta)_{jj} = \eta \sum_{k=1}^{2n} \frac{|u_{k,j}|^2}{\lambda_i(W_{n,z})^2 + \eta^2}$$
for any $1 \leq j \leq 2n$, where $u_1,\dots,u_{2n}$ are an orthonormal basis of eigenvectors for $W_{n,z}$, and $u_{k,j}$ is the $j^{\th}$ coefficient of $u_k$.  Thus, if we can obtain Proposition \ref{Resolv} for $\eta \geq n^{-1+c}$, we conclude with overwhelming probability that
\begin{equation}\label{defense}
 \eta \sum_{k=1}^{2n} \frac{|u_{k,j}|^2}{\lambda_k(W_{n,z})^2 + \eta^2} \ll n^{o(1)}
\end{equation}
for all $\eta \geq n^{-1+c}$, and hence that
$$ \sum_{1 \leq k \leq 2n: \lambda_k(W_{n,z}) \leq \eta} |u_{k,j}|^2 \ll n^{o(1)} \eta$$
for all $\eta \geq n^{-1+c}$.  This implies that
$$ \sum_{1 \leq k \leq 2n: \lambda_k(W_{n,z}) \leq \eta} |u_{k,j}|^2 \ll n^{o(1)} ( \eta + n^{-1+c} )$$
for all $\eta>0$.    By dyadic summation (using the crude upper bound $\lambda_k(W_{n,z}) = O(n^{O(1)})$), this implies that
$$ \sum_{k=1}^{2n} \frac{|u_{k,j}|^2}{(\lambda_k(W_{n,z})^2 + \eta^2)^{1/2}} \ll n^{c+o(1)} (1+\frac{1}{n\eta})$$
for all $\eta > 0$.  Similarly with $u_{k,j}$ replaced by $u_{k,i}$.  By Cauchy-Schwarz, we conclude that
$$
|\sum_{k=1}^{2n} \frac{u_{k,j} \overline{u_{k,i}}}{\lambda_k(W_{n,z}) - \sqrt{-1} \eta}| \ll n^{c+o(1)} (1+\frac{1}{n\eta})$$
for any $\eta > 0$.  The left-hand side is $R(\sqrt{-1}\eta)_{ij}$.  The claim then follows by using a diagonalisation argument.

A similar argument reveals that we may assume without loss of generality that $\eta$ is an integer power of two.  Note that the above argument shows that one only needs to verify the diagonal case $i=j$; by symmetry and the union bound we may take $i=j=2n$.  The claim is trivially verified for $\eta \geq n^{10}$ (say), so we may assume that $\eta$ lies between $n^{-1+c}$ and $n^{10}$; by the union bound, we may now consider $\eta$ as fixed.  By diagonalisation (and the imaginary nature of the resolvent), it will now suffice to show that
\begin{equation}\label{joe}
 \Im R(\sqrt{-1}\eta)_{2n,2n} \ll n^{c+o(1)}
 \end{equation}
with overwhelming probability.

From \eqref{resolve-eq} (and the fact that $R(\sqrt{-1}\eta)_{2n,2n}$ is imaginary) we have
\begin{equation}\label{roast}
 \Im R(\sqrt{-1}\eta)_{2n,2n} = \frac{1}{\eta + \Im Y_n}
\end{equation}
where
$$ Y_n := X^* (W'_{n,z} - \sqrt{-1}\eta)^{-1} X.$$
From the block-diagonal nature of $W'_{n,z}$ as before we see that $Y_n$ is purely imaginary, with non-negative imaginary part; indeed, we have
\begin{equation}\label{imoen-2}
\Im Y_n = \eta \tilde X^* (A A^* + \eta^2)^{-1} \tilde X
\end{equation}
where $A$ is the $n \times n-1$ matrix
$$ A := \begin{pmatrix} M_{n-1} - z \\ Y \end{pmatrix}.$$
Thus we have the crude bound
\begin{equation}\label{crudito}
 \Im R(\sqrt{-1}\eta)_{2n,2n} \leq \frac{1}{\eta}
\end{equation}
which already takes care of the case when $\eta$ is large (e.g. $\eta \geq n^{-c}$).

On the other hand, we see from Proposition \ref{concentration} that with overwhelming probability one has
\begin{align*}
\tilde X^* (A A^* + \eta^2)^{-1} \tilde X &= \frac{1}{n} \tr (A A^* + \eta^2)^{-1} + \frac{|z|^2}{n} e_n^* (A A^* + \eta^2)^{-1} e_n \\
&\quad + O( n^{-1+o(1)} \|(A A^* + \eta^2)^{-1}\|_F ) + O( n^{-1+o(1)} |z| \| (A A^* + \eta^2)^{-1} e_n \| ).
\end{align*}
From the spectral theorem one has
$$ \| (A A^* + \eta^2)^{-1} e_n \| \leq (e_n^* (A A^* + \eta^2)^{-1} e_n)^{1/2} \eta^{-1} $$
and thus by Young's inequality (or the arithmetic mean-geometric mean inequality)
$$ n^{-1+o(1)} |z| \| (A A^* + \eta^2)^{-1} e_n \| = o( \frac{|z|^2}{n} e_n^* (A A^* + \eta^2)^{-1} e_n) + O( n^{-1+o(1)} \eta^{-2} ).$$
Also, we may expand
$$ \|(A A^* + \eta^2)^{-1}\|_F = (\sum_{j=1}^n \frac{1}{(\sigma_j(A)^2 + \eta^2)^2})^{1/2} $$
where $\sigma_1(A),\dots,\sigma_n(A)$ are the $n$ singular values of $A$ (thus one of these singular values is automatically zero).  From Proposition \ref{ni} and the Cauchy interlacing law, we see with overwhelming probability that for any interval $[-r,r]$, the number of singular values of $A$ in this interval is $O( n^{o(1)} (1+nr) )$.  From dyadic summation we then see that
\begin{equation}\label{aaf}
\|(A A^* + \eta^2)^{-1}\|_F \ll n^{o(1)} (n \eta)^{1/2} / \eta^2.
\end{equation}
Similarly, one has
$$ \tr (A A^* + \eta^2)^{-1} = \sum_{j=1}^n \frac{1}{\sigma_j(A)^2 + \eta^2}$$
and thus by interlacing
$$ \tr (A A^* + \eta^2)^{-1} = \sum_{j=1}^n \frac{1}{\sigma_j(M_n-z)^2 + \eta^2} + O( \frac{1}{\eta^2} ).$$
But from \eqref{emo} we have
$$ \sum_{j=1}^n \frac{1}{\sigma_j(M_n-z)^2 + \eta^2} = \frac{n}{\eta} s(\sqrt{-1} \eta)$$
and thus
\begin{equation}\label{bbb}
 \frac{\eta}{n} \tr (A A^* + \eta^2)^{-1} = s(\sqrt{-1} \eta) + O( \frac{1}{n\eta} ).
\end{equation}

Putting all this together with \eqref{imoen-2}, we see that with overwhelming probability one has
$$ \Im Y_n = \Im s(\sqrt{-1} \eta) + (1+o(1)) \frac{|z|^2}{n} \eta e_n^* (A A^* + \eta^2)^{-1} e_n
+ O( \frac{n^{o(1)}}{n\eta} ) + O( \frac{n^{o(1)}}{\sqrt{n\eta}} ),
$$
which, in view of the lower bound $\eta \geq n^{-1+c}$, simplifies to
\begin{equation}\label{roast-2}
 \Im Y_n = \Im s(\sqrt{-1} \eta) + (1+o(1)) \frac{|z|^2}{n} \eta e_n^* (A A^* + \eta^2)^{-1} e_n + o(1).
\end{equation}
Now we evaluate the expression $e_n^* (A A^* + \eta^2)^{-1} e_n$.  Observe that
$$
AA^* + \eta^2 = \begin{pmatrix}
(M_{n-1}-z) (M_{n-1}-z)^* + \eta^2 & (M_{n-1}-z)Y^* \\
Y (M_{n-1}-z)^* & YY^* + \eta^2.
\end{pmatrix}.$$
By Schur's complement, we thus have
$$
e_n^* (A A^* + \eta^2)^{-1} e_n = \frac{1}{YY^*+\eta^2 - Y (M_{n-1}-z)^* ((M_{n-1}-z) (M_{n-1}-z)^* + \eta^2)^{-1} (M_{n-1} -z) Y^*}.$$
One can simplify this using the identity
$$ B^* (BB^* + \eta^2)^{-1} B = 1 - \eta^2 (B^* B + \eta^2)^{-1},$$
valid for any matrix $B$ (which can be seen either from the singular value decomposition, or by multiplying both sides of the identity by $(B^* B + \eta^2)$) to conclude that
$$ \eta e_n^* (A A^* + \eta^2)^{-1} e_n  = \frac{1}{\eta + \eta Y ((M_{n-1}-z)^* (M_{n-1}-z) + \eta^2)^{-1} Y^*}.$$
Applying Lemma \ref{concentration}, we see with overwhelming probability that
\begin{align*}
\eta Y ((M_{n-1}-z)^* (M_{n-1}-z) + \eta^2)^{-1} Y^* &= \frac{\eta}{n} \tr((M_{n-1}-z)^* (M_{n-1}-z) + \eta^2)^{-1} \\
&\quad +
O( n^{-1+o(1)} \eta \|(M_{n-1}-z)^* (M_{n-1}-z) + \eta^2\|_F ).
\end{align*}
By mimicking the proof of \eqref{aaf}, one has
$$ \|(M_{n-1}-z)^* (M_{n-1}-z) + \eta^2\|_F \ll n^{o(1)} (n \eta)^{1/2} / \eta^2$$
with overwhelming probability.  Similarly, by mimicking the proof of \eqref{bbb} one has
$$
 \frac{\eta}{n} \tr((M_{n-1}-z)^* (M_{n-1}-z) + \eta^2)^{-1} = \Im s(\sqrt{-1} \eta) + O( \frac{1}{n\eta} ).
$$
Putting these bounds together, we conclude that
$$ \eta e_n^* (A A^* + \eta^2)^{-1} e_n  = \frac{1}{\eta + \Im s(\sqrt{-1}\eta) + o(1)}$$
with overwhelming probability; inserting this back into \eqref{roast-2} and \eqref{roast} we conclude that
\begin{equation}\label{overwhelm}
 \Im R(\sqrt{-1}\eta)_{2n,2n} = \frac{1}{\eta + \Im s(\sqrt{-1}\eta) + (1+o(1)) \frac{|z|^2/n}{\eta + \Im s(\sqrt{-1}\eta) + o(1)} + o(1)}
\end{equation}
with overwhelming probability.

Suppose now that $|z|^2/n \geq 1/2$.  Then we have
$$ |y + \frac{|z|^2/n}{y}| \gg 1$$
for any $y$; this implies that the denominator in \eqref{overwhelm} has magnitude $\gg 1$, which gives \eqref{joe}.  Thus we may assume that $|z|^2/n < 1/2$.

The bound \eqref{overwhelm} similarly with the index $2n$ replaced by any other index.  Averaging over these indices, we obtain the \emph{self-consistent equation}
\begin{equation}\label{self}
\Im s(\sqrt{-1}\eta) = \frac{1}{2n} \sum_{i=1}^{2n}
 \frac{1}{\eta + \Im s(\sqrt{-1}\eta) + (1+o(1)) \frac{|z|^2/n}{\eta + \Im s(\sqrt{-1}\eta) + o(1)} + o(1)}
\end{equation}
with overwhelming probability.  If we write $x := \eta + \Im s(\sqrt{-1}\eta)$, we thus have
$$ x = \frac{1}{2n} \sum_{i=1}^{2n} \frac{1}{x + (1+o(1)) \frac{|z|^2/n}{x+o(1)} + o(1)} + \eta$$
with overwhelming probability.  Note that either $x=o(1)$ or $x+o(1) = (1+o(1)) x$.  In the latter case, we can simplify the above equation as
$$  x = \frac{1}{2n} \sum_{i=1}^{2n} \frac{1+o(1)}{x + \frac{|z|^2/n}{x}} + \eta$$
and thus
$$  x = \frac{(1+o(1)) x}{x^2 + |z|^2/n} + \eta.$$
In particular, this forces $x^2 + |z|^2/n \geq 1+o(1)$. Since we have assumed that $|z|^2/n \leq  1/2$, we conclude that $x \geq 1/2$ (say).  We conclude that for each $n^{-1+c} \leq \eta \leq n^{10}$, we have
$$ \Im s(\sqrt{-1}\eta) + \eta = o(1)$$
or
$$ \Im s(\sqrt{-1}\eta) + \eta \geq 1/2$$
with overwhelming probability.  Rounding $\eta$ to the nearest multiple of (say) $n^{-100}$ and using the union bound (and crude perturbation theory estimates), we conclude with overwhelming probability that this dichotomy in fact holds for \emph{all} $n^{-1+c} \leq \eta \leq n^{10}$.  On the other hand, for $\eta=n^{10}$, one is clearly in the second case of the dichotomy rather than the first.  By continuity, we conclude that the second case of this dichotomy in fact holds for all $n^{-1+c} \leq \eta \leq n^{10}$; in particular, we have with overwhelming probability that
$$ \Im s(\sqrt{-1}\eta) \gg 1$$
when $n^{-1+c} \leq \eta \leq n^{-c}$.  Inserting this bound into \eqref{overwhelm}, we conclude with overwhelming probability that
$$
\Im R(\sqrt{-1}\eta)_{2n,2n} \ll 1$$
when $n^{-1+c} \leq \eta \leq n^{-c}$, which gives Proposition \ref{Resolv} in this case.  Finally, the case $\eta > n^{-c}$ can be handled by \eqref{crudito}.

\begin{remark}  A refinement of the above analysis can be used to give more precise control on the Stieltjes transform of $W_{n,z}$, as well as the counting function $N_I$.  See \cite{bai} for more details.
\end{remark}

\section{Asymptotics for the real gaussian ensemble}\label{real-app}

The purpose of this appendix is to establish Lemma \ref{corf}.  Our arguments here will rely heavily on those in \cite{borodin}.

By reflection we may restrict attention to the case when $z_1,\ldots,z_l$ lie in the upper half-plane $\C_+$.  Our starting point is the explicit formula
$$ \rho^{(k,l)}_n(x_1,\ldots,x_k,z_1,\ldots,z_l) = \operatorname{Pf} \begin{pmatrix} \tilde K_n(x_i,x_{i'}) & \tilde K_n(x_i,z_{j'}) \\
\tilde K_n(z_j, x_{i'}) & \tilde K_n(z_j, z_{j'}) \end{pmatrix}_{1 \leq i,i' \leq k; 1 \leq j,j' \leq l} $$
for the correlation functions, where $\tilde K_n: (\R \cup \C_+) \times (\R \cup \C_+) \to M_2(\C)$ is a certain explicit $2 \times 2$ matrix kernel obeying the anti-symmetry law
\begin{equation}\label{kt}
 \tilde K(\zeta, \zeta') = -\tilde K(\zeta',\zeta)^T,
\end{equation}
making the expression inside the Pfaffian $\operatorname{Pf}$ an anti-symmetric $2(k+l) \times 2(k+l)$ matrix; see \cite[Theorem 8]{borodin}.  In view of this formula, we see that Lemma \ref{corf} will follow if we can establish the uniform bound
$$ \tilde K_n(\zeta,\zeta') = O(1)$$
for all $\zeta,\zeta' \in \R \cup \C_+$.

To do this, we will need the explicit description of the kernel $\tilde K_n$.  Following \cite{borodin}, we will need the partial cosine and exponential functions
\begin{align*}
c_{n/2}(\gamma) &:= \sum_{m=0}^{n/2-1}\frac{\gamma^{2m}}{(2m)!} \\
e_{n/2}(\gamma) &:= \sum_{m=0}^{n-2} \frac{\gamma^m}{m!}
\end{align*}
as well as the function
$$
r_{n/2}(z,x) := \frac{e^{-z^2/2}}{\sqrt{2\pi}} \sqrt{\operatorname{erfc}(\sqrt{2}\Im z)} \frac{2^{(n-3)/2}}{(n-2)!} \sgn(x) z^{n-1} \gamma(\frac{n-1}{2},\frac{x^2}{2})
$$
where $\operatorname{erfc} := 1-\operatorname{erf}$ is the complementary error function and
$$
\gamma(t,x) = \int_0^x y^{t-1} e^{-y}\ dy
$$
is the incomplete gamma function.  In \cite[Theorem 8]{borodin}, the formula
$$ \tilde K_n(\gamma,\gamma') := \begin{pmatrix} \widetilde{DS}_n(\gamma,\gamma') & \widetilde {S}(\gamma,\gamma') \\ -\widetilde{S}(\gamma',\gamma) & \widetilde{ISM}_n(\gamma,\gamma') + {\mathcal E}(\gamma,\gamma') \end{pmatrix}
$$
is given for the kernel $\tilde K_n$,
where ${\mathcal E}(\gamma,\gamma')$ is equal to $\frac{1}{2} \sgn(\gamma-\gamma')$ when $\gamma,\gamma'$ are real, and equal to $0$ otherwise, and
the scalar quantities $\widetilde{DS}_n(\gamma,\gamma')$, $\widetilde {S}(\gamma,\gamma')$, $\widetilde{ISM}_n(\gamma,\gamma')$, are defined by the following formulae, depending on whether $\gamma,\gamma'$ are real or complex:

\begin{enumerate}
\item (Real-real case) If $x,x' \in \R$, then
\begin{align*}
\widetilde{S}_n(x,x') &:= \frac{e^{-(x-x')^2/2}}{\sqrt{2\pi}} e^{-xx'} e_{n/2}(xx') + r_{n/2}(x,x') \\
\widetilde{DS}_n(x,x') &:= \frac{e^{-(x-x')^2/2}}{\sqrt{2\pi}} (x'-x) e^{-xx'} e_{n/2}(xx') \\
\widetilde{IS}_n(x,x') &:= \frac{e^{-x^2/2}}{2\sqrt{\pi}} \sgn(x') \int_0^{(x')^2/2} \frac{e^{-t}}{\sqrt{t}} c_{n/2}(x\sqrt{2t})\ dt -
\frac{e^{-(x')^2/2}}{2\sqrt{\pi}} \sgn(x) \int_0^{x^2/2} \frac{e^{-t}}{\sqrt{t}} c_{n/2}(x'\sqrt{2t})\ dt.
\end{align*}
\item (Complex-complex case) If $z,z' \in \C_+$, then
\begin{align*}
\widetilde{S}_n(z,z') &:= \frac{i e^{-\frac{1}{2}(z-\overline{z'})^2}}{\sqrt{2\pi}} (\overline{z'} - z) \sqrt{\operatorname{erfc}(\sqrt{2} \Im(z)) \operatorname{erfc}(\sqrt{2} \Im(z'))} e^{-z\overline{z'}} e_{n/2}(z\overline{z'})\\
\widetilde{DS}_n(z,z') &:= \frac{e^{-\frac{1}{2}(z-z')^2}}{\sqrt{2\pi}} (z' - z) \sqrt{\operatorname{erfc}(\sqrt{2} \Im(z)) \operatorname{erfc}(\sqrt{2} \Im(z'))} e^{-z z'} e_{n/2}(z z') \\
\widetilde{IS}_n(z,z') &:= -\frac{e^{-\frac{1}{2}(\overline{z}-\overline{z'})^2}}{\sqrt{2\pi}} (\overline{z'} - \overline{z}) \sqrt{\operatorname{erfc}(\sqrt{2} \Im(z)) \operatorname{erfc}(\sqrt{2} \Im(z'))} e^{-\overline{z z'}} e_{n/2}(\overline{z z'}).
\end{align*}
\item (Real-complex case) If $x \in \R$ and $z \in \C_+$, then
\begin{align*}
\widetilde{S}_n(x,z) &:= \frac{i e^{-\frac{1}{2}(x-\overline{z})^2}}{\sqrt{2\pi}} \sqrt{\operatorname{erfc}(\sqrt{2} \Im(z))} e^{-x\overline{z}} e_{n/2}(x\overline{z}) \\
\widetilde{S}_n(z,x) &:= \frac{e^{-\frac{1}{2}(x-z)^2}}{\sqrt{2\pi}} \sqrt{\operatorname{erfc}(\sqrt{2} \Im(z))} e^{-xz} e_{n/2}(xz) + r_{n/2}(z,x) \\
\widetilde{DS}_n(x,z) &:= \frac{e^{-\frac{1}{2}(x-z)^2}}{\sqrt{2\pi}} (z-x) \sqrt{\operatorname{erfc}(\sqrt{2} \Im(z))} e^{-xz} e_{n/2}(xz) \\
\widetilde{IS}_n(x,z) &:= -\frac{ie^{-\frac{1}{2}(x-\overline{z})^2}}{\sqrt{2\pi}} \sqrt{\operatorname{erfc}(\sqrt{2} \Im(z))} e^{-x\overline{z}} e_{n/2}(x\overline{z}) - ir_{n/2}(\overline{z},x).
\end{align*}
\end{enumerate}

As ${\mathcal E}(\gamma,\gamma')$ is clearly bounded, it thus suffices (in view of \eqref{kt}) to show that all the expressions $\widetilde{S}_n(x,x')$, $\widetilde{DS}_n(x,x')$, $\widetilde{IS}_n(x,x')$, $\widetilde{S}_n(z,z')$, $\widetilde{DS}_n(z,z')$, $\widetilde{IS}_n(z,z')$, $\widetilde{S}_n(x,z)$, $\widetilde{S}_n(z,x)$, $\widetilde{DS}_n(x,z)$, $\widetilde{IS}_n(x,z)$ are all $O(1)$ for $x,x' \in \R$ and $z,z' \in \C_+$.  This will be a variant of the estimates in \cite[Section 9]{borodin}, which were concerned with the asymptotic values of these expressions as $n \to \infty$ rather than uniform bounds.

We first dispose of the $r_{n/2}$ terms.  In the proof of \cite[Corollary 9]{borodin}, the estimate
$$ |r_{n/2}(z,x)| \leq e^{-\frac{1}{2} \Re(z^2)} \sqrt{\operatorname{erfc}(\sqrt{2}\Im(z))} \frac{|z|^{n-1}}{2^{n/2} (n/2 - 1)!}$$
is established for any $x \in \R$ and $z \in \C^+$.  Using the standard bound
\begin{equation}\label{erfc-bound}
\operatorname{erfc}(x) = O( \frac{e^{-x^2}}{1+x} )
\end{equation}
for any $x \geq 0$, we thus have
$$ |r_{n/2}(z,x)| \ll e^{-|z|^2/2} \frac{|z|^{n-1}}{2^{(n-1)/2} (n/2 - 1)!}.$$
But $\frac{|z|^{n-1}}{2^{(n-1)/2} (n/2 - 1)!}$ is one of the Taylor coefficients of $e^{|z|^2/2}$, and so
\begin{equation}\label{rhin}
r_{n/2}(z,x) = O(1).
\end{equation}
Thus we may ignore all terms involving $r_{n/2}$.

Now we handle the real-real case. Recall from the triangle inequality and Taylor expansion that
\begin{equation}\label{enz}
 |e_{n/2}(z)| \leq e_{n/2}(|z|) \leq \exp(|z|)
\end{equation}
for any complex number $z$.    Thus, for instance, we have
$$
|\widetilde{S}_n(x,x')| \ll \exp( - (x-x')^2/2 - xx' + |xx'| ) + 1 \ll 1$$
since the expression inside the exponential is either $-(x-x')^2/2$ or $-(x+x')^2/2$.

If one applies the same method to bound $\widetilde{DS}_n(x,x')$, one obtains

Similarly one has
$$
|\widetilde{DS}_n(x,x')| \ll |x-x'| \exp( - (x-x')^2/2 - xx' + |xx'| ).$$
This bound is $O(1)$ when $xx'$ is positive, but can grow linearly when $xx'$ is negative.  To deal with this issue, we need an alternate bound to \eqref{enz} that saves an additional polynomial factor in some cases:

\begin{lemma}[Alternate bound]\label{enz-alt}  For any complex number $z$, one has
$$ |e_{n/2}(z)| \ll \frac{|z|^{1/2}}{\left||z|-z\right|} \exp(|z|),$$
with the convention that the right-hand side is infinite when $z$ is a non-negative real.
\end{lemma}

\begin{proof}  The claim is trivial for $|z| \leq 1$, so we may assume that $|z| > 1$.  Observe that
\begin{equation}\label{zam}
 (|z|-z) e_{n/2}(z) = \sum_{m=0}^{n/2} \frac{z^m}{m!} (|z| - m) - \frac{z^{n/2+1}}{(n/2)!}.
\end{equation}
An application of Stirling's formula reveals that
$$ \frac{z^m}{m!} = O( \frac{1}{|z|^{1/2}} \exp(|z|) )$$
for all $m$, so the second term on the right-hand side of \eqref{zam} is $O( |z|  \frac{1}{|z|^{1/2}} \exp(|z|)  )$.  It thus suffices to show that
$$ \sum_{m=0}^{n/2} \frac{z^m}{m!} (|z| - m) = O(|z|^{1/2} \exp(|z|) ).$$
By the triangle inequality, the left-hand side can be bounded by
$$ \sum_{m \leq |z|} \frac{|z|^m}{m!} (|z|-m) + \sum_{m > |z|} \frac{|z|^m}{m!} (m-|z|).$$
This expression telescopes to
$$ 2 \frac{|z|^{m+1}}{m!} $$
where $m := \lfloor |z| \rfloor$. By Stirling's formula, this expression is $O( |z|^{1/2}  \exp(|z|) )$ as required.
\end{proof}

Inserting this bound in the case when $xx'$ is negative, we conclude that
$$ |\widetilde{DS}_n(x,x')| \ll |x-x'| \frac{1}{(xx')^{1/2}} \exp( - (x-x')^2/2 - xx' + |xx'| )
= \frac{|x|+|x'|}{|x|^{1/2} |x'|^{1/2}} \exp( (|x|-|x'|)^2 / 2 ) $$
and one easily verifies that this expression is $O(1)$.

Finally, to control $\widetilde{IS}_n(x,x')$, it suffices by symmetry to show that
\begin{equation}\label{clash}
\int_0^{(x')^2/2} \frac{e^{-t}}{\sqrt{t}} c_{n/2}(x\sqrt{2t})\ dt = O( \exp( x^2 / 2 ) ).
\end{equation}
But by Taylor expansion we may bound $c_{n/2}(x\sqrt{2t})$ by $\cosh(x\sqrt{2t})$.  Since
$$ \int_0^{(x')^2/2} \frac{e^{-t}}{\sqrt{t}} \cosh(x\sqrt{2t}) = \frac{\sqrt{\pi}}{2} e^{(x')^2/2} (\operatorname{erf}(\frac{|x|+|x'|}{\sqrt{2}}) -
\operatorname{erf}(\frac{|x'|-|x|}{\sqrt{2}}) ),$$
we see from \eqref{erfc-bound} that the left--hand side of \eqref{clash} is
$$ \ll \exp( (x')^2/2 ) \exp( - \max(|x'| - |x|,0)^2 / 2 ) \leq \exp(x^2/2)$$
as required.

Next we turn to the complex-complex case.  From \eqref{erfc-bound} and \eqref{enz} we see that
$$
|\widetilde{S}_n(z,z')| \ll \exp( - \frac{1}{2} \Re( (z-\overline{z'})^2 ) ) |\overline{z'}-z| \frac{\exp(-\Im(z)^2-\Im(z')^2)}{(1 + \Im(z))^{1/2} (1+\Im(z'))^{1/2}} \exp( |z\overline{z'}| - \Re (z \overline{z'}) ).$$
After some rearrangement, the right-hand side here becomes
$$ \frac{|\overline{z'}-z|}{(1 + \Im(z))^{1/2} (1+\Im(z'))^{1/2}} \exp( -\frac{1}{2}( |z| - |z'| )^2 ).$$
If one uses Lemma \ref{enz-alt} instead of \eqref{enz}, one gains an additional factor of $\frac{|z|^{1/2} |z'|^{1/2}}{\left||z||z'| - z\overline{z'}\right|}$.  Thus, it suffices to show that
\begin{equation}\label{fli}
\frac{|\overline{z'}-z|}{(1 + \Im(z))^{1/2} (1+\Im(z'))^{1/2}} \min( 1, \frac{|z|^{1/2} |z'|^{1/2}}{\left||z||z'| - z\overline{z'}\right|} )
\exp( -\frac{1}{2}( |z| - |z'| )^2 ) \ll 1.
\end{equation}
By symmetry, we may assume that $0 < \Im(z) \leq \Im(z')$.  We may assume that $|z|$ and $|z'|$ are comparable and larger than $1$, since otherwise the claim easily follows from the $\exp( -\frac{1}{2}( |z| - |z'| )^2 )$ term.

Let $\theta$ denote the angle subtended by $z$ and $z'$.  Observe from the triangle inequality that
\begin{equation}\label{sash}
 |\overline{z'}-z| \ll ||z|-|z'|| + \Im(z) + |z| \theta
\end{equation}
and
$$ \left||z||z'| - z\overline{z'}\right| \gg |z|^2 \theta.$$
The first two terms on the right-hand side of \eqref{sash} give an acceptable contribution to \eqref{fli} (bounding the minimum crudely by $1$), so it suffices to show that
$$ \frac{|z| \theta}{(1 + \Im(z))^{1/2} (1+\Im(z'))^{1/2}} \min( 1, \frac{|z|}{|z|^2 \theta} ) \ll 1,$$
but this is clear after discarding the denominator and using the second term in the minimum.  This establishes the bound $|\widetilde{S}_n(z,z')| \ll 1$.  Similar arguments, which we leave to the reader, show that $|\widetilde{DS}_n(z,z')| \ll 1$ and $|\widetilde{IS}_n(z,z')| \ll 1$.

Finally, we turn to the real-complex case.  Using \eqref{enz} and \eqref{erfc-bound}, we can bound
$$
|\widetilde{S}_n(x,z)| \ll \exp(-\frac{1}{2} \Re( (x-\overline{z})^2 ) ) \frac{\exp( - \Im(z)^2 )}{1+\Im(z)^{1/2}} \exp( - x \overline{z} + |x| |z| ).$$
The right-hand side simplifies to $\exp( -(x-|z|)^2/2 ) / (1 + \Im(z)^{1/2})$, which is clearly $O(1)$.

A similar argument (using \eqref{rhin}) shows that $\widetilde{S}_n(x,z) = O(1)$ and $\widetilde{IS}_n(x,z) = O(1)$.  The bound $\widetilde{DS}_n(x,z) = O(1)$ can be established by the same arguments used to handle the complex-complex case; we leave the details to the reader.  This completes the proof of Lemma \ref{corf}.


\begin{thebibliography}{10}

\bibitem{ake}
G. Akemann, E. Kanzieper, {Integrable structure of Ginibre's ensemble of
real random matrices and a Pfaffian integration theorem}, \emph{J. Stat. Phys.} \textbf{129} (2007), 1159--1231.

\bibitem{alon}
N. Alon, J. Spencer, The probabilistic method, John Wiley \& Sons, Inc., Hoboken, NJ, 2008.

\bibitem{bai}
Z. D. Bai, Circular law, \emph{Ann. Probab.} \textbf{25} (1997),
494--529.

\bibitem{baiyin}
Z. D. Bai, Y. Q. Yin, \emph{Limiting behavior of the norm of products of random matrices and two problems of Geman-Hwang}, Probab. Theory Relat. Fields \textbf{73} (1986), 555--569.

\bibitem{chafai}
C. Bordernave, D. Chafai, \emph{Around the circular law}, Probability Surveys \textbf{9} (2012) 1--89.

\bibitem{borodin-0}
A. Borodin, C. D. Sinclair, Correlation Functions of Ensembles of Asymmetric Real Matrices, preprint.

\bibitem{borodin}
A. Borodin, C. D. Sinclair, The Ginibre ensemble of real random matrices and its scaling limits, Comm. Math. Phys. \textbf{291} (2009), no. 1, 177--224.

\bibitem{byy}
P. Bourgade, H.-T. Yau, J. Yin, Local Circular Law for Random Matrices, preprint.

\bibitem{byy-2}
P. Bourgade, H.-T. Yau, J. Yin, The local circular law II: the edge case, preprint.

\bibitem{brown}
L. G. Brown, Lidskii's theorem in the type II case. Geometric methods in operator algebras (Kyoto, 1983), 1--35,
Pitman Res. Notes Math. Ser., 123, Longman Sci. Tech., Harlow, 1986.

\bibitem{Chat} S. Chatterjee, A generalization of the Lindenberg principle,  \emph{Ann. Probab.} \textbf{34} (2006), no. 6, 2061--2076.

\bibitem{costin}
A. Costin, J.L. Lebowitz, Gaussian fluctuations in random matrices, \emph{Physical Review Letters} \textbf{75} (1995), 69--72.

\bibitem{eks}
A. Edelman, E. Kostlan, M. Shub, \emph{How many eigenvalues of a random matrix are real?} J. Amer. Math. Soc. \textbf{7} (1994), no. 1, 247--267.

\bibitem{edel} A. Edelman,  Probability that a Random Real Gaussian Matrix Has $k$ Real Eigenvalues,
 Related Distributions, and the Circular Law, {\it  Journal of Multivariate Analysis} 60, (1997), 203--232.

\bibitem{Erd} L. Erd\H os, Universality of Wigner random matrices: a Survey of Recent Results, {\it arXiv:1004.0861}.

\bibitem{ERSTVY}
L. Erd\H{o}s, J. Ramirez,
  B. Schlein, T. Tao, V. Vu, and  H.-T. Yau, Bulk universality for Wigner hermitian matrices with subexponential decay, {\it arxiv:0906.4400}, {\it To appear in Math. Research Letters}.

\bibitem{ESY1}
L. Erd\H{o}s, B. Schlein and  H.-T. Yau,
Semicircle law on short scales and delocalization of eigenvectors for Wigner random matrices. {\it Ann. Probab.} \textbf{37} (2009), 815-852 .

\bibitem{ESY2}  L. Erd\H{o}s, B. Schlein and H-T. Yau, Local semi-circle law and complete delocalization for Wigner random matrices,  \emph{Comm. Math. Phys.} \textbf{287} (2009), no. 2, 641--655.

\bibitem{ESY3}
L. Erd\H{o}s, B. Schlein and  H.-T. Yau, Wegner estimate and level repulsion for Wigner random matrices.
Int. Math. Res. Notices \textbf{2010} (2010), 436--479.

\bibitem{ESY}
L. Erd\H{o}s, B. Schlein and  H.-T. Yau, Universality of Random Matrices and Local Relaxation Flow, {\it arXiv:0907.5605}

\bibitem{ESYY}
L. Erd\H{o}s, B. Schlein, H.-T. Yau and J. Yin, The local relaxation flow approach to universality of the
local statistics for random matrices. {\it arXiv:0911.3687}

\bibitem{EYY}
L. Erd\H{o}s, H.-T.Yau, and J. Yin, Bulk universality for generalized Wigner matrices. {\it arXiv:1001.3453}

\bibitem{mays}
P. J. Forrester and A. Mays, A method to calculate correlation functions for $\beta=1$ random matrices of odd size, J. Stat. Phys. \textbf{134} (2009), no. 3, 443--462.

\bibitem{FN} P. J. Forrester and T. Nagao, Eigenvalue Statistics of the Real Ginibre Ensemble, \emph{Phys. Rev. Lett.} 99 (2007), 050603.

\bibitem{geman}
S. Geman, \emph{The spectral radius of large random matrices}, Ann. Probab. \textbf{14} (1986), 1318--1328.

\bibitem{gin}
J. Ginibre, \emph{Statistical ensembles of complex, quaternion, and real matrices}, Journal of Mathematical Physics \textbf{6} (1965), 440--449.

\bibitem{girko}
V. L. Girko, Circular law, \emph{Theory Probab. Appl.} (1984),
694--706.

\bibitem{Alice}
A. Guionnet, Grandes matrices al\'eatoires et th\'eor\`emes d'universalit\'e, \emph{S\'eminaire BOURBAKI}. Avril 2010. 62\`eme ann\'ee, 2009-2010, no 1019.

\bibitem{hanson}
D. L. Hanson, F. T. Wright, {A bound on tail probabilities for quadratic forms in independent random
variables}, \emph{Annals of Math. Stat.} \textbf{42} (1971), no.3, 1079-1083.

\bibitem{kanz}
E. Kanzieper, G. Akemann, Statistics of real eigenvalues in Ginibre's ensemble of random real matrices, \emph{Physical Review Letters}, \textbf{95} (2005), 230201.

\bibitem{knowles}
A. Knowles, J. Yin, Eigenvector Distribution of Wigner Matrices, {\it arXiv:1102.0057}

\bibitem{kostlan}
E. Kostlan, On the spectra of Gaussian matrices, \emph{Linear Algebra Appl.} \textbf{162/164} (1992), p. 385--388, Directions in matrix theory (Auburn, AL, 1990).

\bibitem{kv}
M. Krishnapur, B. Virag, The Ginibre ensemble and Gaussian analytic functions, {\it arXiv:1112.2457}

\bibitem{leh}
N. Lehmann, H.-J. Sommers, H.-J., Eigenvalue statistics of random real matrices, \emph{Phys. Rev. Lett.} \textbf{67} (1991), 941--944.

\bibitem{Meh} M.L. Mehta, Random Matrices and the Statistical Theory of Energy Levels, Academic Press, New York, NY, 1967.

\bibitem{nguyen}
H. Nguyen, V. Vu, Random matrices: law of the determinant, {\it preprint.}

\bibitem{nourdin}
I. Nourdin, G. Peccati, Universal Gaussian fluctuations of non-Hermitian matrix ensembles: from weak convergence to almost sure CLTs, \emph{ALEA Lat. Am. J. Probab. Math. Stat.} 7 (2010), 341--375.

\bibitem{rider-2}
B. Rider, A limit theorem at the edge of a non-Hermitian random matrix ensemble, \emph{J. Phys. A: Math. Gen.} \textbf{36} (2003), 3401--3409.

\bibitem{rider}
B. Rider, Deviations from the circular law, \emph{Probab. Theory Related Fields} \textbf{130} (2004), no. 3, 337--367.

\bibitem{rider-silverstein}
B. Rider, J. Silverstein, Gaussian fluctuations for non-Hermitian random matrix ensembles, \emph{Ann. Probab. }34 (2006), no. 6, 2118--2143.

\bibitem{rider-virag}
B. Rider, B. Vir\'ag, The noise in the circular law and the Gaussian free field, \emph{Int. Math. Res. Not.} 2007, no. 2, Art. ID rnm006, 33 pp.

\bibitem{RV}
M. Rudelson, R. Vershynin, The Littlewood-Offord Problem and invertibility of random matrices,
\emph{Advances in Mathematics} \textbf{218} (2008), 600--633.

\bibitem{RV1}
M. Rudelson, R. Vershynin, Non-asymptotic theory of random matrices: extreme singular values. \emph{Proceedings of the International Congress of Mathematicians}, Volume III, 1576--1602, Hindustan Book Agency, New Delhi, 2010.

\bibitem{schlein}
B. Schlein, Spectral Properties of Wigner Matrices, \emph{Proceedings of the Conference QMath 11}, Hradec Kralove, September 2010.

\bibitem{sinclair}
C. D. Sinclair, Correlation functions for $\beta=1$ ensembles of matrices of odd size, \emph{J. Stat. Phys.} \textbf{136} (2009), no. 1, 17--33.

\bibitem{sommers}
H.-J. Sommers, W. Wieczorek, General eigenvalue correlations for the real Ginibre ensemble. \emph{J. Phys. A} \textbf{41} (2008), no. 40, 405003.

\bibitem{sos}
A. Soshnikov, Gaussian fluctuations for Airy, Bessel and and other determinantal random
point fields, Journal of Statistical Physics \textbf{100} (2000), 491--522.

\bibitem{sos2}
A. Soshnikov, Gaussian limits for determinantal random point fields, Annals of Probability
\textbf{30} (2002), 171--181.

\bibitem{Tal} M. Talagrand, A new look at independence, {\it Ann. Probab.}  24
(1996), no. 1, 1--34.

\bibitem{TVbook} T. Tao, V. Vu, Additive combinatorics, Cambridge University Press, 2006.

\bibitem{TVannals} T. Tao and V. Vu, Inverse Littlewood-Offord theorems and the condition number of
 random discrete matrices, \emph{Annals of Math.} \textbf{169} (2009), 595--632

\bibitem{TV-cond}
T. Tao and V. Vu, The condition number of a randomly perturbed matrix, \emph{Proceedings of the thirty-ninth annual ACM symposium on Theory of computing (STOC)} 2007, 248--255.

\bibitem{TV-survey}
T. Tao and V. Vu, From the Littlewood-Offord problem to the circular law: universality of the spectral distribution of random matrices, \emph{Bull. Amer. Math. Soc.} \textbf{46} (2009), 377--396.

\bibitem{TVlocal2}  T. Tao and V. Vu,  Random matrices: universality of local eigenvalue statistics up to the edge, {\it  Comm. Math. Phys}. \textbf{298} (2010), no. 2, 549--572.

\bibitem{TVsmooth}
T. Tao and V. Vu,  Smooth analysis of the condition number and the least singular value, Mathematics of Computation, \textbf{79} (2010), 2333--2352.

\bibitem{TVlocal1}
T. Tao and V. Vu, Random matrices: Universality of the local eigenvalue statistics, {\it Acta Mathematica}
 206 (2011), 127--204.

\bibitem{TVlocal3}
T. Tao, V. Vu, Random covariance matrices: university of local statistics of eigenvalues, {\it to appear in Annals of Probability}.

\bibitem{TVmeh} T. Tao and V. Vu, The Wigner-Dyson-Mehta bulk universality conjecture for Wigner matrices, {\it  Electronic Journal of Probability},
vol 16 (2011), 2104-2121.

\bibitem{TV-vector} T. Tao and V. Vu, Random matrices: Universal properties of eigenvectors, {\tt arXiv:arXiv:1103.2801}.

\bibitem{TV-determinant} T. Tao and V. Vu, A central limit theorem for the determinant of a Wigner matrix, {\tt arXiv:1111.6300}.

\bibitem{TV-survey2} T. Tao and V. Vu, Random matrices: The Universality phenomenon for Wigner ensembles, {\tt arXiv:1202.0068}.

\bibitem{trotter}
H. Trotter, Eigenvalue distributions of large Hermitian matrices; Wigner's
semi-circle law and a theorem of Kac, Murdock, and Szeg\"o, \emph{Adv. in Math.} \textbf{54}(1984), 67--82.

\bibitem{VuNL}
V. Vu, Concentration of non-Lipschitz functions and applications, \emph{Random Structures and Algorithms}, \textbf{20} (2002), 262--316.


\bibitem{wright}
F. T. Wright, {A bound on tail probabilities for quadratic forms in independent random variables
whose distributions are not necessarily symmetric}, \emph{Ann. Probab.} \textbf{1} No. 6. (1973), 1068-1070.

\end{thebibliography}
\end{document}